\documentclass[11pt]{amsart}

\usepackage{amssymb, amscd, txfonts}
\usepackage{graphicx}

\numberwithin{equation}{section}

\newtheorem{theorem}{Theorem}[section]
\newtheorem{proposition}[theorem]{Proposition}
\newtheorem{lemma}[theorem]{Lemma}
\newtheorem{corollary}[theorem]{Corollary}

\theoremstyle{definition}

\newtheorem{example}[theorem]{Example}

\newtheorem{condition}[theorem]{Condition}

\theoremstyle{remark}
\newtheorem{remark}[theorem]{Remark}
\newtheorem{claim}[theorem]{Claim}

\newcommand{\Z}{\mathbb{Z}}
\newcommand{\Q}{\mathbb{Q}}

\newcommand{\C}{\mathbb{C}}
\newcommand{\F}{\mathbb{F}}

\newcommand{\proj}{{\mathbb P}}

\newcommand{\moduli}{\mathcal{M}_{r,a,\delta}}
\newcommand{\cover}{\widetilde{\mathcal{M}}_{r,a,\delta}}
\newcommand{\cove}{\widetilde{\mathcal{M}}}
\newcommand{\lift}{\widetilde{\mathcal{P}}}
\newcommand{\SL}{{\rm SL}}
\newcommand{\GL}{{\rm GL}}
\newcommand{\PGL}{{\rm PGL}}
\newcommand{\cohomology}{H^{2}(X, \mathbb{Z})}
\newcommand{\Oline}{\mathcal{O}_{{\mathbb P}^{1}}}
\newcommand{\Oplane}{\mathcal{O}_{{\mathbb P}^{2}}}
\newcommand{\Ospace}{\mathcal{O}_{{\mathbb P}^{3}}}

\newcommand{\OHirze}{\mathcal{O}_{\mathbb{F}_{n}}}
\newcommand{\OQ}{\mathcal{O}_{Q}}
\newcommand{\sheaf}{\mathcal{O}}
\newcommand{\Hom}{{\rm Hom}}
\newcommand{\Pic}{{\rm Pic}}
\newcommand{\Or}{{\rm O}}

\DeclareMathOperator{\aut}{Aut}

\begin{document}

\title[]{Rationality of the moduli spaces of 2-elementary $K3$ surfaces}
\author[]{Shouhei Ma}
\thanks{Supported by Grant-in-Aid for JSPS fellows [21-978] and Grant-in-Aid for Scientific Research (S), No 22224001.} 
\address{Graduate~School~of~Mathematical~Sciences, the~University~of~Tokyo, Tokyo 153-8914, Japan}
\address{Graduate~School~of~Mathematics, Nagoya~University, Nagoya 464-8604, Japan}
\email{ma@math.nagoya-u.ac.jp}
\maketitle 

\begin{abstract}
$K3$ surfaces with non-symplectic involution are classified by 
open sets of seventy-five arithmetic quotients of type IV. 
We prove that those moduli spaces are rational except two classical cases. 
\end{abstract}

\maketitle


\section{Introduction}

$K3$ surfaces with non-symplectic involution are basic objects in the study of $K3$ surfaces. 
They connect $K3$ surfaces with rational surfaces, Enriques surfaces, and low genus curves. 
In this article we address the rationality problem for the moduli spaces of $K3$ surfaces with non-symplectic involution. 

To be more precise, let $X$ be a complex $K3$ surface with an involution $\iota$. 
When $\iota$ acts nontrivially on $H^0(K_X)$, $\iota$ is called \textit{non-symplectic}, 
and the pair $(X, \iota)$ is called a \textit{2-elementary $K3$ surface}. 
By Nikulin \cite{Ni2}, the deformation type of $(X, \iota)$ is determined by its \textit{main invariant} 
$(r, a, \delta)$, a triplet of integers associated to the lattice of $\iota$-invariant cycles. 
He classified all main invariants of 2-elementary $K3$ surface, 
which turned out to be seventy-five in number (see Figure \ref{Nikulin table}). 
For each main invariant $(r, a, \delta)$, let $L_-$ be an even lattice of signature $(2, 20-r)$ 
whose discriminant form is 2-elementary of length $a$ and parity $\delta$. 
Yoshikawa \cite{Yo1}, \cite{Yo3} showed that 
the moduli space ${\moduli}$ of 2-elementary $K3$ surfaces of type $(r, a, \delta)$ is 
the complement of a Heegner divisor in the arithmetic quotient defined by 
the orthogonal group ${\rm O}(L_-)$ of $L_-$. 
In particular, ${\moduli}$ is irreducible of dimension $20-r$. 

Our main result is the following. 

\begin{theorem}\label{main}
The moduli space ${\moduli}$ of 2-elementary $K3$ surfaces of type $(r, a, \delta)$ is rational, 
possibly except when $(r, a, \delta)=(1, 1, 1)$ and $(2, 2, 0)$. 
\end{theorem}

The remaining spaces $\mathcal{M}_{1,1,1}$ and $\mathcal{M}_{2,2,0}$ are respectively birational to 
the moduli spaces of plane sextics and of bidegree $(4, 4)$ curves on ${\proj}^1\times{\proj}^1$. 
The rationality questions for them are rather classical, and remain open. 

A few ${\moduli}$ have been known to be rational: 
$\mathcal{M}_{10,10,0}$ is the moduli of Enriques surfaces and is rational by Kond\=o \cite{Ko}; 
he also proved that $\mathcal{M}_{10,2,0}$ is rational by identifying it with 
the moduli of certain trigonal curves \cite{Ko}; 
$\mathcal{M}_{5,5,1}$ is birational to the moduli of genus six curves (see \cite{SB2}, \cite{A-K}), 
and hence rational by Shepherd-Barron \cite{SB2}; 
Dolgachev and Kond\=o \cite{D-K} recently proved the rationality of $\mathcal{M}_{11,11,1}$, 
which is the moduli of Coble surfaces. 
Now Theorem \ref{main} shows that these results are not of sporadic nature. 
In other terms, the arithmetic quotients defined by ${\rm O}(L_-)$ 
for the 2-elementary primitive sublattices $L_-$ of the $K3$ lattice are mostly rational.

In \cite{Ma} we proved that all ${\moduli}$ are unirational. 
The approach there was to use some isogenies between ${\moduli}$ to reduce the problem to fewer main invariants.  
In fact, in \cite{Ma} we studied certain covers of ${\moduli}$ rather than ${\moduli}$ themselves. 
But the rationality problem is much more subtle, 
requiring individual treatment and delicate analysis. 
Thus in this article we study the spaces ${\moduli}$ one by one.

Although our approach is case-by-case, the schemes of proofs are common for most cases, 
which we shall try to explain here. 
We first find a \textit{birational} period map 
\begin{equation}\label{period map intro}
\mathcal{P}: U/G \dashrightarrow {\moduli}
\end{equation}
from the quotient space $U/G$ of an explicit parameter space $U$ by an algebraic group $G$. 
Then we prove that $U/G$ is rational as a problem in invariant theory (cf. \cite{Do}, \cite{P-V}). 
The space $U$ parametrizes (possibly singular) $-2K_Y$-curves $B$ on some smooth rational surfaces $Y$, 
and the map $\mathcal{P}$ is defined by 
taking the desingularizations of the double covers of $Y$ branched along $B$. 
Thus our approach for 2-elementary $K3$ surfaces is essentially from quotient surfaces and branch curves. 

The fact that $\mathcal{P}$ is birational means that 
our construction is ``canonical" for the generic member of ${\moduli}$. 
The verifications of this property mostly fall into the following two patterns: 
$(\alpha)$ when the curves $B$ are smooth, so that the desingularization process does not take place, 
it is just a consequence of the equivalence between (variety, involution) and (quotient, branch). 
This simple approach is taken when $r=a$ ($Y$ are del Pezzo) and when $(r, a, \delta)=(10, 8, 1)$. 
$(\beta)$ For many other $(r, a, \delta)$, the surfaces $Y$ are ${\proj}^2$ or Hirzebruch surfaces ${\F}_n$, 
the curves $B$ have simple singularities of prescribed type, and the group $G$ is ${\aut}(Y)$. 
For this type of period map we can calculate its degree in a systematic manner. 
A recipe of such calculation is presented in \S \ref{ssec: recipe}. 
By finding a period map of degree $1$ of this type, 
our rationality problem is reduced to invariant theory over ${\proj}^2$ or ${\F}_n$. 

It should be noted, however, that 
the existence of birational period map as above, especially of type $(\beta)$, 
is a priori not clear. 
Indeed, when $r+a=20$, $6\leq a\leq9$, $\delta=1$ and when $(r, a, \delta)=(15, 7, 1)$, 
we give up the above two approaches and study a \textit{non-birational} period map 
$U/G\dashrightarrow{\moduli}$ of type $(\beta)$. 
Using the recipe in \S \ref{ssec: recipe}, 
we show that it is a quotient map by a Weyl group, and then analyze that Weyl group action to find a canonical construction.  
The proof for these cases would be more advanced than the above cases $(\alpha)$, $(\beta)$. 
On the other hand, when $r+a\leq18$, some of our birational period maps seem related to 
Nakayama's minimal model process for log del Pezzo surfaces of index $2$ \cite{Na}, 
which may explain their birationality. 

Recall that for most $(r, a, \delta)$ 
the fixed locus $X^{\iota}$ of $(X, \iota)\in{\moduli}$ 
is the union of a genus $g=11-2^{-1}(r+a)$ curve $C^g$ and some other $(-2)$-curves (\cite{Ni2}). 
As an application of our birational period maps, we will determine the generic structure of the fixed curve map 
\begin{equation}
F : {\moduli}\to\mathcal{M}_g, \qquad (X, \iota)\mapsto C^g, 
\end{equation}
when $g\geq3$. 
The point is that $F$ is the composition of 
the inverse period map $\mathcal{P}^{-1}\colon{\moduli}\dashrightarrow U/G$ 
with the map $U/G\to\mathcal{M}_g$ that associates to a $-2K_Y$-curve its component of maximal genus.   
We will find that the latter map identifies $U/G$ with 
a natural fibration over a sublocus of $\mathcal{M}_g$ defined in terms of special line bundles and points. 
Thus, via $F$, ${\moduli}$ is in a nice relation to $\mathcal{M}_g$. 
This generalizes the descriptions of $\mathcal{M}_{10,2,0}$ and $\mathcal{M}_{5,5,1}$ referred above. 

Finally, we should explain that this article has been developed in two-steps. 
In an earlier version of this article, we could not prove the rationality of 
$\mathcal{M}_{10,10,1}$, $\mathcal{M}_{12,10,1}$, $\mathcal{M}_{13,9,1}$, $\mathcal{M}_{14,8,1}$, 
(and $\mathcal{M}_{11,11,1}$) by the above methods. 
After that version had circulated, Dolgachev and Kond\=o \cite{D-K} settled the case of $\mathcal{M}_{11,11,1}$ 
using a lattice-theoretic trick. 
Then the referee and Kond\=o independently suggested that 
a similar technique is also applicable for $\mathcal{M}_{12,10,1}$, which is produced in the present article. 
Afterwards, we developed the method of Dolgachev-Kond\=o further and proved the rationality of 
$\mathcal{M}_{10,10,1}$, $\mathcal{M}_{13,9,1}$, and $\mathcal{M}_{14,8,1}$. 
In view of this development, we shall treat those four cases separately in the Appendix.

The rest of the article is organized as follows. 
\S\ref{sec:invariant theory} and \S\ref{Sec:Hirze} are preliminaries on invariant theory. 
In \S\ref{sec: 2-ele K3}, after reviewing basic theory of 2-elementary $K3$ surfaces, 
we explain how to calculate the degrees of certain period maps. 
The proof of Theorem \ref{main} begins with \S\ref{sec:k=0}, 
where we treat the case $r=a$ using del Pezzo surfaces. 
The case $r>a$ is divided according to the genus $g$ of main fixed curves. 
This division policy comes from our observation on the relation of ${\moduli}$ to $\mathcal{M}_g$. 
In \S\ref{sec: g>6, k>0} we treat the case $g\geq7$ where trigonal curves of fixed Maroni invariant are central. 
Section $7\leq n\leq11$ is for the case $g=13-n$. 
In \S\ref{sec:g=1 (1)} we treat the case $g=1$, $11\leq r\leq14$, $\delta=1$, 
where del Pezzo surfaces appear again in more nontrivial way. 
We study in \S \ref{sec:g=1 (2)} the rest cases with $g=1$, $r\geq14$, 
and in \S\ref{sec:g=0} the case $g=0$, $r\geq15$. 
Finally, we treat the cases $r+a=22$, $12\leq r\leq14$ and $(r, a, \delta)=(10, 10, 1)$ in the Appendix.


\section{Rationality of quotient variety}\label{sec:invariant theory}

Let $U$ be an irreducible variety acted on by an algebraic group $G$. 
Throughout this article we shall denote by $U/G$ a \textit{rational quotient}, 
namely a variety whose function field is isomorphic to the invariant subfield ${\C}(U)^G$ of 
the function field ${\C}(U)$ of $U$. 
As ${\C}(U)^G$ is finitely generated over ${\C}$, a rational quotient always exists (cf. \cite{P-V}) 
and is unique up to birational equivalence. 
The inclusion ${\C}(U)^G \subset {\C}(U)$ induces a quotient map $\pi\colon U\dashrightarrow U/G$. 
Every $G$-invariant rational map $\varphi\colon U\to V$ is factorized as $\varphi=\psi\circ \pi$ 
by a unique rational map $\psi\colon U/G \dashrightarrow V$. 
We are interested in the problem when $U/G$ is rational. 
We recall some results and techniques following \cite{Do} and \cite{P-V}. 

\begin{theorem}[Miyata \cite{Miy}]\label{Miyata}
If $G$ is a connected solvable group and $U$ is a linear representation of $G$, 
the quotient $U/G$ is rational. 
\end{theorem}

\begin{theorem}[Katsylo, Bogomolov \cite{Ka1.5} \cite{B-K}]\label{Katsylo}
If $U$ is a linear representation of $G={\SL}_2\times ({\C}^{\times})^k$, 
the quotient $U/G$ is rational. 
\end{theorem}

Theorem \ref{Katsylo} is for the most part a consequence of Katsylo's theorem in \cite{Ka1.5}. 
The exception in \cite{Ka1.5} was settled by \cite{B-K}. 
See also \cite{P-V} Theorem 2.12. 

The $G$-action on $U$ is \textit{almost transitive} if $U$ contains an open orbit. 
The following is a special case of the slice method (\cite{Ka1}, \cite{Do}). 

\begin{proposition}[slice method]\label{slice}
Let $f:U\to V$ be a $G$-equivariant dominant morphism with almost transitive $G$-action on $V$. 
If $G_v\subset G$ is the stabilizer group of a general point $v \in V$, 
then $U/G$ is birational to $f^{-1}(v)/G_v$. 
\end{proposition}

\begin{proof}
Indeed, the fiber $f^{-1}(v)$ is a slice of $U$ in the sense of \cite{Do} with the stabilizer $G_v$. 
\end{proof}

The $G$-action on $U$ is \textit{almost free} if a general point of $U$ has trivial stabilizer. 

\begin{proposition}[no-name method]\label{no-name 1}
If $E\to U$ is a $G$-linearized vector bundle with almost free $G$-action on $U$, 
then the induced map $E/G \dashrightarrow U/G$ is birationally equivalent to
a vector bundle over $U/G$. 
\end{proposition}

When $G$ is reductive, the no-name lemma is a consequence of the descent theory for principal $G$-bundle. 
A proof for general $G$ is given in \cite{C-G-R} Lemma 4.4. 
We will use the no-name method also in the following form. 

\begin{proposition} \label{no-name 2}
Let $E\to U$ be a $G$-linearized vector bundle and let $G_0= \{ g\in G, g|_U=\textrm{id} \}$. 
Suppose that $(i)$ $\overline{G}=G/G_0$ acts almost freely on $U$, 
$(ii)$ $G_0$ acts on $E$ by a scalar multiplication $\alpha : G_0 \to {\C}^{\times}$, 
and $(iii)$ there exists a $G$-linearized line bundle $L\to U$ on which $G_0$ acts by $\alpha$. 
Then ${\proj}E/G$ is birational to ${\proj}^N\times (U/G)$. 
\end{proposition}

\begin{proof}
Apply Proposition \ref{no-name 1} to the $\overline{G}$-linearized vector bundle $E\otimes L^{-1}$. 
We have a canonical identification ${\proj}(E\otimes L^{-1})={\proj}E$. 
Notice that the quotient map $E\otimes L^{-1} \dashrightarrow (E\otimes L^{-1})/\overline{G}$ 
is linear on the fibers (cf. \cite{C-G-R}). 
\end{proof}


\section{Automorphism action on Hirzebruch surfaces}\label{Sec:Hirze}

We prepare miscellaneous results concerning automorphism action on Hirzebruch surfaces. 
For $n\geq0$ let $\mathcal{E}_n$ be the vector bundle ${\Oline}(n)\oplus {\Oline}$ over ${\proj}^1$.  
The projectivization ${\F}_n={\proj}\mathcal{E}_n$ is a Hirzebruch surface.  
Our convention is that a point of ${\F}_n$ represents a line in a fiber of $\mathcal{E}_n$. 
Then the section ${\proj}{\Oline}(n)\subset {\F}_n$ is a $(-n)$-curve which we denote by $\Sigma$. 
Let $\pi : {\F}_n \to {\proj}^1$ be the natural projection. 
If $F$ is a $\pi$-fiber, we shall denote the line bundle ${\OHirze}(a(\Sigma+nF)+bF)$ by $L_{a,b}$. 
The Picard group ${\Pic}({\F}_n)$ consists of the bundles $L_{a,b}$, $a, b\in{\Z}$. 
For example, the section ${\proj}{\Oline}\subset {\F}_n$ belongs to $|L_{1,0}|$; 
the fiber $F$ belongs to $|L_{0,1}|$; 
the canonical bundle $K_{{\F}_n}$ is isomorphic to $L_{-2,n-2}$.  
We have $(L_{a,b}.F)=a$ and $(L_{a,b}.\Sigma)=b$. 

\textit{Except for \S\ref{ssec:trigonal}, we assume $n>0$ in this section.} 
Under this assumption, the action of ${\aut}({\F}_n)$ preserves $\pi$ and $\Sigma$. 
Consequently, we have the exact sequence 
\begin{equation}\label{Aut(Hir)}
1 \to R \to {\aut}({\F}_n) \to {\aut}(\Sigma) \to 1 
\end{equation} 
where $R={\aut}(\mathcal{E}_n)/{\C}^{\times}$. 
Via the given splitting $\mathcal{E}_n={\Oline}(n)\oplus {\Oline}$, 
we may identify 
\begin{equation}\label{desrcibe R}
R = \left\{ g_{\alpha, s}=\begin{pmatrix} \alpha &    s  \\
                                             0    &    1  \end{pmatrix}, \; \; \alpha\in{\C}^{\times}, s\in {\Hom}({\Oline}, {\Oline}(n)) \right\} .
\end{equation}                                              
Thus we have $R \simeq {\C}^{\times}\ltimes H^0({\Oline}(n))$. 
When $n$ is even, $\mathcal{E}_n$ admits a ${\PGL}_2$-linearization via that of ${\Oline}(n)$, 
so that the sequence $(\ref{Aut(Hir)})$ splits.

\subsection{Linearization of line bundles}\label{ssec:linearization}

Let $\widetilde{G}={\SL}_2\ltimes R$ where 
${\SL}_2$ acts on the component $H^0({\Oline}(n))$ of $R$ in the natural way. 
Via the identification \eqref{desrcibe R} and the ${\SL}_2$-linearization of ${\Oline}(n)$, 
the bundle $\mathcal{E}_n$ is $\widetilde{G}$-linearized. 
This induces a surjective homomorphism $\widetilde{G}\to {\aut}({\F}_n)$, 
whose kernel is ${\Z}/2{\Z}$ generated by 
$\rho=(-1, g_{-1, 0})$ when $n$ is odd, and by $\sigma=(-1, g_{1, 0})$ when $n$ is even. 

\begin{lemma}\label{cover linearization Hirze}
Every line bundle on ${\F}_n$ is $\widetilde{G}$-linearized. 
\end{lemma}

\begin{proof}
The ${\SL}_2$-linearization of ${\Oline}(1)$ induces 
a $\widetilde{G}$-linearization of $\pi^{\ast}{\Oline}(1)=L_{0,1}$. 
Moreover, ${\OHirze}(\Sigma)$ is the dual of the tautological bundle over ${\F}_n$, 
which is the blow-up of $\mathcal{E}_n$ along the zero section. 
Hence ${\OHirze}(\Sigma)=L_{1,-n}$ is $\widetilde{G}$-linearized. 
\end{proof}

\begin{proposition}\label{linearization Hirze}
When $n$ is odd, every line bundle on ${\F}_n$ is ${\aut}({\F}_n)$-linearized. 
When $n$ is even, the bundle $L_{a, b}$ admits an ${\aut}({\F}_n)$-linearization if $b$ is even. 
\end{proposition}

\begin{proof}
In the $\widetilde{G}$-linearization of $\mathcal{E}_n$, 
the element $\sigma$ (resp. $\rho$) acts trivially when $n$ is even (resp. odd). 
Hence ${\OHirze}(\Sigma)$ is always ${\aut}({\F}_n)$-linearized. 
Also $L_{0,-2}=\pi^{\ast}K_{{\proj}^1}$ is ${\aut}({\F}_n)$-linearized via the ${\PGL}_2$-linearization of $K_{{\proj}^1}$. 
This proves the assertion for even $n$. 
The homomorphism $\widetilde{G}\to{\C}^{\times}$ defined by $(\gamma, g_{\alpha, s})\mapsto\alpha$ 
induces a $1$-dimensional representation $V$ of $\widetilde{G}$ on which $\rho$ acts by $-1$. 
Hence the bundle $L_{0,1}\simeq L_{0,1}\otimes V$ is ${\aut}({\F}_n)$-linearized for odd $n$. 
\end{proof}

\subsection{Some linear systems}\label{ssec:action on linear system}

We study the ${\aut}({\F}_n)$-action on the linear systems 
$|L_{0,1}|$, $|L_{1,0}|$, $|L_{1,1}|$, and $|L_{2,0}|$. 
Recall that we are assuming $n>0$. 

\begin{proposition}\label{stabilizer of fiber}
The group ${\aut}({\F}_n)$ acts on the linear system $|L_{0,1}|$ transitively, 
and the stabilizer of a point of $|L_{0,1}|$ is connected and solvable. 
\end{proposition}

\begin{proof}
We have a canonical identification $|L_{0,1}| \simeq \Sigma$ given by $F\mapsto F\cap \Sigma$. 
The sequence $(\ref{Aut(Hir)})$ restricted to the stabilizer $G_p$ of a $p \in \Sigma$ gives the exact sequence 
\begin{equation*}
1 \to R \to G_p \to {\aut}(\Sigma, p) \to 1, 
\end{equation*} 
where ${\aut}(\Sigma, p)$ is the stabilizer of $p$ in ${\aut}(\Sigma)$. 
Since both $R$ and ${\aut}(\Sigma, p)\simeq {\C}^{\times}\ltimes{\C}$ are 
connected and solvable, so is $G_p$. 
\end{proof}

Every smooth divisor $H$ in $|L_{1,0}|$ is a section of $\pi$ disjoint from $\Sigma$. 
Hence it corresponds to another splitting $\mathcal{E}_n \simeq {\Oline}(n)\oplus{\Oline}$ 
for which the image of ${\proj}{\Oline}$ is $H$. 
 
\begin{proposition}\label{stabilizer of section} 
The group ${\aut}({\F}_n)$ acts transitively on the open set $U\subset |L_{1,0}|$ of smooth divisors. 
The sequence \eqref{Aut(Hir)} restricted to the stabilizer $G_H\subset{\aut}({\F}_n)$ of an $H\in U$ 
gives the exact sequence 
\begin{equation}\label{stab section}
1 \to {\C}^{\times} \to G_H \to {\aut}(\Sigma) \to 1. 
\end{equation} 
\end{proposition}

\begin{proof}
In fact, the subgroup $R\subset {\aut}({\F}_n)$ acts transitively on $U$ because 
any two splittings $\mathcal{E}_n \simeq {\Oline}(n)\oplus{\Oline}$ are ${\aut}(\mathcal{E}_n)$-equivalent. 
When $H$ corresponds to the original splitting of $\mathcal{E}_n$, 
we have $G_H\cap R=\{ g_{\alpha, 0}, \alpha\in{\C}^{\times}\}$.  
The homomorphism $G_H \to {\aut}(\Sigma)$ is surjective thanks to 
the ${\SL}_2$-linearization of ${\Oline}(n)\oplus{\Oline}$.  
\end{proof}

\begin{proposition}\label{L_{1,1}}
The group ${\aut}({\F}_n)$ acts on $|L_{1,1}|$ almost transitively 
with the stabilizer of a general point being connected and solvable. 
\end{proposition}

\begin{proof}
We first show that the subgroup $R\subset{\aut}({\F}_n)$ acts on $|L_{1,1}|$ almost freely. 
Indeed, if an element $g\in R$ preserves (and hence fixes) a smooth $D\in|L_{1,1}|$,  
then for a general $H\in|L_{1,0}|$ the $n+1$ points $H\cap D$ are fixed by $g$. 
Since $(L_{1,0}.L_{1,0})=n$, then $g$ fixes $H$ and so must be trivial. 
The group $R$ preserves the sub linear system ${\proj}V_p\subset|L_{1,1}|$ 
of curves passing through each $p\in\Sigma$. 
Since ${\dim}\, R={\dim}\,{\proj}V_p=n+2$, 
the $R$-action on ${\proj}V_p$ is also almost transitive. 
Thus ${\aut}({\F}_n)$ acts on $|L_{1,1}|$ almost transitively. 
Let $G\subset{\aut}({\F}_n)$ be the stabilizer of a general $D\in|L_{1,1}|$. 
Since $G\cap R=\{ {\rm id}\}$, the natural homomorphism $G\to {\aut}(\Sigma)$ is injective. 
Its image is contained in the stabilizer ${\aut}(\Sigma, p)$ of the point $p=D\cap\Sigma$. 
Then the embedding $G\to{\aut}(\Sigma, p)$ is open because ${\dim}\, G=2$. 
Therefore $G$ is connected and solvable. 
\end{proof}

Every smooth divisor in $|L_{2,0}|$ is a hyperelliptic curve of genus $n-1$ whose $g^1_2$ is given by the restriction of $\pi$. 

\begin{proposition}\label{HE moduli}
If $U\subset |L_{2,0}|$ is the open set of smooth divisors,  
a geometric quotient $U/{\aut}({\F}_n)$ exists and is isomorphic to 
the moduli space $\mathcal{H}_{n-1}$ of hyperelliptic curves of genus $n-1$. 
\end{proposition}

\begin{proof}
Recall that $\mathcal{H}_{n-1}$ is a normal irreducible variety. 
We have a natural ${\aut}({\F}_n)$-invariant morphism $\psi\colon U\to\mathcal{H}_{n-1}$.  
In order to show that $\psi$ is surjective, 
let $C$ be a hyperelliptic curve of genus $n-1$ and let $\pi\colon C\to {\proj}^1$ be its $g^1_2$. 
The curve $C$ is naturally embedded in ${\proj}(\pi_{\ast}\mathcal{O}_C)^{\vee}$ over ${\proj}^1$ 
by the evaluation map. 
Let $\pi_{\ast}\mathcal{O}_C = \mathcal{L}_+\oplus\mathcal{L}_-$ be the decomposition 
with respect to the hyperelliptic involution $\iota$, where $\iota$ acts on $\mathcal{L}_{\pm}$ by $\pm1$. 
It is clear that $\mathcal{L}_+\simeq{\Oline}$, 
and a cohomology calculation shows that $\mathcal{L}_-\simeq{\Oline}(-n)$. 
Thus ${\proj}(\pi_{\ast}\mathcal{O}_C)^{\vee}$ is isomorphic to ${\F}_n$. 
Then $C$ belongs to $|L_{2,b}|$ for some $b\geq0$, 
and the genus formula shows that $b=0$. 
Therefore $\psi$ is surjective. 
Conversely, for a $C\in U$ 
the natural embedding $C\subset{\proj}(\pi_{\ast}\mathcal{O}_C)^{\vee}$ extends to 
an isomorphism ${\F}_n\to{\proj}(\pi_{\ast}\mathcal{O}_C)^{\vee}$ 
(e.g., by the evaluation at $C$ via the canonical identifications 
${\proj}(\pi_{\ast}\mathcal{O}_C)^{\vee}\simeq{\proj}(\pi_{\ast}\mathcal{O}_C)$ and 
${\F}_n\simeq{\proj}\mathcal{E}_n^{\vee}$). 
This implies that the $\psi$-fibers are ${\aut}({\F}_n)$-orbits. 
Now our assertion follows from \cite{GIT} Proposition 0.2.
\end{proof}

By the proof, the hyperelliptic involution of a smooth $C\in|L_{2,0}|$ uniquely extends to an involution $\iota_C$ of ${\F}_n$. 
Concretely, for each $\pi$-fiber $F$ consider the involution of $F$ which fixes the point $F\cap\Sigma$ 
and exchanges the two points $F\cap C$ (or fixes $F\cap C$ when $F$ is tangent to $C$). 
This defines $\iota_C$. 

\begin{corollary}\label{stab HE}
When $n\geq3$, the stabilizer in ${\aut}({\F}_n)$ of a general $C\in |L_{2,0}|$ is $\langle\iota_C\rangle$. 
\end{corollary}

\begin{proof}
By Proposition \ref{HE moduli} $C$ has no automorphism other than its hyperelliptic involution. 
Any automorphism of ${\F}_n$ acting trivially on $C$ must be trivial because 
it fixes three points of general $\pi$-fibers. 
\end{proof}

\subsection{Trigonal curves}\label{ssec:trigonal} 

\textit{We allow $n\geq0$ in this subsection.} 
It is known that trigonal curves are naturally related to Hirzebruch surfaces. 
We recall here some basic facts (cf. \cite{A-C-G-H}, \cite{M-Sc}). 
A canonically embedded trigonal curve $C\subset {\proj}^{g-1}$ of genus $g\geq5$ 
is contained in a unique rational normal scroll, that is, 
the image of a Hirzebruch surface ${\F}_n$ by a linear system $|L_{1,m}|$. 
The scroll is swept out by the lines spanned by the fibers of the trigonal map $C\to {\proj}^1$ (which is unique),   
and may also be cut out by the quadrics containing $C$. 
The number $n$ is the \textit{scroll invariant} of $C$, 
and the number $m$ is the \textit{Maroni invariant} of $C$. 
If we regard $C$ as a curve on ${\F}_n$, 
the system $|L_{1,m}|$ gives the canonical system of $C$, 
and $|L_{0,1}|$ gives the $g^1_3$ of $C$. 
Then $C$ belongs to $|L_{3,b}|$ for $b=m-n+2$. 
The genus formula derives the relation $g=3n+2b-2$. 
Conversely, for a smooth curve $C$ on ${\F}_n$ with $C\in|L_{3,b}|$, 
the linear system $|L_{1,m}|$ with $m=b+n-2$ is identified with $|K_C|$ by restriction, 
and the curve $\phi_{L_{1,m}}(C)$ in $|L_{1,m}|^{\vee}$ is a canonical model of $C$  
which is contained in the scroll $\phi_{L_{1,m}}({\F}_n)$. 

For $g\geq5$ we denote by $\mathcal{T}_{g,n}\subset\mathcal{M}_g$ 
the locus of trigonal curves of scroll invariant $n$ in the moduli space of genus $g$ curves. 
In some literature, $\mathcal{T}_{g,n}$ is called a \textit{Maroni locus}. 
General trigonal curves have scroll invariant $0$ or $1$ depending on whether $g$ is even or odd. 
The above facts infer the following. 

\begin{proposition}\label{moduli of trigonal}
For $g\geq5$ and $n\geq0$ 
the space $\mathcal{T}_{g,n}$ is naturally birational to a rational quotient $|L_{3,b}|/G$, 
where $2b=g-3n+2$ and $G$ is the identity component of ${\aut}({\F}_n)$. 
\end{proposition}

\subsection{A coordinate system}\label{ssec:coordinate}

The  Hirzebruch surface ${\F}_n$ has a natural coordinate system. 
Let $[X,Y]$ be the homogeneous coordinate of ${\proj}^1$.   
Let $V_0=\{ Y\ne0\}$ and $V_1=\{ X\ne0\}$ be open sets of ${\proj}^1$. 
We fix the original splitting $\mathcal{E}_n={\Oline}(n)\oplus{\Oline}$ and 
denote $H={\proj}{\Oline}\subset{\F}_n$. 
An open covering $\{ U_i\}_{i=1}^{4}$ of ${\F}_n$ is defined by 
\begin{eqnarray*}
U_1=\pi^{-1}(V_0)-H, & &   U_2=\pi^{-1}(V_1)-H, \\ 
U_3=\pi^{-1}(V_0)-\Sigma,    & &   U_4=\pi^{-1}(V_1)-\Sigma.  
\end{eqnarray*} 
Let $\mathbf{1}\in H^0({\Oline})$ be the section given by the constant function $1$. 
Let $s_0, s_1\in H^0({\Oline}(n))$ be the sections given by $s_0=Y^n$ and $s_1=X^n$ respectively. 
Note that $s_1=(Y^{-1}X)^n s_0$ where we regard $Y^{-1}X \in H^0(\mathcal{O}_{V_0\cap V_1})$. 
We shall use $(\mathbf{1}, s_i)$ as a local frame of $\mathcal{E}_n$ over $V_i$. 
Then we have an isomorphism $U_1\to{\C}^2$ (resp. $U_3\to{\C}^2$) given by 
\begin{equation*}
([X,Y], \; {\C}(a\mathbf{1}+bs_0)) \mapsto (Y^{-1}X, b^{-1}a) \qquad 
(\textrm{resp.} \: \: \mapsto (Y^{-1}X, a^{-1}b)),  
\end{equation*}
and an isomorphism $U_2\to{\C}^2$ (resp. $U_4\to{\C}^2$) given by 
\begin{equation*}
([X,Y], \; {\C}(a\mathbf{1}+cs_1)) \mapsto (X^{-1}Y, c^{-1}a) \qquad 
(\textrm{resp.} \: \: \mapsto (X^{-1}Y, a^{-1}c)). 
\end{equation*}
Thus the open sets $U_i \simeq {\C}^2$ have coordinates $(x_i, y_i)$ glued by 
\begin{equation*}
x_1=x_3=x_2^{-1}=x_4^{-1}, 
\end{equation*}
\begin{equation*}
y_3=y_1^{-1}, \quad y_4=y_2^{-1}, \quad y_2=x_1^ny_1, \quad y_4=x_3^{-n}y_3.
\end{equation*}
The $(-n)$-curve $\Sigma$ is defined by the equation $y_1=y_2=0$. 

Let us describe some of the ${\aut}({\F}_n)$-action on ${\F}_n$ in terms of these coordinates. 
The elements $g_{\alpha, s}\in R$ leave $U_3$ invariant. 
If $s\in H^0({\Oline}(n))$ is written as $s=\sum_{i=0}^{n}\lambda_iX^iY^{n-i}$, 
then $g_{\alpha, s}$ acts by 
\begin{equation}\label{$R$-action in coordinate}
g_{\alpha, s} : U_3\ni(x_3, y_3) \mapsto (x_3, \alpha y_3 + \sum_{i=0}^{n}\lambda_ix_3^i)\in U_3. 
\end{equation}
The group ${\GL}_2$ acts on ${\F}_n$ via the ${\GL}_2$-linearization of $\mathcal{E}_n$. 
Then the elements 
$h_{\beta}=\begin{pmatrix} \beta &    0  \\
                                               0   &    1  \end{pmatrix}$ 
and   
$j=\begin{pmatrix} 0  &  1  \\
                              1  &  0  \end{pmatrix}$ 
of ${\GL}_2$ act respectively by                                                                                                        
\begin{equation}\label{linearization coordinate 1}
h_{\beta} : U_3\ni(x_3, y_3) \mapsto (\beta x_3, y_3)\in U_3, 
\end{equation}
\begin{equation}\label{linearization coordinate 2}
j : U_3\ni(x_3, y_3) \mapsto (x_3, y_3)\in U_4. 
\end{equation}
In \eqref{linearization coordinate 2}, the second $(x_3, y_3)$ is the coordinate in $U_4$. 
This action is also regarded as the rational transformation 
$U_3\ni(x_3, y_3) \mapsto (x_3^{-1}, x_3^{-n}y_3)\in U_3$. 

We describe the linear systems $|L_{a,b}|$ in terms of these coordinates. 
We may assume $a, b\geq0$, 
which is equivalent to the condition that $|L_{a,b}|$ has no fixed component. 
If $C\in|L_{a,b}|$,  
then $C\cap U_1$ is a curve on $U_1$ and thus defined by the equation 
$F(x_1, y_1)=0$ for a polynomial $F(x_1, y_1)$. 

\begin{proposition}\label{def eqn in Hirze}
For $a, b\geq0$ the linear system $|L_{a,b}|$ is identified by restriction to $U_1$ with 
the projectivization of the vector space 
$\{ \sum_{i=0}^{a}f_i(x_1)y_1^i, \: {\rm deg}f_i\leq b+in \}$ 
of polynomials of $x_1, y_1$.   
\end{proposition}

\begin{proof}
Let $F(x_1, y_1)=0$ be an equation of $C|_{U_1}$ for a general $C\in|L_{a,b}|$. 
Expanding $F(x_1, y_1)=\sum_{i=0}^{d}f_i(x_1)y_1^i$, 
we have $d=(C.F)=a$ and ${\rm deg}f_0=(C.\Sigma)=b$. 
By substitution, the curve $C|_{U_2}$ on $U_2$ is defined by  
$x_2^kF(x_2^{-1}, x_2^ny_2)=0$ for some $k$. 
Putting $y_2=0$, we know that $k=b$. 
Since $x_2^bF(x_2^{-1}, x_2^ny_2)$ should be a polynomial of $x_2, y_2$, 
we must have ${\rm deg}f_i\leq b+in$. 
Then the equality $h^0(L_{a,b})= \chi(L_{a,b})=\frac{1}{2}(a+1)(an+2b+2)$ concludes the proof. 
\end{proof}

A defining polynomial $\sum_{i=0}^{a}f_i(x_1)y_1^i$ in $U_1$ for a curve $C\in |L_{a,b}|$ 
is transformed 
into $\sum_{i=0}^{a}f_i(x_2^{-1})x_2^{b+in}y_2^i$ in $U_2$,  
into $\sum_{i=0}^{a}f_i(x_3)y_3^{a-i}$ in $U_3$,  and                                   
into $\sum_{i=0}^{a}f_i(x_4^{-1})x_4^{b+in}y_4^{a-i}$ in $U_4$.


\section{2-elementary $K3$ surface}\label{sec: 2-ele K3}

\subsection{2-elementary $K3$ surface}\label{ssec: 2-ele K3}

We review basic theory of 2-elementary $K3$ surfaces following \cite{A-N} and \cite{Yo1}. 
Let $(X, \iota)$ be a 2-elementary $K3$ surface, namely 
$X$ is a complex $K3$ surface and $\iota$ is a non-symplectic involution on $X$. 
The presence of $\iota$ implies that $X$ is algebraic. 
We denote by $L_{\pm}(X, \iota)\subset{\cohomology}$ 
the lattice of cohomology classes $l$ with $\iota^{\ast}l=\pm l$. 
Then $L_+(X, \iota)$ is the orthogonal complement of $L_-(X, \iota)$ 
and contained in the N\'eron-Severi lattice $NS_X$. 
If $r$ is the rank of $L_+(X, \iota)$, 
then $L_+(X, \iota)$ and $L_-(X, \iota)$ have signature $(1, r-1)$ and $(2, 20-r)$ respectively. 
Let $L_{\pm}(X, \iota)^{\vee}$ be the dual lattice of $L_{\pm}(X, \iota)$. 
The discriminant form of $L_{\pm}(X, \iota)$ is the finite quadratic form $(D_{L_{\pm}}, q_{\pm})$ where  
$D_{L_{\pm}}=L_{\pm}(X, \iota)^{\vee}/L_{\pm}(X, \iota)$ and   
$q_{\pm}:D_{L_{\pm}}\to{\Q}/2{\Z}$ is induced by the quadratic form on $L_{\pm}(X, \iota)^{\vee}$. 
We have a canonical isometry $(D_{L_+}, q_+)\simeq(D_{L_-}, -q_-)$. 
The abelian group $D_{L_+}$ is 2-elementary, namely $D_{L_+}\simeq({\Z}/2{\Z})^a$ for some $a\geq0$. 
The parity $\delta$ of $q_+$ is defined by 
$\delta=0$ if $q_+(D_{L_+})\subset {\Z}$, 
and $\delta=1$ otherwise. 
The triplet $(r, a, \delta)$ is called the \textit{main invariant} of the lattice $L_+(X, \iota)$, 
and also of the 2-elementary $K3$ surface $(X, \iota)$. 
By \cite{Ni1}, the isometry class of $L_{\pm}(X, \iota)$ is uniquely determined by $(r, a, \delta)$. 

\begin{proposition}[Nikulin \cite{Ni2}]\label{compute main inv}
Let $X^{\iota}\subset X$ be the fixed locus of $\iota$. 

$({\rm i})$ If $(r, a, \delta)=(10, 10, 0)$, then $X^{\iota} = \emptyset$. 

$({\rm ii})$ If $(r, a, \delta)=(10, 8, 0)$, then $X^{\iota}$ is a union of two elliptic curves. 

$({\rm iii})$ In other cases, $X^{\iota}$ is decomposed as 
$X^{\iota} = C^g \sqcup E_{1} \sqcup \cdots \sqcup E_{k}$ 
such that $C^g$ is a genus $g$ curve and $E_{1}, \cdots, E_{k}$ are rational curves with 
\begin{equation}\label{compute r and a}
g=11-\frac{r+a}{2}, \qquad k=\frac{r-a}{2}. 
\end{equation}
One has $\delta=0$ if and only if the class of $X^{\iota}$ is divisible by $2$ in $L_+(X, \iota)$. 
\end{proposition}


\begin{theorem}[Nikulin \cite{Ni2}]\label{classify 2-ele K3}
The deformation type of a 2-elementary $K3$ surface $(X, \iota)$ is determined by 
its main invariant $(r, a, \delta)$. 
All possible main invariants of 2-elementary $K3$ surfaces are shown on the following Figure \ref{Nikulin table} 
(which is identical to the table in \cite{A-N} page 31). 
\begin{figure}[h]
\centerline{\includegraphics[width=9.5cm]{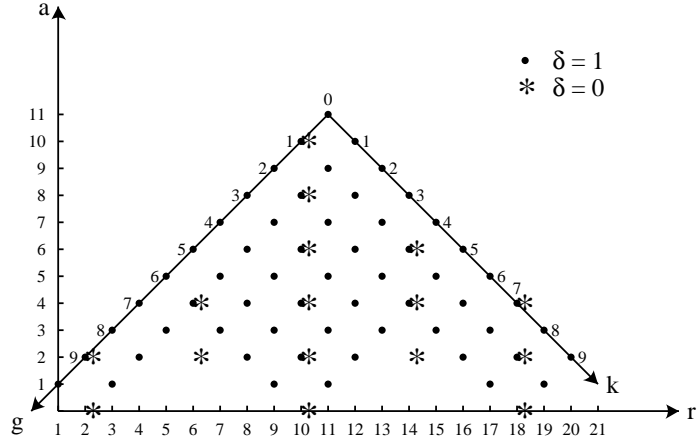}} 
\caption{Distribution of main invariants $(r,a,\delta)$} 
\label{Nikulin table}
\end{figure}
\end{theorem}

A moduli space of 2-elementary $K3$ surfaces of type $(r, a, \delta)$ is constructed as follows. 
Let $L_-$ be an even lattice of signature $(2, 20-r)$ 
whose discriminant form is 2-elementary of length $a$ and parity $\delta$. 
The orthogonal group ${\Or}(L_-)$ of $L_-$ acts on the domain 
\begin{equation}\label{eqn:period domain}
\Omega_{L_-}=\{ {\C}\omega \in {\proj}(L_-\otimes{\C}) \; | \; (\omega, \omega)=0, (\omega, \overline{\omega})>0 \}. 
\end{equation}
The quotient space $\mathcal{F}({\Or}(L_-))={\Or}(L_-)\backslash\Omega_{L_-}$ turns out to be 
an irreducible, normal, quasi-projective variety of dimension $20-r$. 
The complex analytic divisor $\sum \delta^{\perp} \subset \Omega_{L_-}$, 
where $\delta$ are $(-2)$-vectors in $L_-$, 
is the inverse image of an algebraic divisor $\mathcal{D} \subset \mathcal{F}({\Or}(L_-))$. 
We set 
\begin{equation*}
{\moduli}=\mathcal{F}({\Or}(L_-))-\mathcal{D}. 
\end{equation*} 

For a 2-elementary $K3$ surface $(X, \iota)$ of type $(r, a, \delta)$  
we have an isometry $\Phi \colon L_-(X, \iota) \to L_-$. 
Then $\Phi(H^{2,0}(X))$ is contained in $\Omega_{L_-}$. 
The period of $(X, \iota)$ is defined by  
\begin{equation*}
\mathcal{P}(X, \iota) = [\Phi(H^{2,0}(X))] \in {\moduli}, 
\end{equation*}
which is independent of the choice of $\Phi$. 

\begin{theorem}[Yoshikawa \cite{Yo1}, \cite{Yo3}]\label{thm: moduli}
The variety ${\moduli}$ is a moduli space of 2-elementary $K3$ surfaces of type $(r, a, \delta)$ in the sense that,  
for a family $(\frak{X}\to U, \iota)$ of such 2-elementary $K3$ surfaces  
the period map $U \to {\moduli}$, $u \mapsto \mathcal{P}(\frak{X}_{u}, \iota _{u})$, 
is a morphism of varieties, 
and via the period mapping the points of ${\moduli}$ bijectively correspond to 
the isomorphism classes of such 2-elementary $K3$ surfaces. 
\end{theorem}

Let $\mathcal{M}_g$ be the moduli of genus $g$ curves. 
When $(r, a, \delta)\ne(10, 10, 0), (10, 8, 0)$, 
setting $g=11-2^{-1}(r+a)$, 
we have the \textit{fixed curve map} 
\begin{equation}\label{fixed curve map}
F : {\moduli} \to \mathcal{M}_g, \quad (X, \iota)\mapsto C^g, 
\end{equation}
where $C^g$ is the genus $g$ component of $X^{\iota}$.  
In this article we will determine the generic structure of $F$ in terms of $\mathcal{M}_g$ for $g\geq3$. 
For example, one will find that  

\begin{itemize}
\item $F$ is generically injective for $r\leq5$ and for $8\leq r\leq12$, $a\leq2$.  

\item The members of $F({\moduli})$ have Clifford index $\leq2$. 
         If $k>0$ in addition, they have Clifford index $\leq1$. 

\item When $r=2+4n$, we have $\delta=0$ if and only if the generic member of $F({\moduli})$ 
         possesses a theta characteristic of projective dimension $3-n$.  
\end{itemize}


\subsection{DPN pair}\label{ssec: right resol}

We shall explain a generalized double cover construction of 2-elementary $K3$ surfaces. 
Recall from \cite{A-N} that a \textit{DPN pair} is a pair $(Y, B)$ of a smooth rational surface $Y$ 
and a bi-anticanonical curve $B\in|\!-\!2K_Y|$ with only A-D-E singularities. 
When $B$ is smooth, $(Y, B)$ is called a \textit{right DPN pair}. 
2-elementary $K3$ surfaces $(X, \iota)$ with $X^{\iota}\ne \emptyset$ are 
in canonical correspondence with right DPN pairs: 
for such an $(X, \iota)$ the quotient $Y=X/\iota$ is a smooth rational surface, 
and the branch curve $B$ of the quotient map $X\to Y$ is a $-2K_Y$-curve isomorphic to $X^{\iota}$. 
Conversely, for a right DPN pair $(Y, B)$ the double cover $f\colon X\to Y$ branched along $B$ is a $K3$ surface, 
with the covering transformation being a non-symplectic involution. 
From $B$ one knows the invariant $(r, a)$ of $X$ via Proposition \ref{compute main inv}. 
Also one has $r=\rho(Y)$. 
The lattice $L_+(X, \iota)$ is generated by the sublattice $f^{\ast}NS_Y$ and 
the classes of components of $X^{\iota}$ (cf. \cite{Ma}). 
By \cite{A-N}, if $B=\sum_{i=0}^{k}B_i$ is the irreducible decomposition of $B$, 
then $(X, \iota)$ has parity $\delta=0$ if and only if 
$\sum_{i=0}^{k}(-1)^{n_i}B_i \in 4NS_Y$ for some $n_i \in \{0, 1\}$. 

Let $(Y, B)$ be a DPN pair. 
A \textit{right resolution} of $(Y, B)$ is a triplet $(Y', B', \pi)$ such that $(Y', B')$ is a right DPN pair 
and $\pi\colon Y'\to Y$ is a birational morphism with $\pi(B')=B$. 
When $Y$ is obvious from the context, we also call it a right resolution of $B$. 
A right resolution exists and is unique up to isomorphism. 
It may be constructed explicitly as follows (\cite{A-N}). 
Let 
\begin{equation}\label{process of right resol} 
\cdots \stackrel{\pi_{i+1}}{\to} (Y_{i}, B_{i}) \stackrel{\pi_{i}}{\to} (Y_{i-1}, B_{i-1}) \stackrel{\pi_{i-1}}{\to} \cdots 
\stackrel{\pi_{1}}{\to} (Y_0, B_0)=(Y, B)  
\end{equation} 
be the blow-ups defined inductively by 
$\pi_{i+1}:Y_{i+1}\to Y_i$ being the blow-up at the singular points of $B_i$,   
and $B_{i+1}=\widetilde{B}_i + \sum_{p}E_{p}$ where $\widetilde{B}_i$ is the strict transform of $B_i$ and  
$E_p$ are the $(-1)$-curves over the triple points $p$ of $B_i$.  
Each $(Y_i, B_i)$ is also a DPN pair. 
This process will terminate and we finally obtain a right DPN pair $(Y', B')=(Y_N, B_N)$. 
Let $p$ be a singular point of $B$.  
According to the type of singularity, 
the dual graph of the curves on $Y'$ contracted to $p$ is as follows.

\begin{picture}(100,25)\label{graph A_n}
	\thicklines	
	\put(10,-5){{\large $A_n$}}	
	\put(60,0){\circle{8}}
	\put(80,0){\circle*{8}}
	\put(100,0){\circle*{8}}
	\put(140,0){\circle*{8}}	
	\put(64,0){\line(1,0){12}}
	\put(84,0){\line(1,0){12}}
	\put(100,-5){\makebox(40,10){$\centerdot \: \centerdot \: \centerdot  $}}	
	\put(64,-11){1}
	\put(84,-11){2}
	\put(104,-11){3}
	\put(144,-11){$[\frac{n+1}{2}]$}
\end{picture}

\begin{picture}(100,45)\label{graph D_2n}
	\thicklines	
	\put(10,-5){{\large $D_{2n}$}}	
	\put(64,16){\circle{8}}
	\put(64,-16){\circle{8}}
     \put(80,0){\circle{8}}           \put(80,0){\circle{4}}
	\put(100,0){\circle{8}}
	\put(140,0){\circle{8}}        \put(140,0){\circle{4}}
     \put(160,0){\circle{8}}        
	\put(77,3){\line(-1,1){10}}
	\put(77,-3){\line(-1,-1){10}}
	\put(84,0){\line(1,0){12}}
	\put(100,-5){\makebox(40,10){$\centerdot \: \centerdot \: \centerdot $}}
	\put(144,0){\line(1,0){12}}	
	\put(70,15){1}
	\put(70,-25){2}
	\put(84,-11){3}
	\put(104,-11){4}
	\put(164,-11){$2n$}		
\end{picture}

\begin{picture}(100,60)\label{graph D_2n+1}
	\thicklines	
	\put(10,-5){{\large $D_{2n+1}$}}	
	\put(64,16){\circle*{8}}
        \put(80,0){\circle{8}}           
	\put(100,0){\circle{8}}        \put(100,0){\circle{4}}
	\put(120,0){\circle{8}}        
        \put(160,0){\circle{8}}        \put(160,0){\circle{4}}
        \put(180,0){\circle{8}}    
	\put(77,3){\line(-1,1){10}}
	\put(84,0){\line(1,0){12}}
	\put(104,0){\line(1,0){12}}
        \put(120,-5){\makebox(40,10){$\centerdot \: \centerdot \: \centerdot $}}
	\put(164,0){\line(1,0){12}}	
	\put(70,15){1}
	\put(84,-11){2}
	\put(104,-11){3}
	\put(124,-11){4}
        \put(184,-11){$2n$}	
\end{picture}

\begin{picture}(100,50)\label{graph E_6}
	\thicklines
	\put(10,-5){{\large $E_6$}}	
	\put(60,0){\circle*{8}}
	\put(80,0){\circle*{8}}
	\put(100,0){\circle{8}}
	\put(100,20){\circle{8}}           \put(100,20){\circle{4}}	
	\put(64,0){\line(1,0){12}}
	\put(84,0){\line(1,0){12}}
	\put(100,4){\line(0,1){12}}			
	\put(64,-11){1}
	\put(84,-11){2}
	\put(104,-11){3}
	\put(104,11){4}
\end{picture}

\begin{picture}(100,45)\label{graph E_7}
	\thicklines
	\put(10,-5){{\large $E_7$}}	
	\put(60,0){\circle{8}}
	\put(80,0){\circle{8}}        \put(80,0){\circle{4}}
	\put(100,0){\circle{8}}
	\put(120,0){\circle{8}}        \put(120,0){\circle{4}}
     \put(120,20){\circle{8}}           
	\put(140,0){\circle{8}}        
     \put(160,0){\circle{8}}        \put(160,0){\circle{4}}
	\put(64,0){\line(1,0){12}}
	\put(84,0){\line(1,0){12}}
	\put(104,0){\line(1,0){12}}
	\put(120,4){\line(0,1){12}}	
	\put(124,0){\line(1,0){12}}
	\put(144,0){\line(1,0){12}}		
	\put(64,-11){1}
	\put(84,-11){2}
	\put(104,-11){3}
	\put(124,-11){4}
	\put(124,11){5}
	\put(144,-11){6}
	\put(164,-11){7}
\end{picture}

\begin{picture}(100,50)\label{graph E_8}
	\thicklines	
	\put(10,-5){{\large $E_8$}}	
	\put(60,0){\circle{8}}        \put(60,0){\circle{4}}
	\put(80,0){\circle{8}}
	\put(100,0){\circle{8}}        \put(100,0){\circle{4}}
	\put(120,0){\circle{8}}
	\put(140,0){\circle{8}}        \put(140,0){\circle{4}}
     \put(140,20){\circle{8}}           
	\put(160,0){\circle{8}}        
     \put(180,0){\circle{8}}        \put(180,0){\circle{4}}
	\put(64,0){\line(1,0){12}}
	\put(84,0){\line(1,0){12}}
	\put(104,0){\line(1,0){12}}
	\put(124,0){\line(1,0){12}}
	\put(140,4){\line(0,1){12}}	
	\put(144,0){\line(1,0){12}}
	\put(164,0){\line(1,0){12}}		
	\put(64,-11){1}
	\put(84,-11){2}
	\put(104,-11){3}
	\put(124,-11){4}
	\put(144,-11){5}
	\put(144,11){6}
        \put(164,-11){7}
        \put(184,-11){8}
\end{picture}

\begin{figure}[h]
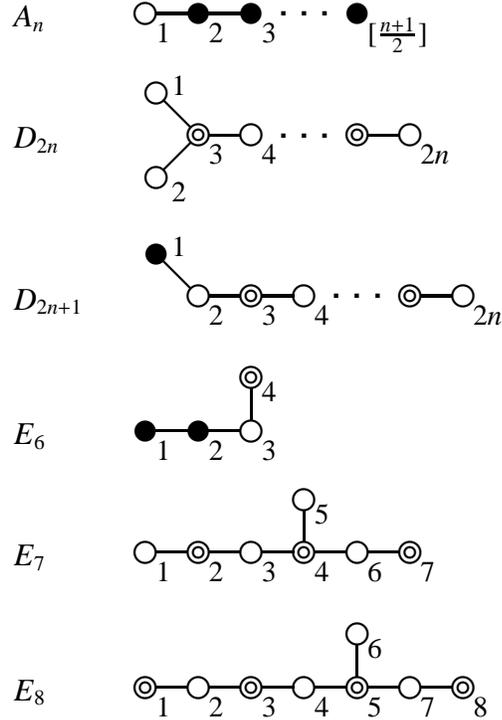

\caption{Dual graphs of exceptional curves}
\label{dual graph}
\end{figure}


Here 
black vertices represent $(-2)$-curves, 
white vertices represent $(-1)$-curves, 
and double circles represent $(-4)$-curves. 
The $(-4)$-curves are components of $B'$, while the $(-2)$-curves are disjoint from $B'$. 
The $(-1)$-curves intersect with $B'$ transversely at two points unless $p$ is $A_{2n}$-type; 
when $p$ is an $A_{2n}$-point, the $(-1)$-curve is tangent to $B'$ at one point. 
The labeling for the vertices will be used later. 
Note that identification of the dual graph of the curves with the above abstract graph 
is not unique when $p$ is $D_{2n}$-type. 
In that case, such an identification is obtained after one distinguishes the three branches 
(resp. two tangential branches) of $B$ at $p$ when $n=2$ (resp. $n>2$).

Let $(Y, B)$ be a DPN pair with a right resolution $(Y', B', \pi)$. 
Taking the double cover of $Y'$ branched over $B'$, 
we associate a 2-elementary $K3$ surface $(X, \iota)$ to $(Y, B)$. 
The composition $X\to Y$ of the quotient map $X\to Y'$ and the blow-down $\pi$ 
is called the \textit{right covering map} for $(Y, B)$. 
Note that $(X, \iota)$ is also the minimal desingularization of the double cover of $Y$ branched over $B$. 
By the above dual graphs, 
the invariant $(r, a)$ of $(X, \iota)$ is calculated in terms of $(Y, B)$ as follows. 
Let $a_n$ (resp. $d_n$, $e_n$) be the number of singularities of $B$ of type $A_n$ (resp. $D_n$, $E_n$). 
Then $r=\rho(Y')$ is given by 
\begin{equation}\label{compute r by A-D-E}
\rho(Y) + \sum_{l\geq1} l \ (a_{2l-1}+a_{2l}) + \sum_{m\geq2}2m(d_{2m}+d_{2m+1}) + 4e_6 + 7e_7 + 8e_8. 
\end{equation}
If $k_0$ is the number of components of $B$, the number $k+1$ of components of $X^{\iota}\simeq B'$ is given by 
\begin{equation}\label{compute k by A-D-E}
k_0+\sum_{m\geq2}(m-1)(d_{2m}+d_{2m+1}) + e_6 + 3e_7 + 4e_8. 
\end{equation}
The genus $g$ of the main component of $X^{\iota}$ is the maximal geometric genus of components of $B$. 
For the parity $\delta$ we will later use the following criteria. 

\begin{lemma}\label{delta=1}
Let $(X, \iota)$ and $(Y, B)$ be as above. 
Then $(X, \iota)$ has parity $\delta=1$ if we have distinct irreducible components $B_i$ of $B$ 
with either of the following conditions. 

$(1)$ $B_i\cap B_j$ contains a node of $B$ for every $1\leq i, j\leq3$.  

$(2)$ $B_1\cap B_2$ contains a node and a $D_4$-point of $B$. 

$(3)$ $B_1\cap B_2$ contains a node and a $D_{2n}$-point of $B$, 
in the latter of which $B_1$ and $B_2$ share a tangent direction. 

$(4)$ $B_1$ has a node which is also a node of $B$. 
\end{lemma}

\begin{proof}
$(1)$ Let $p_i$ be a node of $B$ contained in $B_j\cap B_k$ with $\{ i, j, k\}=\{1, 2, 3\}$. 
Let $C_i\subset X$ be the $(-2)$-curve over $p_i$. 
It suffices to show that the ${\Q}$-cycle $D=2^{-1}\sum_{i=1}^{3}C_i$ belongs to $L_+(X, \iota)^{\vee}$. 
If $F_i$ is the component of $X^{\iota}$ over $B_i$, 
then $(D.F_i)=1$ for every $1\leq i\leq3$. 
We have $(D.F)=0$ for other components $F$ of $X^{\iota}$. 
Since $2D$ is the pullback of a divisor on $X/\iota$, 
$D$ has integral intersection pairing with the pullbacks of divisors on $X/\iota$.  
This proves the assertion. 

$(2)$ Let $Y''$ be the blow-up of $Y$ at the $D_4$-point, 
$E\subset Y''$ the exceptional curve, and  
$\widetilde{B}_i$ (resp. $\widetilde{B}$) the strict transform of $B_i$ (resp. $B$) in $Y''$. 
Then one may apply the case $(1)$ to the DPN pair $(Y'', \widetilde{B}+E)$ and the components 
$\widetilde{B}_1, \widetilde{B}_2, E$ of $\widetilde{B}+E$. 

$(3)$ Blow-up $Y$ at the $D_{2n}$-singularity and use the induction on $n$: 
the assertion is reduced to the case $(2)$. 

$(4)$ If $C\subset X$ is the $(-2)$-curve over that node, we have $2^{-1}C\in L_+(X, \iota)^{\vee}$. 
\end{proof}

In certain cases, 
the right covering map $f\colon X\to Y$ may be recovered from a line bundle on $X$. 

\begin{lemma}[cf. \cite{Ma}]\label{covering map & LB}
Let $(X, \iota)$ and $(Y, B)$ be as above and suppose that $Y$ is either 
${\proj}^2$ or  ${\proj}^1\times{\proj}^1$ or ${\F}_n$ with $n>0$.  
Let $L\in{\rm Pic}(Y)$ be ${\Oplane}(1)$, $\mathcal{O}_{{\proj}^1\times{\proj}^1}(1, 1)$, $L_{1,0}$ 
for respective case. 
Then the map $f^{\ast}\colon |L| \to |f^{\ast}L|$ is isomorphic. 
\end{lemma} 

\begin{proof}
The bundle $f^{\ast}L$ is nef of degree $2(L.L)>0$ so that 
$h^0(f^{\ast}L)=\chi(f^{\ast}L)=2+(L.L)$. 
Hence the assertion follows from the coincidence $h^0(L)=2+(L.L)$. 
\end{proof}

By this lemma, the morphism $\phi_{f^{\ast}L}\colon X \to |f^{\ast}L|^{\vee}$ is 
identified with the composition $\phi_L\circ f \colon X \to Y \to |L|^{\vee}$. 
The morphism $\phi_L$ is an embedding when $Y={\proj}^2$ or ${\proj}^1\times{\proj}^1$. 
If $Y={\F}_n$ with $n>0$, $\phi_L$ contracts the $(-n)$-curve $\Sigma$. 
Therefore, for $Y\ne{\F}_1$, 
one may recover the morphism $f\colon X\to Y$ from the bundle $f^{\ast}L$  
by desingularizing the surface $\phi_{f^{\ast}L}(X)=\phi_L(Y)$. 
For $Y={\F}_1$, $\phi_{L}\colon{\F}_1\to{\proj}^2$ is the blow-down of the $(-1)$-curve $\Sigma$, 
and one may recover $f$ from $f^{\ast}L$ if one could identify the point $\phi_L(\Sigma)$ in ${\proj}^2$.

In the rest of this subsection 
we prepare some auxiliary results under the following assumption 
(which is necessary if one wants to obtain general members of ${\moduli}$).  

\begin{condition}\label{genericity assumption}
The singularities of $B$ are only of type $A_1, D_{2n}, E_7, E_8$.
\end{condition} 

For such a singularity $p$ of $B$, 
the irreducible curves on $Y'$ contracted to $p$ are only $(-4)$-curves and $(-1)$-curves transverse to $B'$.
The reduced preimages of those curves in $X$ are $\iota$-invariant $(-2)$-curves, 
and their dual graph is the Dynkin graph of same type as the singularity of $p$. 
We denote by $\Lambda_p \subset L_+(X, \iota)$ the generated root lattice. 
Let $B=\sum_{i=1}^{l}B_i$ be the irreducible decomposition of $B$, 
and let $F_i$ be the component of $X^{\iota}$ with $f(F_i)=B_i$. 

\begin{lemma}\label{L_+ and A-D-E}
The lattice $L_+(X, \iota)$ is generated by the sublattice $f^{\ast}NS_Y\oplus(\oplus_p \Lambda_p)$, 
where $p$ are the singularities of $B$, and the classes of $F_i$, $1\leq i\leq l$. 
\end{lemma}

\begin{proof}
Let $f'\colon X\to Y'$ be the quotient map by $\iota$. 
By the construction of $\Lambda_{p}$, 
the lattice $(f')^{\ast}NS_{Y'}$ is  contained in $f^{\ast}NS_Y\oplus(\oplus_p \Lambda_p)$. 
Also the components of $X^{\iota}$ other than $F_1,\cdots, F_l$ are contracted by $f$ to the triple points of $B$, 
so that their classes are contained in $\oplus_p \Lambda_p$. 
\end{proof}


We shall construct an ample divisor class on $X$ using the objects in Lemma \ref{L_+ and A-D-E}. 
Let $p$ be a triple point of $B$. 
Choose an identification of the dual graph of the exceptional curves over $p$ 
with an abstract graph presented in Figure \ref{dual graph}. 
Via the labeling given for the latter, 
we denote by $\{ E_{p,i} \}_i$ the $(-2)$-curves generating $\Lambda_p$.
Then we define a divisor $D_p$ on $X$ by 
$D_p=E_{p,1}+E_{p,2}+\sum_{i=4}^{2n}10^{i-4}E_{p,i}$ when $p$ is $D_{2n}$-type; 
$D_p=\sum_{i=3}^{5}10^{i-3}E_{p,i}+E_{p,6}$ when $p$ is $E_7$-type; 
and $D_p=\sum_{i=2}^{6}10^{i-2}E_{p,i}+E_{p,7}$ when $p$ is $E_8$-type. 
This divisor is independent of the choice of an identification of the graphs. 

\begin{lemma}\label{ample class}
For an arbitrary ample class $H\in NS_Y$ 
the divisor class 
\begin{equation}\label{eqn: ample class}
10^{30}f^{\ast}H + 10^{20}\sum_{i=1}^{l}F_i + \sum_pD_p, 
\end{equation}
where $p$ are the triple points of $B$, is an $\iota$-invariant ample class on $X$. 
\end{lemma}

\begin{proof}
The class $(\ref{eqn: ample class})$ is the pullback of a divisor class $L$ on $Y'$. 
Note the bounds $\sum_{p}{\rm rk}(\Lambda_p)\leq20$ and $l\leq10$ and  
apply the Nakai criterion to $L$. 
\end{proof}

By Lemmas \ref{L_+ and A-D-E} and \ref{ample class} 
we have a basis and a polarization of the lattice $L_+(X, \iota)$ defined explicitly in terms of $(Y, B)$.


\subsection{Degree of period map}\label{ssec: recipe}

Let $Y$ be one of the following rational surfaces: 
\begin{equation*}
{\proj}^2, \qquad {\proj}^1\times{\proj}^1, \qquad {\F}_n \: \: \: (1\leq n\leq4). 
\end{equation*}
Suppose that we are given an irreducible, ${\aut}(Y)$-invariant locus  
\begin{equation*}
U\subset |\!-\!2K_Y|
\end{equation*} 
such that 
$({\rm i})$ every member $B_u\in U$ satisfies Condition \ref{genericity assumption},  
$({\rm ii})$ the singularities of $B_u$ of each type form an etale cover of $U$, and  
$({\rm iii})$ the number of components of $B_u$ is constant. 
Then the 2-elementary $K3$ surfaces associated to the DPN pairs $(Y, B_u)$ have constant main invariant $(r, a, \delta)$, 
and we obtain a period map $p\colon U\to {\moduli}$. 
Since $p$ is ${\aut}(Y)$-invariant, it descends to a rational map 
\begin{equation*}
\mathcal{P}:U/{\aut}(Y)\dashrightarrow {\moduli}. 
\end{equation*}
Here $U/{\aut}(Y)$ stands for a rational quotient (see \S \ref{sec:invariant theory}). 
In this section we try to explain how to calculate the degree of $\mathcal{P}$. 
Such calculations have been done in some cases (e.g., \cite{M-S}, \cite{Ko}, \cite{A-K}), 
and the method to be explained is more or less a systemization of them. 
Roughly speaking, ${\rm deg}(\mathcal{P})$ expresses for a general $(X, \iota)\in{\moduli}$ 
how many contractions $X/\iota \to Y$ exist by which the branch of $X \to X/\iota$ is mapped to a member of $U$. 

Below we exhibit a recipe of calculation supplemented by few typical examples. 
The recipe will be applied in the rest of the article to about fifty main invariants $(r, a, \delta)$. 
Since the materials are diverse, it seems hard to formulate those calculations into some general proposition. 
Instead, we shall give sufficient instruction for each $(r, a, \delta)$ 
and then leave the detail to the reader by referring to the recipe below, 
which we believe is not difficult to master.

We use a certain cover of ${\moduli}$. 
Let $L_-$ be the 2-elementary lattice of signature $(2, 20-r)$ used in the definition of ${\moduli}$, and let 
$\widetilde{{\Or}}(L_-)$ be the group of isometries of $L_-$ which act trivially on the discriminant group $D_{L_-}$. 
Let ${\cover}$ be the arithmetic quotient 
$\mathcal{F}(\widetilde{{\Or}}(L_-)) = \widetilde{{\Or}}(L_-)\backslash\Omega_{L_-}$,  
which is irreducible for $\widetilde{{\Or}}(L_-)$ has an element exchanging the two components of $\Omega_{L_-}$. 
The natural projection 
\begin{equation*}
\pi:{\cover}\dashrightarrow{\moduli}
\end{equation*} 
is a Galois covering. 
The Galois group is the orthogonal group ${\Or}(D_{L_-})$ of the discriminant form,
for the homomorphism ${\Or}(L_-)\to{\Or}(D_{L_-})$ is surjective (\cite{Ni1}) 
and $-1\in{\Or}(L_-)$ acts trivially on $D_{L_-}$. 
In particular, we have 
\begin{equation*}
{\rm deg}(\pi) = |{\Or}(D_{L_-})|. 
\end{equation*}
Since $D_{L_-}$ is 2-elementary, 
one may calculate $|{\Or}(D_{L_-})|$ by using \cite{M-S} Corollary 2.4 and Lemma 2.5. 
We shall use standard notation for orthogonal/symplectic groups in characteristic $2$: 
${\Or}^+(2n, 2)$, ${\Or}^-(2n, 2)$, and ${\rm Sp}(2n, 2)$ (see \cite{M-S}, \cite{D-O}). 

The cover ${\cover}$ is birationally a moduli of 2-elementary $K3$ surfaces with lattice-marking. 
We fix a primitive embedding $L_-\subset\Lambda_{K3}$ where $\Lambda_{K3}=U^3\oplus E_8^2$,  
an even hyperbolic 2-elementary lattice $L_+$ of main invariant $(r, a, \delta)$, 
and an isometry $L_+\simeq(L_-)^{\perp}\cap\Lambda_{K3}$. 
Suppose that we have a 2-elementary $K3$ surface $(X, \iota)\in{\moduli}$ 
and a lattice isometry $j\colon L_+\to L_+(X, \iota)$. 
By Nikulin's theory \cite{Ni1}, $j$ can be extended to an isometry $\Phi\colon\Lambda_{K3}\to H^2(X, {\Z})$. 
Then $\Phi^{-1}(H^{2,0}(X))$ belongs to $\Omega_{L_-}$, 
and we define the period of $((X, \iota), j)$ by $[\Phi^{-1}(H^{2,0}(X))]\in{\cover}$. 
Since the restriction $\Phi|_{L_+}$ is fixed, this definition does not depend on the choice of $\Phi$.  
Two such objects $((X, \iota), j)$ and $((X', \iota'), j')$ have the same period in ${\cover}$ if and only if 
there exists a Hodge isometry $\Psi\colon H^2(X, {\Z})\to H^2(X', {\Z})$ with $\Psi\circ j=j'$. 
The open set of ${\cover}$ over ${\moduli}$ parametrizes such periods of lattice-marked 2-elementary $K3$ surfaces. 

In order to calculate ${\deg}(\mathcal{P})$, 
we construct a generically injective lift of $\mathcal{P}$, 
\begin{equation*}
{\lift}: \widetilde{U}/G \dashrightarrow {\cover}, 
\end{equation*}
where $\widetilde{U}$ is a certain cover of $U$ and $G$ is the identity component of ${\aut}(Y)$. 
Then we compare the degree of the projection $\widetilde{U}/G\to U/{\aut}(Y)$ with $|{\Or}(D_{L_-})|$. 
Here is a more precise procedure. 

\begin{enumerate}
  \item  We define a cover $\widetilde{U}\to U$ in which $\widetilde{U}$ parametrizes pairs $(B_u, \mu)$ 
             where $B_u\in U$ and $\mu$ is  a ``reasonable" labeling of 
             the singularities, the branches at the $D_{2n}$-singularities, and the components of $B_u$. 
  \item  Let $(X, \iota)=p(B_u)$ be the 2-elementary $K3$ surface associated to $(Y, B_u)$. 
            By Lemma \ref{L_+ and A-D-E}, the labeling $\mu$ and a natural basis of $NS_Y$ 
            induce a lattice marking $j$ of $L_+(X, \iota)$. 
            Actually, Lemma \ref{L_+ and A-D-E} implies an appropriate definition of the reference lattice $L_+$, 
            and then $j$ should be obtained naturally. 
            Considering the period of $((X, \iota), j)$ as defined above, 
            we obtain a lift $\tilde{p}\colon\widetilde{U}\to {\cover}$ of $p$. 
            By Borel's extension theorem (\cite{Bo}), $\tilde{p}$ is a morphism of varieties.  
  \item  One observes that the group $G$ acts on $\widetilde{U}$ and that $\tilde{p}$ is $G$-invariant. 
            Hence $\tilde{p}$ induces a rational map ${\lift}:\widetilde{U}/G\dashrightarrow{\cover}$, 
            which is a lift of $\mathcal{P}$. 
  \item  We prove that ${\lift}$ is generically injective. 
            For that it suffices to show that the $\tilde{p}$-fibers are $G$-orbits. 
            If $\tilde{p}(B_u, \mu)=\tilde{p}(B_{u'}, \mu')$,  
            we have a Hodge isometry $\Phi\colon H^2(X', {\Z})\to H^2(X, {\Z})$ with $\Phi\circ j'=j$ 
            for the associated $((X, \iota), j)$ and $((X', \iota'), j')$. 
            By Lemma \ref{ample class}, $\Phi$ turns out to preserve the ample cones. 
            Then we obtain an isomorphism $\varphi\colon X\to X'$ with $\varphi^{\ast}=\Phi$ by the Torelli theorem. 
            Let $f\colon X\to Y$, $f'\colon X'\to Y$ be the right covering maps
            and let $L\in {\rm Pic}(Y)$ be the bundle as in Lemma \ref{covering map & LB}. 
            Since $\varphi^{\ast}(f'^{\ast}L)=\Phi(f'^{\ast}L)=f^{\ast}L$, by Lemma \ref{covering map & LB} 
            we obtain an automorphism $\psi$ of $Y$ with $\psi\circ f=f'\circ \varphi$. 
            This will imply that $\psi(B_u, \mu)=(B_{u'}, \mu')$. 
            Since $\psi$ acts trivially on $NS_Y$, we have $\psi\in G$. 
  \item  Now assume that ${\dim}(U/{\aut}(Y))=20-r$. 
            Since ${\cover}$ is irreducible, then ${\lift}$ is birational. 
            Therefore ${\rm deg}(\mathcal{P})$ is equal to 
            $|{\Or}(D_{L_-})|$ divided by the degree of $\widetilde{U}/G\to U/{\aut}(Y)$. 
            The latter may be calculated geometrically.         
\end{enumerate}

In the above recipe the construction of $\widetilde{U}$ and ${\lift}$ is left rather ambiguous. 
It could be formulated generally using monodromies for the universal curve over $U$. 
But in order to give an effective account, we find it better to describe it by typical examples. 
(When $Y={\proj}^2, {\proj}^1\times{\proj}^1$, the idea can also be found in \cite{Ma}.) 

\begin{example}\label{ex:1}
We consider curves on $Q={\proj}^1\times{\proj}^1$. 
Let $U\subset|\mathcal{O}_Q(3, 3)|\times|\mathcal{O}_Q(1, 1)|$ be the open set of pairs $(C, H)$ 
such that $C$ and $H$ are smooth and transverse to each other. 
The space $U$ parametrizes the nodal $-2K_Q$-curves $C+H$, 
and we obtain a period map $\mathcal{P}\colon U/{\aut}(Q)\dashrightarrow\mathcal{M}_{8,6,1}$. 
In this case, the cover $\widetilde{U}$ should be 
\begin{equation*}
\widetilde{U} = \{  (C, H, p_1, \cdots, p_6) \in U\times Q^6, \: \{p_i\}_{i=1}^{6}=C\cap H \}, 
\end{equation*}
which parametrizes the curves $C+H$ endowed with a labeling of its six nodes $C\cap H$. 
The projection $\widetilde{U}\to U$ is an $\frak{S}_6$-covering. 

We shall prepare the reference lattice $L_+$. 
Let $\{ e_1, \cdots, e_6 \}$ be a natural orthogonal basis of the lattice $A_1^6$ 
and let $\{ u, v\}$ be a natural basis of the lattice $U(2)$ with $(u, u)=(v, v)=0$, $(u, v)=2$.  
We define the vectors $f_1, f_2\in (U(2)\oplus A_1^6)^{\vee}$ by 
$2f_1=3(u+v)-\sum_{i=1}^{6}e_i$ and $2f_2=u+v-\sum_{i=1}^{6}e_i$. 
Then the overlattice 
$L_+=\langle U(2)\oplus A_1^6, f_1, f_2 \rangle$ 
is even and 2-elementary of main invariant $(8, 6, 1)$. 

For a $(C, H, \cdots, p_6)\in \widetilde{U}$,  
we let $(X, \iota)=\mathcal{P}(C, H)$ and $f\colon X\to Q$ be the right covering map. 
The fixed curve $X^\iota$ is decomposed as $X^{\iota}=F_1+F_2$ such that $f(F_1)=C$ and $f(F_2)=H$. 
Then we define an embedding $j\colon L_+\to L_+(X, \iota)$ of lattices by 
$u\mapsto [f^{\ast}\mathcal{O}_Q(1,0)]$, $v\mapsto [f^{\ast}\mathcal{O}_Q(0,1)]$, 
$e_i\mapsto [f^{-1}(p_i)]$, and $f_i\mapsto [F_i]$. 
By Lemma \ref{L_+ and A-D-E} we have $j(L_+)=L_+(X, \iota)$. 
This gives a lattice-marked 2-elementary $K3$ surface $((X, \iota), j)$, 
and we obtain a period map $\tilde{p}\colon\widetilde{U}\to{\cove}_{8,6,1}$. 

The morphism $\tilde{p}$ is \textit{not} ${\aut}(Q)$-invariant 
for ${\aut}(Q)$ may exchange $\mathcal{O}_Q(1,0)$ and $\mathcal{O}_Q(0,1)$. 
Rather $\tilde{p}$ is invariant under the identity component $G=({\PGL}_2)^2$ of ${\aut}(Q)$. 
We prove that the $\tilde{p}$-fibers are $G$-orbits. 
If two points $(C, H,\cdots, p_6)$ and $(C', H',\cdots, p_6')$ of $\widetilde{U}$ have the same $\tilde{p}$-period, 
we have a Hodge isometry $\Phi\colon H^2(X', {\Z})\to H^2(X, {\Z})$ with $\Phi\circ j'=j$ 
for the associated $((X, \iota), j)$ and $((X', \iota'), j')$. 
In particular, $\Phi$ maps $[(f')^{\ast}\mathcal{O}_Q(a,b)], [(f')^{-1}(p_i)], [F_j']$  
respectively to $[f^{\ast}\mathcal{O}_Q(a,b)], [f^{-1}(p_i)], [F_j]$. 
By Lemma \ref{ample class} $\Phi$ preserves the ample cones. 
Hence by the Torelli theorem we obtain an isomorphism $\varphi\colon X\to X'$ with $\varphi^{\ast}=\Phi$. 
Then we have an automorphism $\psi$ of $Q$ with $\psi\circ f=f'\circ\varphi$ by Lemma \ref{covering map & LB}. 
Considering the branch loci of $f$ and $f'$, we have $\psi(C+H)=C'+H'$. 
Also we have $\psi(p_i)=p_i'$ because $\varphi(f^{-1}(p_i))=(f')^{-1}(p_i')$.  
Since $\psi$ leaves the two rulings of $Q$ invariant, we have $\psi\in G$. 
This concludes that the $\tilde{p}$-fibers are $G$-orbits. 
In view of the equality ${\dim}(\widetilde{U}/G)=12$,  
$\tilde{p}$ induces a birational lift 
${\lift}\colon\widetilde{U}/G\dashrightarrow{\cove}_{8,6,1}$ of $\mathcal{P}$. 
 
Since $L_+\simeq U\oplus A_1^6$, we have 
$|{\Or}(D_{L_+})|=2\cdot|{\rm Sp}(4, 2)|=2\cdot6!$ by \cite{M-S}. 
On the other hand, the projection $\widetilde{U}/G\to U/{\aut}(Q)$ has degree 
$[{\aut}(Q)\colon G]\cdot|\frak{S}_6|=2\cdot6!$ because ${\aut}(Q)$ acts on $U$ almost freely. 
Therefore $\mathcal{P}$ is birational. 
\end{example}

\begin{example}\label{ex:2}
We consider curves on ${\F}_1$. 
Keeping the notation of \S \ref{Sec:Hirze}, 
we let $U\subset |L_{3,2}|\times|L_{0,1}|$ be the locus of pairs $(C, F)$ such that 
$C$ is smooth, transverse to $F$ and $\Sigma$ respectively, and passes through the point $p=F\cap\Sigma$. 
The locus $U$ parametrizes the $-2K_{{\F}_1}$-curves $C+F+\Sigma$, 
whose singularities are the $D_4$-point $p$ and the three nodes $C\cap(F+\Sigma)\backslash p$. 
Hence we obtain a period map $\mathcal{P}\colon U/{\aut}({\F}_1)\dashrightarrow \mathcal{M}_{9,3,1}$. 
In this case, $\widetilde{U}$ should be the double cover  
\begin{equation*}
\widetilde{U} = \{ \: (C, F, p_1, p_2)\in U\times ({\F}_1)^2, \: \: \{p_1, p_2\} = C\cap F\backslash\Sigma \: \} . 
\end{equation*}
In $\widetilde{U}$ only the two nodes $C\cap F\backslash p$ are labelled, 
and the rest two singularities are left unmarked. 
However, the node $C\cap\Sigma\backslash p$ is distinguished from $C\cap F\backslash p$  
by the irreducible decomposition of $C+F+\Sigma$, 
and the $D_4$-point $p$ is evidently distinguished from the nodes. 
Thus $\widetilde{U}$ actually parametrizes the curves $C+F+\Sigma$ endowed with 
a complete and reasonable labeling of the singularities. 
Also the components of $C+F+\Sigma$ are distinguished by their classes in $NS_{{\F}_1}$, 
and this distinguishes the three branches of $C+F+\Sigma$ at $p$. 

Let us prepare the lattice $L_+$. 
Let $\{ h, e\}$ and $\{ e_1, e_2, e_3\}$ be natural basis of 
the lattices $\langle 2\rangle\oplus A_1$ and $A_1^3$ respectively. 
We denote the root basis of the $D_4$-lattice by $\{ f_4, e_5, e_6, e_7\}$ where 
$(f_4, e_i)=1$ and $(e_i, e_j)=-2\delta_{ij}$. 
We put $e_4=2f_4+\sum_{i=5}^{7}e_i$. 
Then we define the vectors $f_i\in(\langle 2\rangle\oplus A_1^4\oplus D_4)^{\vee}$ by 
$2f_1=3h+2(h-e)-\sum_{i=1}^{5}e_i$, 
$2f_2=h-e-(e_1+e_2+e_4+e_6)$, and 
$2f_3=e-(e_3+e_4+e_7)$. 
The overlattice 
$L_+ = \langle \langle 2\rangle\oplus A_1^4\oplus D_4, \{ f_i\}_{i=1}^3 \rangle$ 
is even and 2-elementary of main invariant $(9, 3, 1)$. 

For $(C, F, p_1, p_2)\in \widetilde{U}$, 
let $(X, \iota)=\mathcal{P}(C, F)$ and $f\colon X\to{\F}_1$ be the right covering map. 
We denote $p=F\cap\Sigma$ as above. 
On $X$ we define line bundles and $(-2)$-curves by 
$H=f^{\ast}L_{1,0}$, $E=f^{\ast}L_{1,-1}$, 
$E_i=f^{-1}(p_i)$ for $i\leq2$, and $E_3=f^{-1}(C\cap\Sigma\backslash p)$. 
By \S \ref{ssec: right resol}, 
the four $(-2)$-curves on $X$ contracted by $f$ to $p$ form a Dynkin graph of type $D_4$. 
We denote by $E_5, E_6, E_7$ those over the infinitely near points of $p$ given by $C, F, \Sigma$ respectively. 
The remaining one, denoted by $F_4$, is a component of $X^{\iota}$. 
We have an embedding $i\colon \langle 2\rangle\oplus A_1^4\oplus D_4 \to L_+(X, \iota)$  
by $h\mapsto[H], e\to[E], e_i\mapsto[E_i]$, and $f_4\mapsto [F_4]$. 
The fixed curve $X^{\iota}$ is decomposed as $X^{\iota}=\sum_{j=1}^{4}F_j$ such that 
$f(F_1)=C$, $f(F_2)=F$, and $f(F_3)=\Sigma$. 
Then $i$ extends to an isometry $j\colon L_+\to L_+(X, \iota)$ by sending $f_j\mapsto [F_j]$. 
Thus we associate a lattice-marked 2-elementary $K3$ surface $((X, \iota), j)$, 
which defines a period map $\tilde{p}\colon\widetilde{U}\to{\cove}_{9,3,1}$. 
 
Since ${\aut}({\F}_1)$ acts trivially on $NS_{{\F}_1}$, 
the morphism $\tilde{p}$ is ${\aut}({\F}_1)$-invariant. 
We show that the $\tilde{p}$-fibers are ${\aut}({\F}_1)$-orbits. 
Indeed, if $\tilde{p}(C, F, p_1, p_2)=\tilde{p}(C',F', p_1', p_2')$, 
we will obtain an isomorphism $\varphi\colon X\to X'$ with $\varphi^{\ast}\circ j'=j$ 
for the associated $((X, \iota), j)$ and $((X', \iota'), j')$, 
as in the previous example.  
Let $f\colon X\to{\F}_1$ and $f'\colon X'\to{\F}_1$ be the respective right covering maps. 
We fix a contraction $\pi\colon{\F}_1\to{\proj}^2$ of $\Sigma$. 
Since $\varphi^{\ast}(f')^{\ast}L_{1,0}\simeq f^{\ast}L_{1,0}$, 
by Lemma \ref{covering map & LB} we obtain an automorphism $\psi$ of ${\proj}^2$ 
such that $\psi\circ\pi\circ f = \pi\circ f'\circ\varphi$. 
The point is that $\psi$ fixes $\pi(\Sigma)$, 
which is the unique $D_6$-singularity of the branch curves of both $\pi\circ f$ and $\pi\circ f'$. 
Therefore $\psi$ lifts to an automorphism $\bar{\psi}$ of ${\F}_1$ 
with $\bar{\psi}\circ f = f'\circ\varphi$. 
The rest of the proof is similar to the previous example. 
Since ${\dim}(U/{\aut}({\F}_1)) = 11$,  
$\tilde{p}$ descends to a birational lift 
${\lift}\colon\widetilde{U}/{\aut}({\F}_1) \dashrightarrow {\cove}_{9,3,1}$ of $\mathcal{P}$. 

The projection $\widetilde{U}/{\aut}({\F}_1) \to U/{\aut}({\F}_1)$ is a double covering. 
On the other hand, since $L_+\simeq U(2)\oplus E_7$,  we have ${\Or}(D_{L_+})\simeq\frak{S}_2$. 
Therefore $\mathcal{P}$ is birational.  
\end{example}

\begin{example}\label{ex:3}
We consider curves on ${\F}_2$. 
We keep the notation in \S \ref{Sec:Hirze}. 
Let $U\subset|L_{3,0}|\times|L_{0,2}|$ be the open locus of pairs $(C, D)$ such that 
$C$ and $D$ are smooth and transverse to each other. 
We consider the $-2K_{{\F}_2}$-curves $C+D+\Sigma$, 
which have the eight nodes $(C+\Sigma)\cap D$ and no other singularity. 
The associated 2-elementary $K3$ surfaces $(X, \iota)$ have parity $\delta=0$ 
because the strict transforms of $C-D+\Sigma$ in $X/\iota$ belong to $4NS_{X/\iota}$. 
Thus we obtain a period map $\mathcal{P}\colon U/{\aut}({\F}_2)\dashrightarrow\mathcal{M}_{10,4,0}$. 
In this case, we should label the six nodes $C\cap D$ taking into account the irreducible decomposition of $D$. 
Let $\widetilde{U}\subset U\times({\F}_2)^3\times({\F}_2)^3$ be the locus of those 
$(C, D, p_{1+}, p_{2+},\cdots, p_{3-})$ such that 
$C\cap D=\{ p_{1+}, p_{2+},\cdots, p_{3-} \}$ and that 
$\{ p_{1+}, p_{2+}, p_{3+}\}$ belong to a same component of $D$. 
Accordingly, we may distinguish the two components of $D$ as $D=D_++D_-$ such that 
$D_{\pm}$ is the component through $p_{i\pm}$. 
The rest two nodes of $C+D+\Sigma$, $\Sigma\cap D$, are automatically labelled as $p_{4\pm}=\Sigma\cap D_{\pm}$. 
In this way the nodes and components of $C+D+\Sigma$ are labelled compatibly. 
The natural projection $\widetilde{U}\to U$ is an $\frak{S}_2\ltimes(\frak{S}_3)^2$-covering. 

We prepare the reference lattice $L_+$ as follows. 
Let $\{ u, v\}$ and $\{ e_{i+}\}_{i=1}^{4}\cup\{ e_{i-}\}_{i=1}^{4}$ be natural basis of 
the lattices $U(2)$ and $A_1^8=A_1^4\oplus A_1^4$ respectively. 
We define vectors $f_1, f_2, f_+, f_- \in (U(2)\oplus A_1^8)^{\vee}$ by 
$2f_1=3(u+v)-\sum_{i=1}^{3}(e_{i+}+e_{i-})$, 
$2f_2=u-v-e_{4+}-e_{4-}$, and 
$2f_{\pm}=v-\sum_{i=1}^{4}e_{i\pm}$. 
Then the overlattice $L_+=\langle U(2)\oplus A_1^8, f_1, f_2, f_+, f_-\rangle$ is even and 
2-elementary of main invariant $(10, 4, 0)$. 

For $(C, D,\cdots, p_{3-})\in\widetilde{U}$, 
let $(X, \iota)=\mathcal{P}(C, D)$ and $f\colon X\to{\F}_2$ be the natural projection. 
The fixed curve $X^{\iota}$ is decomposed as $X^{\iota}=F_1+F_2+F_++F_-$ 
such that $f(F_1)=C$, $f(F_2)=\Sigma$, and $f(F_{\pm})=D_{\pm}$. 
The labelling $(p_{1+},\cdots, p_{3-})$ induces a lattice-marking $j\colon L_+\to L_+(X, \iota)$ by 
$a(u+v)+bv\mapsto[f^{\ast}L_{a,b}]$, 
$e_{i\pm}\mapsto[f^{-1}(p_{i\pm})]$ for $1\leq i\leq4$, 
$f_j\mapsto[F_j]$ for $j=1, 2$, and $f_{\pm}\mapsto[F_{\pm}]$. 
Thus we obtain a lift ${\lift}\colon\widetilde{U}/{\aut}({\F}_2)\to{\cove}_{10,4,0}$ of $\mathcal{P}$. 
One may recover the morphism $f$ from the class $j(u+v)$ by Lemma \ref{covering map & LB} 
(desingularize the quadratic cone), 
the points $p_{i\pm}$ from $j(e_{i\pm})$, and the curve $C+D+\Sigma$ from $f$ as the branch locus. 
As before, these imply that ${\lift}$ is generically injective. 
Since $\widetilde{U}/{\aut}({\F}_2)$ has the same dimension $10$ as ${\cove}_{10,4,0}$, 
${\lift}$ is actually birational. 

Since $L_+\simeq U\oplus D_4^2$, we have $|{\Or}(D_{L_+})|=|{\Or}^+(4, 2)|=72$ by \cite{M-S}. 
On the other hand, the projection $\widetilde{U}/{\aut}({\F}_2)\to U/{\aut}({\F}_2)$ has degree $2\cdot(3!)^2$. 
This concludes that $\mathcal{P}$ has degree $1$. 
\end{example}

\begin{example}\label{ex:4}
Let $U\subset|{\Oplane}(4)|\times|{\Oplane}(1)|^2$ be the locus of triplets $(C, L_1, L_2)$ such that 
$C$ is smooth, $L_1$ is tangent to $C$ with multiplicity $4$, 
and $L_2$ is transverse to $C$ and passes through $q=C\cap L_1$. 
The sextics $C+L_1+L_2$ have a $D_{10}$-singularity at $q$, three nodes at $C\cap L_2\backslash q$, 
and no other singularity. 
The associated 2-elementary $K3$ surfaces have invariant $(g, k)=(3, 6)$, 
and thus we obtain a period map $\mathcal{P}\colon U/{\PGL}_3\dashrightarrow\mathcal{M}_{14,2,0}$. 
In this case we label the three nodes $C\cap L_2\backslash q$ 
by considering the locus $\widetilde{U}\subset U\times({\proj}^2)^3$ of those $(C, L_1, L_2, p_1, p_2, p_3)$ 
such that $C\cap L_2\backslash q=\{ p_1, p_2, p_3\}$. 
The rest singularity $q$ of $C+L_1+L_2$ is apparently distinguished from those nodes. 
Moreover, the two lines $L_1, L_2$ are distinguished by their intersection with $C$, 
and the three branches of $C+L_1+L_2$ at $q$ are distinguished by the irreducible decomposition of $C+L_1+L_2$. 
Thus, by $\widetilde{U}$ the relevant datum for $C+L_1+L_2$ are completely labelled. 

We prepare the reference lattice $L_+$. 
Let $h$ and $\{ e_1, e_2, e_3\}$ be natural basis of the lattices $\langle2\rangle$ and $A_1^3$ respectively. 
Let $\{d_1, d_2, f_3, d_4,\cdots, f_9, d_{10}\}$ be the root basis of the $D_{10}$-lattice  
whose numbering corresponds to the one given in the graph in Figure \ref{dual graph}.  
We define $d_{2i+1}$, $1\leq i\leq4$, inductively by 
 $d_{2i+1}=2f_{2i+1}+d_{2i+2}+d_{2i}+d_{2i-1}$. 
Then $\{ d_i\}_{i=1}^{10}$ form an orthogonal basis of a sublattice of $D_{10}$ isometric to $A_1^{10}$. 
Now we define vectors $f_0, f_1, f_2\in (\langle2\rangle\oplus A_1^3\oplus D_{10})^{\vee}$ by 
$2f_0=4h-\sum_{i=0}^{4}d_{2i+1}-\sum_{i=1}^{3}e_i$, 
$2f_1=h-d_2-\sum_{i=1}^{4}d_{2i+1}$, and 
$2f_2=h-d_9-d_{10}-\sum_{i=1}^{3}e_i$. 
The overlattice 
$L_+=\langle \langle2\rangle\oplus A_1^3\oplus D_{10}, f_0, f_1, f_2\rangle$ 
is even and 2-elementary of main invariant $(14, 2, 0)$. 

For $(C, L_1,\cdots, p_3)\in\widetilde{U}$, let $(X, \iota)=\mathcal{P}(C+L_1+L_2)$ and 
$f\colon X\to{\proj}^2$ be the natural projection. 
Let $\Lambda_q$ be the sublattice of $L_+(X, \iota)$ 
generated by the $(-2)$-curves contracted by $f$ to the $D_{10}$-point $q$. 
Assigning $e_1$ and $e_2$ respectively to the branches of $C$ and $L_1$ at $q$, 
we have a canonical isometry $D_{10}\to\Lambda_q$ which maps the root basis to the classes of $(-2)$-curves. 
Then the labelling $(p_1, p_2, p_3)$ determines an embedding 
$i\colon\langle2\rangle\oplus A_1^3\oplus D_{10}\hookrightarrow L_+(X, \iota)$ 
by $h\mapsto[f^{\ast}{\Oplane}(1)]$ and $e_i\mapsto[f^{-1}(p_i)]$. 
The fixed curve $X^{\iota}$ is decomposed as $X^{\iota}=\sum_{i=0}^{2}F_i+\sum_{j=1}^{4}F_{2j+1}$ 
such that $f(F_0)=C$, $f(F_i)=L_i$ for $i=1, 2$, and $j(f_{2j+1})=[F_{2j+1}]$ for $1\leq j\leq4$. 
The assignment $f_i\mapsto[F_i]$ extends the embedding $i$ to an isometry $j\colon L_+\to L_+(X, \iota)$. 
Considering the period of $((X, \iota), j)$, we obtain a lift 
${\lift}\colon \widetilde{U}/{\PGL}_3\dashrightarrow {\cove}_{14,2,0}$ of $\mathcal{P}$. 
In this construction, one may recover the morphism $f$ from the class $j(h)$ by Lemma \ref{covering map & LB}, 
the points $p_i$ from $j(e_i)$, and the sextic $C+L_1+L_2$ as the branch curve of $f$. 
Checking that ${\dim}(\widetilde{U}/{\PGL}_3)=6$, 
we deduce that ${\lift}$ is birational as before. 
   
Since $L_+\simeq U\oplus E_8\oplus D_4$, we have ${\Or}(D_{L_+})\simeq \frak{S}_3$.  
Hence the two projections ${\cove}_{14,2,0}\to\mathcal{M}_{14,2,0}$ and 
$\widetilde{U}/{\PGL}_3\to U/{\PGL}_3$ are both $\frak{S}_3$-coverings,  
so that $\mathcal{P}$ is birational. 
\end{example}



\section{The case $k=0$}\label{sec:k=0}

In this section we prove that the spaces $\mathcal{M}_{r,r,1}$ with $2\leq r\leq9$ are rational. 
The quotient $Y=X/\iota$ of a general $(X, \iota)\in\mathcal{M}_{r,r,1}$ is 
a del Pezzo surface of degree $10-r$.  
Let $\mathcal{M}_{DP}(d)$ be the moduli of del Pezzo surfaces of degree $d$ 
(we exclude ${\proj}^1\times{\proj}^1$). 
By the correspondence between 2-elementary $K3$ surfaces and right DPN pairs, 
we have a fibration $\mathcal{M}_{r,r,1}\dashrightarrow\mathcal{M}_{DP}(10-r)$ 
whose fiber over $Y$ is birational to $|\!-\!2K_Y|/{\aut}(Y)$. 
When $r\leq5$, $Y$ has no moduli so that $\mathcal{M}_{r,r,1}\sim|\!-\!2K_Y|/{\aut}(Y)$. 
The rationality of $\mathcal{M}_{r,r,1}$ is reduced to the ${\aut}(Y)$-representation $H^0(-2K_Y)$.  
By contrast, when $r\geq6$, $\mathcal{M}_{DP}(10-r)$ has positive dimension and ${\aut}(Y)$ is a small finite group. 
We then analyze the fibration $\mathcal{M}_{r,r,1}\dashrightarrow\mathcal{M}_{DP}(10-r)$ 
or the fixed curve map $\mathcal{M}_{r,r,1}\to\mathcal{M}_{11-r}$ to deduce the rationality.


\subsection{$\mathcal{M}_{2,2,1}$ and one-nodal sextics}\label{sec:(2,2,1)}

Let $Y$ be the blow-up of ${\proj}^2$ at one point $p\in {\proj}^2$. 
By associating to a smooth curve $B\in |\!-\!2K_Y|$ the double cover of $Y$ branched over $B$, 
we have a birational map 
\begin{equation*}
\mathcal{P} : |\!-\!2K_Y|/{\aut}(Y) \dashrightarrow \mathcal{M}_{2,2,1}. 
\end{equation*}
Indeed, $\mathcal{P}$ is generically injective 
by the correspondence between 2-elementary $K3$ surfaces and right DPN pairs, 
and is dominant because $|\!-\!2K_Y|/{\aut}(Y)$ has dimension $18$. 
One may identify $|\!-\!2K_Y|$ with the linear system of plane sextics singular at $p$, 
and ${\aut}(Y)$ with the stabilizer of $p$ in ${\PGL}_3$.  

\begin{proposition}\label{(2,2,1)}
The quotient $|\!-\!2K_Y|/{\aut}(Y)$ is rational. 
Therefore $\mathcal{M}_{2,2,1}$ is rational. 
\end{proposition}

\begin{proof}
Let $\Sigma\subset Y$ be the $(-1)$-curve. 
Let $\varphi\colon|\!-\!2K_Y| \dashrightarrow |\mathcal{O}_{\Sigma}(2)|$ be 
the ${\aut}(Y)$-equivariant map defined by $B \to B|_{\Sigma}$. 
The group ${\aut}(Y)$ acts on $|\mathcal{O}_{\Sigma}(2)|$ almost transitively. 
For two disinct points $p_1, p_2 \in \Sigma$, 
the fiber $\varphi^{-1}(p_1+p_2)$ is an open set of the linear system ${\proj}V$ of 
$-2K_Y$-curves through $p_1$ and $p_2$. 
If $G\subset{\aut}(Y)$ is the stabilizer of $p_1+p_2$, 
by the slice method \ref{slice} we have $|\!-\!2K_Y|/{\aut}(Y)\sim{\proj}V/G$. 

Let $F_i\subset Y$ be the strict transform of the line passing through $p$ with tangent $p_i$. 
Set $W=|\mathcal{O}_{F_1}(3)| \times |\mathcal{O}_{F_2}(3)|$ and 
consider the $G$-equivariant map 
\begin{equation}
\psi : {\proj}V \dashrightarrow W, \quad B \mapsto (B|_{F_1}-p_1, B|_{F_2}-p_2). 
\end{equation}
The fiber $\psi^{-1}(D_1, D_2)$ over a general $(D_1, D_2)\in W$ 
is an open set of a linear subspace of ${\proj}V$. 
Since $-2K_Y$ is ${\aut}(Y)$-linearized, $G$ acts on $V$ so that 
$\psi$ is $G$-birational to the projectivization of a $G$-linearized vector bundle over an open set of $W$. 
The group $G$ acts on $W$ almost freely. 
Indeed, for a general $(D_1, D_2)\in W$ the four points $p_1+D_1$ on $F_1\simeq{\proj}^1$ 
are not projectively equivalent to the four points $p_2+D_2$ on $F_2$, 
and any nontrivial $g\in{\PGL}_2$ with $g(p_1+D_1)=p_1+D_1$ does not fix $p_1$. 
Clearly an automorphism of $Y$ acting trivially on $F_1+F_2$ must be trivial. 
Thus we may apply the no-name method to see that 
${\proj}V/G\sim{\proj}^{16}\times(W/G)$. 
Since ${\dim}(W/G)=2$, $W/G$ is rational. 
\end{proof}

Note that the natural moduli map $f\colon|\!-\!2K_Y|/{\aut}(Y) \dashrightarrow \mathcal{M}_9$ is generically injective, 
for the normalization of a one-nodal plane sextic has only one $g_6^2$, 
the restriction of $|{\Oplane}(1)|$ (see \cite{A-C-G-H} Appendix A.20). 
The composition $f\circ\mathcal{P}^{-1}$ is nothing but the fixed curve map 
$\mathcal{M}_{2,2,1}\to\mathcal{M}_9$. 
Therefore  

\begin{corollary}\label{fixed curve (2,2,1)}
The fixed curve map $\mathcal{M}_{2,2,1}\to\mathcal{M}_9$ is generically injective 
with a generic image being the locus of 
non-hyperelliptic, non-trigonal, non-bielliptic curves possessing $g_6^2$.  
\end{corollary}


\subsection{$\mathcal{M}_{3,3,1}$ and two-nodal sextics}\label{sec:(3,3,1)}

Let $Y$ be the blow-up of ${\proj}^2$ at two distinct points $p_1, p_2$. 
As in \S \ref{sec:(2,2,1)}, we have a natural birational map 
$|\!-\!2K_Y|/{\aut}(Y) \dashrightarrow \mathcal{M}_{3,3,1}$ by the double cover construction. 
One may identify $|\!-\!2K_Y|$ with the linear system of plane sextics singular at $p_1, p_2$,  
and ${\aut}(Y)$ with the stabilizer of $p_1+p_2$ in ${\PGL}_3$. 
In this form, Casnati and del Centina \cite{C-dC} proved that $|\!-\!2K_Y|/{\aut}(Y)$ is rational. 
Their proof is based on a direct calculation of an invariant field. 
Here we shall present another simple proof. 

\begin{proposition}[\cite{C-dC}]\label{(3,3,1)}
The quotient $|\!-\!2K_Y|/{\aut}(Y)$ is rational. 
Therefore $\mathcal{M}_{3,3,1}$ is rational. 
\end{proposition}

\begin{proof}
Let $E_i\subset Y$ be the $(-1)$-curve over $p_i$ 
and let $W=|\mathcal{O}_{E_1}(2)| \times |\mathcal{O}_{E_2}(2)|$.  
We consider the map 
$\varphi\colon|\!-\!2K_Y| \dashrightarrow W$, $B\mapsto(B|_{E_1}, B|_{E_2})$, 
which is ${\aut}(Y)$-equivariant. 
The identifications $E_i={\proj}(T_{p_i}{\proj}^2)$ show that ${\aut}(Y)$ acts on $W$ almost transitively. 
If $\mathbf{q}=(q_{11}+q_{12}, q_{21}+q_{22})$ is a general point of $W$,  
let $L_{ij}\subset{\proj}^2$ be the line passing through $p_i$ with tangent $q_{ij}$. 
Then the stabilizer $G\subset{\aut}(Y)$ of $\mathbf{q}$ is 
identified with the group of $g\in{\PGL}_3$ which preserve $\sum_{i,j}L_{ij}$ and $p_1+p_2$.  
In particular, $G\simeq \frak{S}_2\ltimes(\frak{S}_2)^2$. 
The fiber $\varphi^{-1}(\mathbf{q})$ is an open set of the linear system ${\proj}V$ 
of sextics through $\{ q_{ij}\}_{i,j}$. 
By the slice method we have $|\!-\!2K_Y|/{\aut}(Y)\sim{\proj}V/G$. 
 
The net 
${\proj}V_0 = \overline{p_1p_2}+\sum_{i,j}L_{ij}+|{\Oplane}(1)|$ 
is a $G$-invariant linear subspace of ${\proj}V$.  
Since $G$ is finite, we can decompose the $G$-representation $V$ as $V=V_0 \oplus V_0^{\perp}$. 
The projection ${\proj}V \dashrightarrow {\proj}V_0$ from $V_0^{\perp}$ is a $G$-linearized vector bundle. 
Since $G$ acts on ${\proj}V_0$ almost freely, 
by the no-name method we have  
${\proj}V/G \sim {\C}^{15} \times ({\proj}V_0/G)$.  
The quotient ${\proj}V_0/G$ is clearly rational. 
\end{proof}

As in Corollary \ref{fixed curve (2,2,1)}, we have the following. 

\begin{corollary}\label{fixed curve (3,3,1)}
The fixed curve map $\mathcal{M}_{3,3,1}\to\mathcal{M}_8$ is generically injective 
with a generic image being the locus of 
non-hyperelliptic, non-trigonal, non-bielliptic curves possessing $g_6^2$.  
\end{corollary}


\subsection{$\mathcal{M}_{4,4,1}$ and three-nodal sextics}\label{sec:(4,4,1)}

Let $p_1, p_2, p_3$ be three linearly independent points in ${\proj}^2$. 
The blow-up $Y$ of ${\proj}^2$ at $p_1+p_2+p_3$ is a sextic del Pezzo surface. 
As before, we have a birational equivalence $|\!-\!2K_Y|/{\aut}(Y)\sim\mathcal{M}_{4,4,1}$. 
One may identify $|\!-\!2K_Y|$ with the linear system of plane sextics singular at $p_1+p_2+p_3$, 
and ${\aut}(Y)$ with the group $\frak{S}_2\ltimes G$ where $G$ is the stabilizer of $p_1+p_2+p_3$ in ${\PGL}_3$ 
and $\frak{S}_2$ is generated by the standard Cremona transformation based at $p_1+p_2+p_3$. 
In this form, $|\!-\!2K_Y|/{\aut}(Y)$ is proved to be rational by Casnati and del Centina \cite{C-dC}.  

\begin{proposition}[\cite{C-dC}]\label{(4,4,1)}
The space $\mathcal{M}_{4,4,1}$ is rational. 
\end{proposition}

\begin{remark}\label{another proof (4,4,1)}
The proof in \cite{C-dC} is by a direct calculation of an invariant field. 
Actually, it is also possible to give a geometric proof as in the previous sections, 
but not so short. 
\end{remark}

As noted in \cite{C-dC}, the natural map $|\!-\!2K_Y|/{\aut}(Y)\dashrightarrow\mathcal{M}_7$ is generically injective. 
This is a consequence of the classical fact that a generic three-nodal plane sextic has exactly two $g^2_6$, 
$|{\Oplane}(1)|$ and its transformation by the Cremona map. 
Thus 

\begin{corollary}\label{fixed curve (4,4,1)}
The fixed curve map for $\mathcal{M}_{4,4,1}$ is generically injective 
with a generic image being the locus of 
non-hyperelliptic, non-trigonal, non-bielliptic curves possessing $g^2_6$. 
\end{corollary}


\subsection{$\mathcal{M}_{5,5,1}$ and a quintic del Pezzo surface}\label{sec:(5,5,1)}

Let $Y$ be a quintic del Pezzo surface. 
As in the previous sections, 
we have a birational map $|\!-\!2K_Y|/{\aut}(Y)\dashrightarrow\mathcal{M}_{5,5,1}$ 
by the double cover construction (see also \cite{A-K}). 
Shepherd-Barron \cite{SB2} proved that $|\!-\!2K_Y|/{\aut}(Y)$ is rational. 
Also it is classically known that the natural map 
$|\!-\!2K_Y|/{\aut}(Y)\dashrightarrow\mathcal{M}_6$ is birational (cf. \cite{SB2}). 
Therefore 

\begin{proposition}[\cite{SB2}, \cite{A-K}]\label{rational (5,5,1)}
The space $\mathcal{M}_{5,5,1}$ is rational. 
The fixed curve map $\mathcal{M}_{5,5,1}\to\mathcal{M}_6$ is birational. 
\end{proposition}



\subsection{$\mathcal{M}_{6,6,1}$ and genus five curves}\label{sec:(6,6,1)}

Let ${\proj}V=|\mathcal{O}_{{\proj}^4}(2)|$ 
and let $\mathbb{G}(1, {\proj}V)$ be the Grassmannian of pencils of quadrics in ${\proj}^4$.  
Let $\mathcal{E} \to \mathbb{G}(1, {\proj}V)$ be the universal quotient bundle. 
The fiber $\mathcal{E}_l$ over a pencil $l={\proj}W$ is the linear space $H^0(\mathcal{O}_{{\proj}^4}(2))/W$. 
A general pencil $l$ defines a smooth $(2, 2)$ complete intersection $Y_l$ in ${\proj}^4$,  
which is an anticanonical model of a quartic del Pezzo surface (cf. \cite{De}). 
One has the identification $\mathcal{E}_l=H^0(\mathcal{O}_{Y_l}(2))=H^0(-2K_{Y_l})$ 
by a general property of complete intersections. 
Hence a general point of the bundle ${\proj}\mathcal{E}$ corresponds to a pair $(Y_l, B)$ of 
a quartic del Pezzo surface $Y_l$ and a $-2K_{Y_l}$-curve $B$. 
This defines a period map ${\proj}\mathcal{E}\dashrightarrow \mathcal{M}_{6,6,1}$, 
which descends to a rational map $\mathcal{P}\colon{\proj}\mathcal{E}/{\PGL}_5\dashrightarrow\mathcal{M}_{6,6,1}$. 
Since $Y_l \subset {\proj}^4$ is an anticanonical model, we see that $\mathcal{P}$ is generically injective. 
The equality ${\dim}({\proj}\mathcal{E}/{\PGL}_5)=14$ shows that $\mathcal{P}$ is birational.  

The $-2K_{Y_l}$-curve on $Y_l$ defined by a general point of ${\proj}\mathcal{E}_l$ is 
a $(2, 2, 2)$ complete intersection in ${\proj}^4$, which is a canonical genus five curve. 
We study ${\proj}\mathcal{E}/{\PGL}_5$ from this viewpoint. 

Let $\mathbb{G}(2, {\proj}V)$ be the Grassmannian of nets of quadrics in ${\proj}^4$, 
and let $\mathcal{F} \to \mathbb{G}(2, {\proj}V)$ be the universal sub bundle. 
The fiber $\mathcal{F}_{P}$ over a net $P={\proj}U$ is the linear subspace $U$ of $V$.  
The bundle ${\proj}\mathcal{E}$ parametrizes pairs $(W, L)$ of 
a $2$-plane $W\subset V$ and a line $L\subset V/W$, 
while the bundle ${\proj}\mathcal{F}^{\vee}$ parametrizes pairs $(U, H)$ of 
a $3$-plane $U \subset V$ and a 2-plane $H \subset U$. 
These two objects canonically correspond by 
$(W, L)\mapsto(\langle W, L\rangle, W)$ and $(U, H)\mapsto(H, U/H)$. 
Thus we have a canonical ${\PGL}_5$-isomorphism ${\proj}\mathcal{E}\simeq {\proj}\mathcal{F}^{\vee}$. 

Let $P={\proj}U$ be a general point of $\mathbb{G}(2, {\proj}V)$, 
and $B\subset{\proj}^4$ be the $g=5$ curve defined by $P$. 
One may identify $V=H^0(\mathcal{O}_{{\proj}^4}(2))$ with $S^2H^0(K_B)$, 
and $\mathcal{F}_P=U$ with the kernel of the natural linear map $S^2H^0(K_B) \to H^0(2K_B)$, 
which is surjective by M. Noether. 
Let $\pi\colon\mathcal{X}_5\to\mathcal{M}_5$ be the universal curve 
(over the open locus of curves with no automorphism), 
$K_{\pi}$ be the relative canonical bundle for $\pi$, 
and $\mathcal{G}$ be the kernel of the bundle map $S^2\pi_{\ast}K_{\pi}\to\pi_{\ast}K_{\pi}^2$. 
Since $\mathbb{G}(2, {\proj}V)/{\PGL}_5$ is naturally birational to $\mathcal{M}_5$, 
the above identification gives rise to the birational equivalence 
\begin{equation}
{\proj}\mathcal{F}^{\vee}/{\PGL}_5 \sim {\proj}\mathcal{G}^{\vee} \sim \mathcal{M}_5\times{\proj}^2. 
\end{equation}
The moduli space $\mathcal{M}_5$ is rational by Katsylo \cite{Ka2}. 
Therefore 

\begin{proposition}\label{rational (6,6,1)}
The quotient ${\proj}\mathcal{F}^{\vee}/{\PGL}_5$ is rational. 
Hence $\mathcal{M}_{6,6,1}$ is rational. 
\end{proposition}

The following assertion is also obtained. 

\begin{corollary}\label{fixed curve (6,6,1)} 
The fixed curve map $\mathcal{M}_{6,6,1}\to\mathcal{M}_5$ is dominant, 
with the fiber over a general $B\in \mathcal{M}_5$ being birationally identified with ${\proj}U^{\vee}$ 
where $U$ is the kernel of the natural map $S^2H^0(K_B) \to H^0(2K_B)$. 
\end{corollary} 

The first bundle structure ${\proj}\mathcal{E}\to\mathbb{G}(1, {\proj}V)$ corresponds to the quotient surface map 
$\mathcal{M}_{6,6,1}\dashrightarrow\mathcal{M}_{DP}(4)$, $(X, \iota)\mapsto X/\iota$, 
while the second one ${\proj}\mathcal{F}^{\vee}\to\mathbb{G}(2, {\proj}V)$ corresponds to the fixed curve map. 
The latter is easier to handle with thanks to the absence of automorphism of general $g=5$ curves.


\subsection{$\mathcal{M}_{7,7,1}$ and cubic surfaces}\label{sec:(7,7,1)}

If $Y \subset {\proj}^3$ is a cubic surface, 
the restriction map $H^0(\mathcal{O}_{{\proj}^3}(2)) \to H^0(-2K_Y)$ is isomorphic. 
Hence a general point of $U=|\mathcal{O}_{{\proj}^3}(3)|\times|\mathcal{O}_{{\proj}^3}(2)|$ 
corresponds to a pair $(Y, B)$ of a smooth cubic surface $Y$ and a smooth $-2K_Y$-curve $B$, 
which gives a 2-elementary $K3$ surface of type $(7, 7, 1)$. 
This induces a period map $\mathcal{P}\colon U/{\PGL}_4\dashrightarrow\mathcal{M}_{7,7,1}$. 
Since cubic surfaces are anticanonically embedded, $\mathcal{P}$ is generically injective. 
By the equality ${\dim}(U/{\PGL}_4)=13$, we see that $\mathcal{P}$ is birational. 

\begin{proposition}\label{rational (7,7,1)}
The quotient $U/{\PGL}_4$ is rational. 
Therefore $\mathcal{M}_{7,7,1}$ is rational. 
\end{proposition}

\begin{proof}
For a general $(Y, Q)\in U$ the intersection $B=Y\cap Q$ is a canonical $g=4$ curve. 
This induces a rational map $f\colon U/{\PGL}_4\dashrightarrow\mathcal{M}_4$, 
which is identified with the fixed curve map $\mathcal{M}_{7,7,1}\to\mathcal{M}_4$.  
The $f$-fiber over a general $B\in\mathcal{M}_4$ is the linear system of cubics containing $B$, 
for the quadric containing $B$ is unique. 
Therefore, 
if $\pi\colon\mathcal{X}_4\to\mathcal{M}_4$ is the universal curve (over an open locus) 
and $\mathcal{E}$ is the kernel of the bundle map $S^3\pi_{\ast}K_{\pi}\to\pi_{\ast}K_{\pi}^3$ 
where $K_{\pi}$ is the relative canonical bundle, 
then we have the birational equivalence 
\begin{equation}
U/{\PGL}_4 \sim {\proj}\mathcal{E} \sim \mathcal{M}_4\times{\proj}^4. 
\end{equation}
The rationality of $\mathcal{M}_4$ is proved by Shepherd-Barron \cite{SB1}. 
\end{proof}

\begin{corollary}\label{fixed curve (7,7,1)}
The fixed curve map $\mathcal{M}_{7,7,1}\to\mathcal{M}_4$ is dominant, 
with the fiber over a general $B$ being birationally identified with 
the projectivization of the kernel of the natural map $S^3H^0(K_B)\to H^0(3K_B)$.  
\end{corollary}

One may also deduce the rationality of $\mathcal{M}_{7,7,1}$ by 
applying the no-name lemma to the projection $U\to|\mathcal{O}_{{\proj}^3}(3)|$ and 
resorting to the rationality of $\mathcal{M}_{DP}(3)$ (cf. \cite{Do}).


\subsection{$\mathcal{M}_{8,8,1}$ and quadric del Pezzo surfaces}\label{sec:(8,8,1)}

As before, one may use the fibration 
$\mathcal{M}_{8,8,1}\dashrightarrow\mathcal{M}_{DP}(2)$, $(X, \iota)\mapsto X/\iota$, 
to reduce the rationality of $\mathcal{M}_{8,8,1}$ to that of $\mathcal{M}_{DP}(2)$ due to Katsylo \cite{Ka3}. 
However, in order to describe the fixed curve map, we here adopt a more roundabout approach. 

Recall that if $Y$ is a quadric del Pezzo surface, 
its anticanonical map $\phi\colon Y\to{\proj}^2$ is a double covering branched over a smooth quartic $\Gamma$. 
The correspondence $Y\mapsto\Gamma$ induces a birational map 
$\mathcal{M}_{DP}(2)\dashrightarrow\mathcal{M}_3$. 
The covering transformation $i$ of $\phi$ is the \textit{Geiser involution} of $Y$. 
Its fixed curve $Y^i$ belongs to $|\!-\!2K_Y|$. 
Let $V_-\subset H^0(-2K_Y)$ be the line given by $Y^i$ and 
let $V_+=\phi^{\ast}H^0({\Oplane}(2))\subset H^0(-2K_Y)$. 
We have the $i$-decomposition $H^0(-2K_Y)=V_+\oplus V_-$ 
where $i^{\ast}|_{V_{\pm}}=\pm1$. 

\begin{lemma}\label{basics of |-2K| (8,8,1)}
Every smooth curve $B\in|\!-\!2K_Y|$ has genus $3$, 
and the map $\phi|_B\colon B\to{\proj}^2$ is a canonical map of $B$. 
In particular, $B$ is hyperelliptic if and only if $B\in{\proj}V_+$. 
\end{lemma}

\begin{proof}
The first sentence follows from the adjunction formula $-K_Y|_B\simeq K_B$ and 
the vanishings $h^0(K_Y)=h^1(K_Y)=0$. 
By the $\iota$-decomposition of $H^0(-2K_Y)$ we have $i(B)=B$ if and only if $B\in{\proj}V_{\pm}$. 
Hence $\phi|_B$ is generically injective (actually an embedding) unless $B\in{\proj}V_+$, 
which proves the last assertion. 
\end{proof}

Let $M_{\Gamma}\subset|{\Oplane}(4)|$ be the cone 
over the locus of double conics $2Q$, $Q\in|{\Oplane}(2)|$, with vertex $\Gamma\in|{\Oplane}(4)|$. 
It is the closure of the locus of smooth quartics tangent to $\Gamma$ at eight points lying on a conic. 

\begin{lemma}\label{cone structure}
The morphism $\phi_{\ast}\colon|\!-\!2K_Y|\to|{\Oplane}(4)|$ 
induces an isomorphism between $|\!-\!2K_Y|/i$ and $M_{\Gamma}$ 
which maps the pencils $\langle \phi^{\ast}Q, Y^i\rangle$, $Q\in|{\Oplane}(2)|$, 
to the pencils $\langle 2Q, \Gamma\rangle$. 
\end{lemma}

\begin{proof}
If $B\in\langle\phi^{\ast}Q, Y^i\rangle$, then $B|_{Y^i}=\phi^{\ast}Q|_{Y^i}=\phi^{-1}(Q\cap\Gamma)$ 
so that $\phi_{\ast}B$ is tangent to $\Gamma$ at $Q\cap\Gamma$. 
Hence $\phi_{\ast}(|\!-\!2K_Y|)\subset M_{\Gamma}$, 
and the equality ${\dim}\, M_{\Gamma}={\dim}\,|\!-\!2K_Y|$ proves the assertion. 
\end{proof}


Now let $U\subset|{\Oplane}(4)|\times|{\Oplane}(4)|$ be the locus of pairs $(\Gamma, \Gamma')$ such that 
$\Gamma$ and $\Gamma'$ are smooth and tangent to each other at eight points lying on a conic. 
For a $(\Gamma, \Gamma')\in U$ we take the double cover $\phi\colon Y\to{\proj}^2$ branched over $\Gamma$. 
By Lemma \ref{cone structure} we have $\phi^{\ast}\Gamma'=B+i(B)$ for a smooth $B\in|\!-\!2K_Y|$ 
where $i$ is the Geiser involution of $Y$. 
Taking the 2-elementary $K3$ surface corresponding to $(Y, B)\simeq(Y, i(B))$,  
we obtain a well-defined morphism $U\to\mathcal{M}_{8,8,1}$, 
which descends to a rational map $\mathcal{P}\colon U/{\PGL}_3\dashrightarrow\mathcal{M}_{8,8,1}$. 

\begin{proposition}\label{birat (8,8,1)}
The map $\mathcal{P}$ is birational. 
\end{proposition}

\begin{proof}
By the birational equivalence $\mathcal{M}_{DP}(2)\sim\mathcal{M}_3$ and Lemma \ref{cone structure}, 
the first projection $U\to|{\Oplane}(4)|$, $(\Gamma, \Gamma')\mapsto\Gamma$, 
induces a fibration $U/{\PGL}_3\to\mathcal{M}_{DP}(2)$ whose fiber over 
a general $Y\in\mathcal{M}_{DP}(2)$ is an open set of $|\!-\!2K_Y|/i$. 
Since ${\aut}(Y)=\langle i\rangle$ for a general $Y$, 
this shows that $\mathcal{P}$ is generically injective.  
The equality ${\dim}(U/{\PGL}_3)=12$ concludes the proof. 
\end{proof}

\begin{proposition}\label{rational (8,8,1)}
The quotient $U/{\PGL}_3$ is rational. 
Hence $\mathcal{M}_{8,8,1}$ is rational. 
\end{proposition}

\begin{proof}
We have a ${\PGL}_3$-equivariant dominant morphism 
\begin{equation*}
\varphi:U\to|{\Oplane}(4)|\times|{\Oplane}(2)|, \qquad (\Gamma, \Gamma')\mapsto(\Gamma, Q), 
\end{equation*}
where $Q$ is the unique conic with $2Q$ contained in the pencil $\langle\Gamma, \Gamma'\rangle$. 
The $\varphi$-fiber over a general $(\Gamma, Q)$ is identified with 
an open set of the pencil $\langle2Q, \Gamma\rangle$ in $|{\Oplane}(4)|$. 
Then we may use the no-name lemma \ref{no-name 2} to see that 
\begin{equation*}
U/{\PGL}_3 \sim {\proj}^1\times(|{\Oplane}(4)|\times|{\Oplane}(2)|)/{\PGL}_3. 
\end{equation*}
The quotient $(|{\Oplane}(4)|\times|{\Oplane}(2)|)/{\PGL}_3$ is rational by 
the slice method for the projection $|{\Oplane}(4)|\times|{\Oplane}(2)|\to|{\Oplane}(2)|$ 
and Katsylo's theorem \ref{Katsylo}. 
\end{proof}

By $\mathcal{P}$ we identify $\mathcal{M}_{8,8,1}$ birationally with $U/{\PGL}_3$. 
The first projection $U\to|{\Oplane}(4)|$, $(\Gamma, \Gamma')\mapsto\Gamma$, 
induces the quotient surface map 
\begin{equation*}
Q:\mathcal{M}_{8,8,1} \dashrightarrow \mathcal{M}_{DP}(2) \sim \mathcal{M}_3, 
\qquad (X, \iota)\mapsto X/\iota. 
\end{equation*}
On the other hand, the second projection $U\to|{\Oplane}(4)|$, $(\Gamma, \Gamma')\mapsto\Gamma'$, 
induces the fixed curve map $F\colon\mathcal{M}_{8,8,1}\to\mathcal{M}_3$. 
Let $J$ be the rational involution of $\mathcal{M}_{8,8,1}$ induced by the involution 
$(\Gamma, \Gamma')\mapsto(\Gamma', \Gamma)$ of $U$. 
Now we know that 

\begin{proposition}\label{fixed curve (8,8,1)}
The two fibrations $F, Q\colon\mathcal{M}_{8,8,1}\dashrightarrow\mathcal{M}_3$ 
are exchanged by the involution $J$ of $\mathcal{M}_{8,8,1}$. 
In particular, the generic fiber $F^{-1}(\Gamma')$ of $F$ is birationally identified with 
the cone $M_{\Gamma'}$ in ${\proj}(S^4H^0(K_{\Gamma'}))$. 
\end{proposition}

The fixed locus of $J$ contains the locus $\mathcal{B}\subset\mathcal{M}_{8,8,1}$ 
of pairs $(\Gamma, \Gamma)\in U$, $\Gamma\in|{\Oplane}(4)|$. 
Generically, $\mathcal{B}$ may be characterized either as the locus  
of $(1)$ the vertices of $F$-fibers $M_{\Gamma'}$, $\Gamma'\in\mathcal{M}_3$, 
of $(2)$ the vertices of $Q$-fibers $|\!-\!2K_Y|/i$, $Y\in\mathcal{M}_{DP}(2)$, 
and of $(3)$ quadruple covers of ${\proj}^2$ branched over smooth quartics. 
Via the last description, $\mathcal{B}$ admits the structure of a ball quotient 
by a result of Kond\=o \cite{Ko2}.


\subsection{$\mathcal{M}_{9,9,1}$ and del Pezzo surfaces of degree $1$}\label{sec:(9,9,1)}

Let $Y$ be a del Pezzo surface of degree $1$. 
The bi-anticanonical map $\phi_{-2K_Y}\colon Y\to{\proj}^3$ is a degree $2$ morphism onto a quadratic cone $Q$,  
which maps the base point of $|\!-\!K_Y|$ to the vertex $p_0$ of $Q$, 
and which is branched over a smooth curve $C\in |\mathcal{O}_Q(3)|$ and $p_0$. 
In view of this, we may define a map 
$|\mathcal{O}_Q(3)|\times|\mathcal{O}_Q(1)|\dashrightarrow \mathcal{M}_{9,9,1}$ as follows. 
For a general $(C, H)\in|\mathcal{O}_Q(3)|\times|\mathcal{O}_Q(1)|$ 
we take the double cover $Y\to Q$ branched over $C$ and $p_0$, 
and let $B\in|\!-\!2K_Y|$ be the pullback of $H$. 
(Equivalently, we take the double cover of the desingularization ${\F}_2$ of $Q$ branched over $C+\Sigma$, 
and then contract the $(-1)$-curve over $\Sigma$.) 
Then we associate the 2-elementary $K3$ surface corresponding to the right DPN pair $(Y, B)$. 
This construction, being ${\aut}(Q)$-invariant, defines a period map 
$\mathcal{P}\colon(|\mathcal{O}_Q(3)|\times|\mathcal{O}_Q(1)|)/{\aut}(Q)\dashrightarrow\mathcal{M}_{9,9,1}$. 
Since the double cover $Y\to Q$ is a bi-anticanonical map of $Y$, $\mathcal{P}$ is generically injective. 
Then $\mathcal{P}$ is birational because 
$(|\mathcal{O}_Q(3)|\times|\mathcal{O}_Q(1)|)/{\aut}(Q)$ has dimension $11$. 

\begin{proposition}\label{rational (9,9,1)}
The quotient $(|\mathcal{O}_Q(3)|\times|\mathcal{O}_Q(1)|)/{\aut}(Q)$ is rational. 
Therefore $\mathcal{M}_{9,9,1}$ is rational. 
\end{proposition}

\begin{proof}
We may apply the no-name lemma \ref{no-name 2} to the projection 
$|\mathcal{O}_Q(3)|\times|\mathcal{O}_Q(1)|\to|\mathcal{O}_Q(3)|$ 
to see that 
$(|\mathcal{O}_Q(3)|\times|\mathcal{O}_Q(1)|)/{\aut}(Q)\sim{\proj}^3\times(|\mathcal{O}_Q(3)|/{\aut}(Q))$. 
The quotient $|\mathcal{O}_Q(3)|/{\aut}(Q)$ is birational to $\mathcal{M}_{DP}(1)$, 
which is rational by Dolgachev \cite{Do}. 
\end{proof}


\section{The case $g\geq7$}\label{sec: g>6, k>0}

In this section we prove that the spaces ${\moduli}$ with $g\geq7$ and $k>0$ are rational. 
In \S \ref{sec: period maps g>6, k>0} we construct birational period maps 
using trigonal curves of fixed Maroni invariant. 
Then we prove the rationality for each space. 


\subsection{Period maps and Maroni loci}\label{sec: period maps g>6, k>0}

We first construct birational period maps for $\mathcal{M}_{r,r-2,\delta}$ with $2\leq r\leq5$ 
using curves on the Hirzebruch surface ${\F}_{6-r}$. 
We keep the notation of \S \ref{Sec:Hirze}. 
Let $U_r \subset |L_{3, r-2}|$ be the open set of smooth curves 
which are transverse to the $(r-6)$-curve $\Sigma$. 
For each $C\in U_r$ the curve $C+\Sigma$ belongs to $|\!-\!2K_{{\F}_{6-r}}|$ 
and has the $r-2$ nodes $C\cap\Sigma$ as the singularities. 
Considering the 2-elementary $K3$ surfaces associated to the DPN pairs $({\F}_{6-r}, C+\Sigma)$, 
we obtain a period map 
$\mathcal{P}_r\colon U_r/{\aut}({\F}_{6-r})\dashrightarrow\mathcal{M}_{r,r-2,\delta}$. 

\begin{proposition}\label{birational (r,r-2,delta), r<6} 
The map $\mathcal{P}_r$ is birational. 
\end{proposition}

\begin{proof}
Of course one may follow the recipe in \S \ref{ssec: recipe}, 
but here a more direct proof is possible. 
Indeed, by Proposition \ref{moduli of trigonal}, 
$U_r/{\aut}({\F}_{6-r})$ is naturally birational to the moduli $\mathcal{T}_{12-r,6-r}$ of 
trigonal curves of genus $12-r$ and scroll invariant $6-r$. 
Then $\mathcal{P}_r$ is generically injective because the fixed curve map for 
$\mathcal{M}_{r,r-2,\delta}$ gives the left inverse. 
By comparison of dimensions, 
$\mathcal{P}_r$ is birational. 
\end{proof}
 
\begin{corollary}\label{fixed curve (r,r-2), r<6}
The fixed curve map for $\mathcal{M}_{r,r-2,\delta}$ with $2\leq r\leq5$ is generically injective, 
with a generic image the Maroni locus $\mathcal{T}_{12-r,6-r}$ of Maroni invariant $2$. 
\end{corollary}

Next we construct a birational period map for $\mathcal{M}_{6,2,0}$ using curves on ${\F}_3$. 
Let $U\subset |L_{3,0}|\times|L_{0,1}|$ be the open set of pairs $(C, F)$ such that $C$ is smooth and transverse to $F$. 
For each $(C, F)\in U$ the curve $C+F+\Sigma$ belongs to $|\!-\!2K_{{\F}_3}|$ and 
has the four nodes $F\cap(C+\Sigma)$ as the singularities. 
By taking the right resolution of $C+F+\Sigma$, we obtain a period map 
$\mathcal{P}\colon U/{\aut}({\F}_3)\dashrightarrow\mathcal{M}_{6,2,0}$. 

\begin{proposition}\label{birat (6,2,0)}
The map $\mathcal{P}$ is birational. 
\end{proposition}

\begin{proof}
We proceed as in Example \ref{ex:2}. 
Consider the following $\frak{S}_3$-cover of $U$: 
\begin{equation*}
\widetilde{U} = \{ (C, F, p_1, p_2, p_3) \in U\times({\F}_3)^3, \;  \{ p_1, p_2, p_3\}=C\cap F \} . 
\end{equation*}
The variety $\widetilde{U}$ parametrizes the curves $C+F+\Sigma$ 
equipped with labelings of the three nodes $C\cap F$. 
The rest one node $F\cap\Sigma$ is distinguished from those three 
by the irreducible decomposition of $C+F+\Sigma$. 
Therefore we will obtain a generically injective lift 
${\lift}\colon\widetilde{U}/{\aut}({\F}_3)\dashrightarrow{\cove}_{6,2,0}$ of $\mathcal{P}$.  
Since $\widetilde{U}/{\aut}({\F}_3)$ has dimension $14$, ${\lift}$ is birational. 
The projection 
$\widetilde{U}/{\aut}({\F}_3)\dashrightarrow U/{\aut}({\F}_3)$ is an $\frak{S}_3$-covering,  
while the projection ${\cove}_{6,2,0}\dashrightarrow\mathcal{M}_{6,2,0}$ 
is an ${\Or}(D_{L_+})$-covering for the lattice $L_+=U\oplus D_4$. 
It is easy to see that ${\Or}(D_{L_+})\simeq\frak{S}_3$. 
Hence $\mathcal{P}$ is birational. 
\end{proof}

The quotient $|L_{3,0}|/{\aut}({\F}_3)$ is birationally identified with the Maroni locus $\mathcal{T}_{7,3}$, 
and the fibers of the projection 
$(|L_{3,0}|\times|L_{0,1}|)/{\aut}({\F}_3) \dashrightarrow |L_{3,0}|/{\aut}({\F}_3)$ 
are the trigonal pencils. 
Therefore  

\begin{corollary}\label{fixed curve (6,2)}
The fixed curve map for $\mathcal{M}_{6,2,0}$ gives a dominant map 
$\mathcal{M}_{6,2,0} \dashrightarrow \mathcal{T}_{7,3}$ 
whose general fibers are identified with the trigonal pencils. 
\end{corollary}

\begin{remark}
The locus $\mathcal{T}_{7,3}$ is the complement of $\mathcal{T}_{7,1}$ in the moduli of trigonal curves. 
Corollary \ref{fixed curve (6,2)} may also be obtained from Corollary \ref{fixed curve (r,r-2), r<6} for $r=5$ 
by extending the fixed curve map for $\mathcal{M}_{5,3,1}$ to a component of the discriminant divisor. 
\end{remark}

\begin{remark}
The Maroni loci $\mathcal{T}_{g,n}$ in this section may be characterized in terms of special divisors: 
a trigonal curve $C$ of genus $8\leq g\leq10$ (resp. $g=7$) has Maroni invariant $2$ (resp. $1$)  
if and only if $W^1_{g-3}(C)$ (resp. $W^1_5(C)$) is irreducible. 
This follows from Maroni's description of $W^r_d(C)$ (see \cite{M-Sc} Proposition 1). 
\end{remark} 



\subsection{$\mathcal{M}_{2,0,0}$ and Jacobian $K3$ surfaces}\label{sec:(2,2,0)}

By a \textit{Jacobian $K3$ surface} 
we mean a $K3$ surface endowed with an elliptic fibration and with its zero section. 

\begin{proposition}\label{rational (2,2,0)}
The space $\mathcal{M}_{2,0,0}$ is birational to 
the moduli space of Jacobian $K3$ surfaces. 
Therefore $\mathcal{M}_{2,0,0}$ is rational. 
\end{proposition}

\begin{proof}
In Proposition \ref{birational (r,r-2,delta), r<6} we saw that 
a general member $(X, \iota)$ of $\mathcal{M}_{2,0,0}$ is canonically 
a double cover $f\colon X\to{\F}_4$ branched over a smooth curve $C+\Sigma, C\in|L_{3,0}|$. 
The natural projection ${\F}_4\to {\proj}^1$ gives rise to an elliptic fibration $X\to{\proj}^1$, 
and the $(-2)$-curve $E=f^{-1}(\Sigma)$ is its section. 
The involution $\iota$ is the inverse map of the fibration with respect to $E$. 
In this way a general member of $\mathcal{M}_{2,0,0}$ is 
equipped with the structure of a Jacobian $K3$ surface. 
Conversely, for a Jacobian $K3$ surface $(X\to{\proj}^1, E)$ such that every singular fiber is irreducible, 
the inversion map $\iota$ of $X/{\proj}^1$ with respect to $E$ is a non-symplectic involution of $X$, 
and the quotient of $X\to{\proj}^1$ by $\iota$ is the natural projection ${\F}_n\to {\proj}^1$ of 
a Hirzebruch surface ${\F}_n$. 
As the image of $E$ in ${\F}_n$ is a $(-4)$-curve, we have $n=4$. 
Thus our first assertion is verified. 
It is known that the moduli of Jacobian $K3$ surfaces is birational to the quotient 
\begin{equation*}
(H^0({\Oline}(8))\oplus H^0({\Oline}(12)))/{\C}^{\times}\times{\SL}_2,  
\end{equation*}
via the Weierstrass forms of elliptic fibrations (for example, see \cite{Mir}). 
By Katsylo's theorem \ref{Katsylo} this quotient is rational. 
\end{proof}


\subsection{The rationality of $\mathcal{M}_{3,1,1}$}\label{sec:(3,1,1)}

By Proposition \ref{birational (r,r-2,delta), r<6} we have a birational equivalence 
$\mathcal{M}_{3,1,1}\sim|L_{3,1}|/{\aut}({\F}_3)$ for the bundle $L_{3,1}$ on ${\F}_3$. 

\begin{proposition}\label{rational (3,1,1)}
The quotient $|L_{3,1}|/{\aut}({\F}_3)$ is rational. 
Hence $\mathcal{M}_{3,1,1}$ is rational. 
\end{proposition}

\begin{proof}
We apply the slice method to the map 
$|L_{3,1}|\dashrightarrow\Sigma$, $C\mapsto C|_{\Sigma}$. 
Let $G\subset{\aut}({\F}_3)$ be the stabilizer of a point $p\in\Sigma$, 
and let ${\proj}V \subset |L_{3,1}|$ be the linear system of curves through $p$. 
Then we have $|L_{3,1}|/{\aut}({\F}_3) \sim {\proj}V/G$. 
The group $G$ is connected and solvable by Proposition \ref{stabilizer of fiber}, 
and $G$ acts linearly on $V$ by Proposition \ref{linearization Hirze}. 
Hence ${\proj}V/G$ is rational by Miyata's theorem \ref{Miyata}. 
\end{proof}


\subsection{The rationality of $\mathcal{M}_{4,2,1}$}\label{sec:(4,2,1)}

By Proposition \ref{birational (r,r-2,delta), r<6} we have a birational equivalence 
$\mathcal{M}_{4,2,1} \sim |L_{3,2}|/{\aut}({\F}_2)$ for the bundle $L_{3,2}$ on ${\F}_2$. 

\begin{proposition}\label{rational (4,2,1)} 
The quotient $|L_{3,2}|/{\aut}({\F}_2)$ is rational. 
Hence $\mathcal{M}_{4,2,1}$ is rational. 
\end{proposition}

\begin{proof}
First we define an ${\aut}({\F}_2)$-equivariant map 
$\varphi_1\colon|L_{3,2}|\dashrightarrow S^2{\F}_2$ 
to the symmetric product of ${\F}_2$ as follows. 
For a general $C\in|L_{3,2}|$, 
let $C|_{\Sigma}=p_1+p_2$ and $F_i$ be the $\pi$-fiber through $p_i$. 
We have a unique involution $\iota_{F_i}$ of $F_i$ 
which fixes $p_i$ and exchanges the two points $C|_{F_i}-p_i$. 
Then we let $q_i\in F_i$ be the fixed point of $\iota_{F_i}$ other than $p_i$, 
and set $\varphi_1(C)=q_1+q_2$. 
The $\varphi_1$-fiber over a general $q_1+q_2\in S^2{\F}_2$ is 
the \textit{linear} space ${\proj}V_1$ of curves $C$ with 
$C|_{\Sigma}=p_1+p_2$ and $\iota_i(C|_{F_i})=C|_{F_i}$, 
where $F_i$ is the $\pi$-fiber through $q_i$, $p_i$ is $F_i\cap\Sigma$, and 
$\iota_i$ is the involution of $F_i$ fixing $q_i$ and $p_i$. 
If $G_1\subset{\aut}({\F}_2)$ is the stabilizer of $q_1+q_2$, 
by the slice method we have 
\begin{equation*}
|L_{3,2}|/{\aut}({\F}_2) \sim {\proj}V_1/G_1. 
\end{equation*}

Next we consider the $G_1$-equivariant map 
\begin{equation*}
\varphi_2 : {\proj}V_1 \dashrightarrow {\proj}(T_{p_1}{\F}_2) \times {\proj}(T_{p_2}{\F}_2), 
\quad C\mapsto (T_{p_1}C, T_{p_2}C). 
\end{equation*}
The $\varphi_2$-fiber over a general $(v_1, v_2)$ is the sub linear system 
${\proj}V_2\subset {\proj}V_1$ of curves passing through $v_i$ or singular at $p_i$. 
One checks that $G_1$ acts almost transitively on ${\proj}(T_{p_1}{\F}_2)\times{\proj}(T_{p_2}{\F}_2)$. 
If $G_2\subset G_1$ is the stabilizer of $(v_1, v_2)$, 
by the slice method we have  
\begin{equation*}
{\proj}V_1/G_1 \sim {\proj}V_2/G_2. 
\end{equation*}

We analyze the $G_2$-representation $V_2$ by using the coordinate system $\{ U_i\}_{i=1}^{4}$ of ${\F}_2$ 
introduced in \S \ref{ssec:coordinate}. 
We may assume that $q_1$ (resp. $q_2$) is the origin 
$(x_3, y_3)=(0, 0)$ in $U_3$ (resp. $(x_4, y_4)=(0, 0)$ in $U_4$). 
Then $p_i$ is the origin $(x_i, y_i)=(0, 0)$ in $U_i$, $i=1, 2$. 
We may also assume that $v_i \in {\proj}(T_{p_i}{\F}_2)$ is expressed as 
$v_i = {\C}(x_i+y_i)$ by the coordinate $(x_i, y_i)$ of $U_i$ around $p_i$. 
Let $g_{\alpha, s}, h_{\beta}$, and $j$ be the automorphisms of ${\F}_2$ described in the equations  
$(\ref{$R$-action in coordinate})$, $(\ref{linearization coordinate 1})$, and $(\ref{linearization coordinate 2})$ respectively.  
We set $\rho=g_{-\sqrt{-1}, 0}\circ h_{\sqrt{-1}}$. 

\begin{lemma}
For $p_i, q_i , v_i$ as above, the stabilizer $G_2$ is given by 
\begin{equation*}
G_2 = \langle j \rangle \ltimes (\langle \rho \rangle \ltimes \{ g_{1, \lambda XY} \}_{\lambda\in{\C}}) 
        \simeq {\Z}/2{\Z} \ltimes ({\Z}/4{\Z}\ltimes{\C}). 
\end{equation*}
Here $j$ acts on ${\Z}/4{\Z}\ltimes{\C}$ by $(-1, 1)$,  
and $\rho$ acts on ${\C}$ by $-1$.         
\end{lemma}

\begin{proof}
First observe that $j, \rho$, and $g_{1, \lambda XY}$ are contained in $G_2$. 
This is clear for $j$ and $\rho$. 
By \eqref{$R$-action in coordinate}, $g_{1, \lambda XY}$ acts on $U_1$ and $U_2$ by 
$x_i\mapsto x_i$, $y_i\mapsto(1+\lambda x_iy_i)^{-1}y_i$, $i=1, 2$. 
Hence it acts on both $T_{p_i}{\F}_2$ trivially. 

Conversely, we set $G_2'=\{ g\in G_2, g(q_i)=q_i \}$. 
The quotient $G_2/G_2'$ is ${\Z}/2{\Z}$ generated by $j$. 
The $G_2'$-action on $\Sigma$ is contained in $\{ h_{\beta}|_{\Sigma} \}_{\beta\in{\C}^{\times}}$. 
Hence by the sequence $(\ref{Aut(Hir)})$ 
every element of $G_2'$ is written as $g_{\alpha, s}\circ h_{\beta}$ for some 
$g_{\alpha, s}\in R$ and $\beta\in{\C}^{\times}$. 
Since $g_{\alpha, 0}$ and $h_{\beta}$ fix $q_i$, 
we have $g_{1,s}(q_i)=q_i$ so that $s=\lambda XY$ for some $\lambda \in {\C}$. 
This implies that $g_{\alpha,0}\circ h_{\beta} \in G_2'$. 
Then $g_{\alpha,0}\circ h_{\beta}$ acts on $U_1$ by $(x_1, y_1)\mapsto(\beta x_1, \alpha^{-1}y_1)$, 
and on $U_2$ by $(x_2, y_2)\mapsto (\beta^{-1} x_2, \alpha^{-1}\beta^{2}y_2)$.   
Thus we must have $\alpha\beta=1$ and $\beta^4=1$. 
\end{proof}

Next we describe the linear system ${\proj}V_2$ by the coordinates. 
By Proposition \ref{def eqn in Hirze} the space $H^0(L_{3,2})$ is isomorphic to the vector space 
$\{ \sum_{i=0}^3f_i(x_1)y_1^i, {\rm deg}f_i \leq 2i+2 \}$ by restriction to $U_1$. 
We shall express the polynomials $f_i$ as $f_i(x)=\sum_{j=0}^{2i+2}a_{ij}x^j$. 

\begin{lemma}
The linear subspace $V_2\subset H^0(L_{3,2})$ is defined by the equations  
\begin{equation}\label{def eq (4,2)}
a_{00}=a_{02}=a_{20}=a_{26}=0, \quad a_{01}=a_{10}=a_{14}. 
\end{equation}
\end{lemma}

\begin{proof}
Let $C\in |L_{3,2}|$ be defined by $\sum_{i=0}^3f_i(x_1)y_1^i=0$. 
The condition that $p_1\in C$ (resp. $p_2\in C$) is expressed by $a_{00}=0$ (resp. $a_{02}=0$). 
The condition on the tangent $T_{p_1}C$ (resp. $T_{p_2}C$) is written as $a_{01}=a_{10}$ (resp. $a_{01}=a_{14}$). 
The restriction of $C$ to the fiber $\{ x_1=0\}$ (resp. $\{ x_2=0\}$) is given by 
$\sum_{i=0}^{3}a_{i0}y_1^i=0$ (resp. $\sum_{i=0}^{3}a_{i,2i+2}y_2^i=0$). 
Since these polynomials should be anti-symmetric, 
we have $a_{20}=0$ and $a_{26}=0$. 
\end{proof}

Using the coordinate $U_3$, we identify $V_2$ with the subspace of 
$\{ \sum_{i=0}^3f_i(x_3)y_3^{3-i} \}$ defined by \eqref{def eq (4,2)}, and calculate the $G_2$-action on it. 
The coefficients $a_{01}=a_{10}=a_{14}, a_{11}, a_{12}, \cdots$ of the polynomials 
are basis of the dual space $V_2^{\vee}$. 
The subspace 
${\C}\langle a_{01}, a_{12}\rangle \subset V_2^{\vee}$
is $G_2$-invariant. 
Indeed, we have 
\begin{eqnarray*}\label{action on (V_2/V_3)^{dual}}
g_{1, \lambda XY}  &:& (a_{01}, a_{12}) \mapsto (a_{01},     a_{12}-3\lambda a_{01}),  \\  
\rho                  &:& (a_{01}, a_{12}) \mapsto (-a_{01},        a_{12}),   \\  
j                  &:& (a_{01}, a_{12}) \mapsto (a_{01}, a_{12}). 
\end{eqnarray*}
Therefore the annihilator of ${\C}\langle a_{01}, a_{12}\rangle$, 
\begin{equation*}
V_3 = \{ F(x_3, y_3) \in V_2, a_{01}=a_{12}=0 \}, 
\end{equation*} 
is a $G_2$-invariant subspace of $V_2$. 
We want to apply the slice method to the projection 
${\proj}V_2 \dashrightarrow {\proj}(V_2/V_3)$ from $V_3$, 
which is $G_2$-equivariant. 
We have the coordinate $(a_{01}, a_{12})\colon V_2/V_3 \to {\C}^2$ for $V_2/V_3$. 
Then $a=a_{01}^{-1}a_{12}$ is an inhomogeneous coordinate of ${\proj}(V_2/V_3)$. 
By the above calculation,  
$G_2$ acts on ${\proj}(V_2/V_3)$ by  
\begin{equation*}
g_{1, \lambda XY}(a) = a-3\lambda,  \qquad
\rho(a) = -a,  \qquad
j(a) = a. 
\end{equation*}
This shows that $G_2$ acts on the affine line ${\proj}(V_2/V_3)\backslash \{ a_{01}=0\}$ transitively 
with the stabilizer of the point $p_0=\{ a_{12}=0\}$ being  
$G_3 =  \langle j \rangle \ltimes \langle \rho \rangle$. 
The fiber of the projection ${\proj}V_2-{\proj}V_3 \to {\proj}(V_2/V_3)$ over $p_0$ is  
the hyperplane $\{ a_{12}=0 \} \subset {\proj}V_2$ minus ${\proj}V_3$, 
which is naturally $G_3$-isomorphic to the $G_3$-representation 
$V_3' = {\rm Hom}({\C}x_3y_3^3, V_3)$. 
Hence by the slice method we have 
\begin{equation*}
{\proj}V_2/G_2 \sim V_3'/G_3. 
\end{equation*} 

Let $V_4\subset V_3'$ be the subspace 
\begin{equation*}
V_4=({\C}x_3y_3^3)^{\vee}\otimes {\C}\langle x_3y_3^2, x_3^3y_3^2 \rangle. 
\end{equation*}
One checks that $V_4$ is $G_3$-invariant,  
and that with respect to the given basis $G_3$ acts on $V_4$ by    
$j = \begin{pmatrix} 0 & 1 \\ 
                                1 & 0  \end{pmatrix}$ 
and                                             
$\rho = \begin{pmatrix} -\sqrt{-1} & 0 \\ 
                                                0 & \sqrt{-1}  \end{pmatrix}$.                                         
In particular, $G_3$ acts on $V_4$ effectively and so almost freely. 
Since $G_3$ is a finite group, we have a $G_3$-decomposition $V_3' = V_4 \oplus V_4^{\perp}$. 
Applying the no-name lemma \ref{no-name 1} to the projection $V_3' \to V_4$ from $V_4^{\perp}$, 
we have  
\begin{equation*}
V_3'/G_3 \sim {\C}^{14}\times (V_4/G_3). 
\end{equation*}
Since ${\dim}\, V_4=2$, the quotient $V_4/G_3$ is rational. 
This completes the proof of Proposition \ref{rational (4,2,1)}.  
\end{proof}


\subsection{The rationality of $\mathcal{M}_{5,3,1}$}\label{sec:(5,3,1)}

By Corollary \ref{fixed curve (r,r-2), r<6}, 
$\mathcal{M}_{5,3,1}$ is birational to $\mathcal{T}_{7,1}$. 
In \cite{Ma2} we proved that the moduli of trigonal curves of genus $7$ is rational. 
Therefore      

\begin{proposition}\label{rational (5,3,1)} 
The moduli space $\mathcal{M}_{5,3,1}$ is rational. 
\end{proposition}


\subsection{The rationality of $\mathcal{M}_{6,2,0}$}\label{sec:(6,2,0)}

Recall from Proposition \ref{birat (6,2,0)} that we have a birational equivalence 
$\mathcal{M}_{6,2,0} \sim (|L_{3,0}|\times|L_{0,1}|)/{\aut}({\F}_3)$. 

\begin{proposition}\label{rational (6,2,0)} 
The quotient $(|L_{3,0}|\times|L_{0,1}|)/{\aut}({\F}_3)$ is rational. 
Therefore $\mathcal{M}_{6,2,0}$ is rational. 
\end{proposition}

\begin{proof}
Applying the slice method to the projection $|L_{3,0}|\times|L_{0,1}| \to |L_{0,1}|$, 
we have 
$(|L_{3,0}|\times|L_{0,1}|)/{\aut}({\F}_3) \sim |L_{3,0}|/G$ 
where $G\subset {\aut}({\F}_3)$ is the stabilizer of a point $p\in\Sigma$. 
As in the proof of Proposition \ref{rational (3,1,1)}, 
we see that $|L_{3,0}|/G$ is rational. 
\end{proof}


\section{The case $g=6$}\label{sec:g=6} 

In this section we prove that the spaces ${\moduli}$ with $g=6$ and $k>0$ are rational. 
We will find curves of Clifford index $1$, i.e., trigonal curves and plane quintics as the main fixed curves. 


\subsection{$\mathcal{M}_{6,4,0}$ and plane quintics}\label{ssec:(6,4,0)}

Let $U\subset |{\Oplane}(5)|\times|{\Oplane}(1)|$ be the open set of pairs $(C, L)$ such that 
$C$ is smooth and transverse to $L$. 
The 2-elementary $K3$ surface associated to the sextic $C+L$ has parity $\delta=0$. 
Indeed, if $(Y, B_1+B_2)$ is the corresponding right DPN pair, 
we have $B_1-B_2\in 4NS_Y$. 
Thus we obtain a period map $\mathcal{P}\colon U/{\PGL}_3 \to \mathcal{M}_{6,4,0}$. 

\begin{proposition}\label{birational (6,4,0)}
The period map $\mathcal{P}$ is birational. 
\end{proposition}

\begin{proof}
We consider the following $\frak{S}_5$-cover of $U$. 
\begin{equation*}
\widetilde{U} = \{ (C, L, p_1, \cdots, p_5) \in U\times({\proj}^2)^5, \: C\cap L=\{ p_1, \cdots, p_5\} \}. 
\end{equation*}
By $\widetilde{U}$ the sextics $C+L$ are endowed with complete labelings of the nodes. 
Since ${\dim}(U/{\PGL}_3)=14$, 
this induces a birational lift $\widetilde{U}/{\PGL}_3\to{\cove}_{6,4,0}$ of $\mathcal{P}$  
by the recipe in \S \ref{ssec: recipe}. 
The projection $\widetilde{U}/{\PGL}_3 \to U/{\PGL}_3$ has degree $|\frak{S}_5|$ 
because ${\PGL}_3$ acts almost freely on $U$. 
On the other hand, 
${\cove}_{6,4,0}$ is an ${\Or}(D_{L_+})$-cover of $\mathcal{M}_{6,4,0}$ for the lattice $L_+=U(2)\oplus D_4$. 
By \cite{M-S} we have $|{\Or}(D_{L_+})|=|{\Or}^-(4, 2)|=5!$. 
This proves the proposition. 
\end{proof}
 
\begin{proposition}\label{rational (6,4,0)}
The quotient $(|{\Oplane}(5)|\times|{\Oplane}(1)|)/{\PGL}_3$ is rational. 
Therefore $\mathcal{M}_{6,4,0}$ is rational. 
\end{proposition}

\begin{proof}
The vector bundle $\mathcal{E}=H^0({\Oplane}(1))\times|{\Oplane}(5)|$ over $|{\Oplane}(5)|$ is 
${\SL}_3$-linearized in which the element $e^{2\pi i/3}I\in{\SL}_3$ acts by 
the scalar multiplication by $e^{4\pi i/3}$. 
On the other hand, the tautological bundle $\mathcal{L}$ over $|{\Oplane}(5)|$ is also 
${\SL}_3$-linearized where $e^{2\pi i/3}I$ acts by the multiplication by $e^{2\pi i/3}$. 
Hence $\mathcal{E}\otimes\mathcal{L}$ is ${\PGL}_3$-linearized. 
Note that ${\proj}\mathcal{E}$ is canonically isomorphic to ${\proj}(\mathcal{E}\otimes\mathcal{L})$. 
Since ${\PGL}_3$ acts almost freely on $|{\Oplane}(5)|$, 
we may apply the no-name lemma to $\mathcal{E}\otimes\mathcal{L}$ to see that 
\begin{equation*}
(|{\Oplane}(5)|\times|{\Oplane}(1)|)/{\PGL}_3 \sim 
{\proj}(\mathcal{E}\otimes\mathcal{L})/{\PGL}_3 \sim 
{\proj}^2\times(|{\Oplane}(5)|/{\PGL}_3). 
\end{equation*}
The quotient $|{\Oplane}(5)|/{\PGL}_3$ is rational by Shepherd-Barron \cite{SB1.5}. 
\end{proof}

\begin{corollary}\label{curve moduli (6,4,0)} 
The fixed curve map for $\mathcal{M}_{6,4,0}$ is a dominant map onto 
the locus of plane quintics, i.e., non-hyperelliptic curves having $g^2_5$, 
and its general fibers are identified with the $g^2_5$. 
\end{corollary}

Note that a smooth plane quintic has only one $g_5^2$.


\subsection{$\mathcal{M}_{6,4,1}$ and trigonal curves}\label{ssec:(6,4,1)}

Let $Q={\proj}^1\times{\proj}^1$  
and let $U\subset|{\sheaf}_Q(3,4)|\times|{\sheaf}_Q(1,0)|$ be the open set of pairs $(C, F)$ such that 
$C$ is smooth and transverse to $F$. 
The group $G={\PGL}_2\times{\PGL}_2$ acts on $U$. 
For a $(C, F)\in U$ the 2-elementary $K3$ surface $(X, \iota)$ associated to the bidegree $(4, 4)$ curve $C+F$ 
has invariant $(r, a)=(6, 4)$. 
In order to calculate $\delta$, 
we blow-up a point $p$ in $C\cap F$ and then contract the two ruling fibers through $p$ 
to obtain an irreducible sextic with a node or a cusp and with a $D_4$-singularity. 
By Lemma \ref{L_+ and A-D-E} this shows that 
$L_+(X, \iota)\simeq \langle2\rangle\oplus A_1\oplus D_4$, 
and so $(X, \iota)$ has $\delta=1$. 
Thus we obtain a period map $\mathcal{P} \colon U/G \to \mathcal{M}_{6,4,1}$. 

\begin{proposition}\label{birat (6,4,1)}
The period map $\mathcal{P}$ is birational. 
\end{proposition}

\begin{proof}
Let $\widetilde{U}\subset U\times Q^4$ be the locus of 
$(C, F, p_1, \cdots, p_4)$ such that $C\cap F= \{ p_1, \cdots, p_4\}$. 
As in Example \ref{ex:1}, we see that $\mathcal{P}$ lifts to a birational map 
$\widetilde{U}/G \to {\cove}_{6,4,1}$. 
Note that ${\dim}(U/G)=14$. 
The projection $\widetilde{U}/G \to U/G$ is an $\frak{S}_4$-covering, 
while ${\cove}_{6,4,1}$ is an ${\Or}(D_{L_+})$-cover of $\mathcal{M}_{6,4,1}$ 
for the lattice $L_+=U\oplus A_1^4$. 
It is easy to calculate that ${\Or}(D_{L_+})\simeq\frak{S}_4$. 
\end{proof}

\begin{proposition}\label{rational (6,4,1)}
The quotient $(|{\sheaf}_Q(3,4)|\times|{\sheaf}_Q(1,0)|)/G$ is rational. 
Therefore $\mathcal{M}_{6,4,1}$ is rational. 
\end{proposition}

\begin{proof}
We denote $V=H^0({\sheaf}_Q(3,4))$. 
The vector bundle $\mathcal{E}=H^0({\sheaf}_Q(1,0))\times{\proj}V$ over ${\proj}V$ admits 
a natural linearization by ${\SL}_2\times{\PGL}_2$, 
in which $(-1, 1)\in{\SL}_2\times{\PGL}_2$ acts by the multiplication by $-1$. 
On the other hand, the hyperplane bundle ${\sheaf}_{{\proj}V}(1)$ over ${\proj}V$ 
is also ${\SL}_2\times{\PGL}_2$-linearized where $(-1, 1)$ acts by $-1$. 
Hence the bundle $\mathcal{E}\otimes{\sheaf}_{{\proj}V}(1)$ is $G$-linearized. 
We have a canonical isomorphism 
${\proj}(\mathcal{E}\otimes{\sheaf}_{{\proj}V}(1))\simeq{\proj}\mathcal{E}$. 
The $G$-action on ${\proj}V$ is almost free, 
with the quotient ${\proj}V/G$ naturally birational to the moduli $\mathcal{T}_6$ of trigonal curves of genus $6$ 
(see Proposition \ref{moduli of trigonal}). 
Hence by the no-name lemma for $\mathcal{E}\otimes{\sheaf}_{{\proj}V}(1)$ we have 
\begin{equation*}
(|{\sheaf}_Q(1,0)|\times{\proj}V)/G \sim {\proj}^1\times \mathcal{T}_6. 
\end{equation*}
The space $\mathcal{T}_6$ is rational by Shepherd-Barron \cite{SB1}. 
\end{proof}

\begin{corollary}\label{curve moduli (6,4,1)}
The fixed curve map for $\mathcal{M}_{6,4,1}$ 
gives a dominant map $\mathcal{M}_{6,4,1}\dashrightarrow \mathcal{T}_6$,  
whose fiber over a general $C$ is identified with the pencil $|K_C-2T|$ 
where $T$ is the trigonal bundle. 
\end{corollary}  

\begin{remark}
One might also deduce Corollaries \ref{curve moduli (6,4,0)} and \ref{curve moduli (6,4,1)} 
from the birational map $\mathcal{M}_{5,5,1}\dashrightarrow\mathcal{M}_6$ in \S \ref{sec:(5,5,1)}, 
by extending it to the discriminant divisor (cf. \cite{A-K}). 
\end{remark}


\subsection{The rationality of $\mathcal{M}_{5+k,5-k,\delta}$ with $k>1$}\label{ssec:g=6,k>1}

We construct 2-elementary $K3$ surfaces using curves on ${\F}_2$. 
We keep the notation of \S \ref{Sec:Hirze}. 
For $2\leq k\leq5$ let $U_k\subset|L_{3,1}|\times|L_{0,1}|$ be the locus of pairs $(C, F)$ such that 
$({\rm i})$ $C$ is smooth, 
$({\rm ii})$ $F$ intersects with $C$ at $C\cap\Sigma$ with multiplicity $k-2$, and 
$({\rm iii})$ $|C\cap F\backslash \Sigma|=5-k$. 
For $k=2$, these conditions mean that $F$ does not pass $C\cap\Sigma$ 
and is transverse to $C$. 
In particular, $U_2$ is open in $|L_{3,1}| \times |L_{0,1}|$. 
For $k\geq3$, $U_k$ is regarded as a sublocus of $|L_{3,1}|$ 
because $F$ is uniquely determined by the point $C\cap\Sigma$. 

\begin{lemma}\label{lemma (5+k,5-k)}
The variety $U_k$ has the expected dimension $22-k$. 
\end{lemma}

\begin{proof}
Let $f(x_1, y_1)=0$ be a defining equation of a smooth $C\in|L_{3,1}|$ as in Proposition \ref{def eqn in Hirze}. 
We normalize $F=\{ x_1=0 \}$.  
Then we have $(C, F)\in U_k$ if and only if  
the cubic polynomial $f(0, y_1)$ of $y_1$ is factorized as $f(0, y_1)=y_1^{k-2}g(y_1)$ such that 
$g(0)\ne0$ and $g$ has no multiple root. 
This proves the assertion. 
\end{proof}

For a pair $(C, F)\in U_k$, the curve $B=C+F+\Sigma$ belongs to $|\!-\!2K_{{\F}_2}|$ 
and is singular at $(C\cap F) \cup (C\cap\Sigma) \cup (F\cap\Sigma)$.  
For $k=2$ those points are distinct nodes of $B$. 
For $k\geq3$ the point $C\cap\Sigma=F\cap\Sigma$ is a $D_{2k-2}$-singularity of $B$,  
and the rest $5-k$ points $C\cap F\backslash \Sigma$ are nodes. 
Hence the 2-elementary $K3$ surface associated to $B$ has invariant $(r, a)=(5+k, 5-k)$, 
and we obtain a period map 
$\mathcal{P}_k\colon U_k/{\aut}({\F}_2) \dashrightarrow \mathcal{M}_{5+k,5-k,\delta}$  
where $\delta=1$ for $k\leq4$ and $\delta=0$ for $k=5$. 

\begin{proposition}\label{birational (5+k,5-k)}
The map $\mathcal{P}_k$ is birational. 
\end{proposition}

\begin{proof}
First we treat the case $k=2$ using the recipe in \S \ref{ssec: recipe}. 
Let $\widetilde{U}_2\subset U_2\times({\F}_2)^3$ be the locus of 
$(C, F, p_1, p_2, p_3)$ such that $C\cap F=\{ p_i\}_{i=1}^{3}$. 
The space $\widetilde{U}_2$ parametrizes the $-2K_{{\F}_2}$-curves $B=C+F+\Sigma$ 
endowed with labelings of the three nodes $C\cap F$. 
The rest two nodes, $F\cap\Sigma$ and $C\cap\Sigma$, 
are distinguished by the irreducible decomposition of $B$, 
and the components of $B$ are identified by their classes in $NS_{{\F}_2}$. 
Thus we will obtain a birational lift 
$\widetilde{U}_2/{\aut}({\F}_2)\dashrightarrow{\cove}_{7,3,1}$ of $\mathcal{P}_2$.  
The projection $\widetilde{U}_2/{\aut}({\F}_2) \dashrightarrow U_2/{\aut}({\F}_2)$ is an $\frak{S}_3$-covering, 
while ${\cove}_{7,3,1}$ is an ${\Or}(D_{L_+})$-cover of $\mathcal{M}_{7,3,1}$ 
for the lattice $L_+=U\oplus A_1\oplus D_4$.  
Since ${\Or}(D_{L_+})\simeq\frak{S}_3$, the map $\mathcal{P}_2$ is birational. 
 
For $k\geq3$ a similar argument is possible, 
but we may also proceed as in the proof of Proposition \ref{birational (r,r-2,delta), r<6}: 
since $U_k$ may be regarded as a sublocus of $|L_{3,1}|$, 
by Proposition \ref{moduli of trigonal} the quotient $U_k/{\aut}({\F}_2)$ is 
naturally birational to a sublocus of the Maroni divisor $\mathcal{T}_{6,2}$. 
Considering the fixed curve map for $\mathcal{M}_{5+k,5-k,\delta}$, 
we see that $\mathcal{P}_k$ is generically injective. 
By Lemma \ref{lemma (5+k,5-k)}, $\mathcal{P}_k$ is birational. 
\end{proof}

\begin{proposition}\label{rational (5+k,5-k)}
The quotient $U_k/{\aut}({\F}_2)$ is rational. 
Therefore $\mathcal{M}_{5+k,5-k,\delta}$ with $2\leq k\leq5$ is rational. 
\end{proposition}

\begin{proof}
This is analogous to the proofs of Propositions \ref{rational (3,1,1)} and \ref{rational (6,2,0)}: 
we apply the slice method to the projection $U_k\to|L_{0,1}|$, 
whose general fiber is an open set of a linear subspace ${\proj}V_k\subset|L_{3,1}|$ 
by the proof of Lemma \ref{lemma (5+k,5-k)}. 
Note that although $L_{3,1}$ may not be ${\aut}({\F}_2)$-linearized, 
one may instead use the group ${\SL}_2\ltimes R$ in \S \ref{Sec:Hirze}: 
it acts on $L_{3,1}$ by Lemma \ref{cover linearization Hirze}, 
and the stabilizer $G\subset{\SL}_2\ltimes R$ of a point of $|L_{0,1}|$ 
is still connected and solvable by the proof of Proposition \ref{stabilizer of fiber}. 
Therefore we may use Miyata's theorem for the $G$-representation $V_k$ to see that 
\begin{equation*}
U_k/{\aut}({\F}_2) \sim U_k/{\SL}_2\ltimes R \sim {\proj}V_k/G 
\end{equation*}
is rational. 
\end{proof}

We shall study the fixed curve maps.  
For $k\geq3$ we let $\mathcal{U}_k\subset\mathcal{T}_{6,2}$ be the closure of 
the image of natural morphism $U_k\to\mathcal{T}_{6,2}$. 
In particular, $\mathcal{U}_3$ coincides to $\mathcal{T}_{6,2}$. 
By Proposition \ref{birational (5+k,5-k)} we have 

\begin{corollary}\label{fixed curve (7,3) (8,2)}
The fixed curve map for $\mathcal{M}_{7,3,1}$ is a dominant map 
$\mathcal{M}_{7,3,1}\dashrightarrow\mathcal{T}_{6,2}$ 
whose general fibers are birationally identified with the $g^1_3$. 
The fixed curve map for $\mathcal{M}_{5+k,5-k,\delta}$ with $k\geq3$ is generically injective 
with a generic image $\mathcal{U}_k$. 
\end{corollary}

In fact, it is more natural to identify the fibers of $\mathcal{M}_{7,3,1}\dashrightarrow\mathcal{T}_{6,2}$ 
with the non-free pencils $|K_C-2T|$ where $T$ is the trigonal bundle, 
whose free part is $|T|$ (cf. Corollary \ref{curve moduli (6,4,1)}). 
The loci $\mathcal{U}_k$ may be described in terms of special divisors as follows. 

\begin{proposition}\label{fixed curve locus g=6}
Let $C$ be a trigonal curve of genus $6$ with the trigonal bundle $T$.  

$(1)$ The curve $C$ has scroll invariant $2$ if and only if $W^1_4(C)$ is irreducible. 

$(2)$ When $C$ has scroll invariant $2$, we have $C\in\mathcal{U}_4$ 
if and only if $2K_C-5T$ is contained in ${\rm Sing}({\rm Sing}W^1_5(C))$. 

$(3)$ The locus $\mathcal{U}_5$ is the intersection of $\mathcal{U}_4$ with the theta-null divisor. 
\end{proposition}

\begin{proof}
$(1)$ This follows from \cite{M-Sc} Proposition 1 (or \cite{A-C-G-H} V. A-9). 
Specifically, $W^1_4(C)$ consists of the curve $T+W_1(C)$ and the point $K_C-2T$, 
and $K_C-2T$ is not contained in $T+W_1(C)$ if and only if $C$ has scroll invariant $0$. 
When $C$ is an $L_{3,1}$-curve on ${\F}_2$, 
we have $K_C-2T \sim T+p_0$ for the point $p_0=C\cap\Sigma$. 

$(2)$ The locus $\mathcal{U}_4\subset\mathcal{T}_{6,2}$ consists of those $C$ 
whose trigonal map ramifies at $p_0$. 
By \cite{M-Sc} Proposition 1 we have $W^1_5(C)=W_+\cup W_-$ 
where $W_+=T+W_2(C)$ and $W_-=K_C-W_+$ is the residual of $W_+$. 
Let $V_+=T+p_0+W_1(C)$ and $V_-=2T-W_1(C)$ be the residual of $V_+$.  
Then we see that 
\begin{equation*}
{\rm Sing}W^1_5(C) = W_+\cap W_- = V_+\cup V_-. 
\end{equation*}
If $T-p_0\sim p_1+p_2$, we have 
$V_+\cap V_-=\{ T+p_0+p_1, T+p_0+p_2\}$. 
Since $2K_C-5T\sim T+2p_0$, this proves the assertion. 

$(3)$ 
A curve $C$ in $\mathcal{T}_{6,2}$ has an effective even theta characteristic if and only if 
the residual involution on $W^1_5(C)$ has a fixed point. 
By the structure of $W^1_5(C)$ described above, 
this is exactly when $V_+\cap V_-$ is one point, i.e., $p_1=p_2$.  
When $C\in \mathcal{U}_4$, this is equivalent to the condition $T\sim3p_0$. 
\end{proof}


\section{The case $g=5$}\label{sec:g=5} 

In this section we prove that the spaces ${\moduli}$ with $g=5$ and $k>0$ are rational. 
The case $(k, \delta)=(4, 0)$ was settled by Kond\=o \cite{Ko} using trigonal curves with vanishing theta-null. 
It turns out that the main fixed curves for other $(k, \delta)$ are also trigonal. 

\subsection{The rationality of $\mathcal{M}_{7,5,1}$}\label{ssec:(7,5,1)}

We construct 2-elementary $K3$ surfaces using curves on ${\F}_1$. 
Let  $U\subset|L_{3,2}|\times|L_{1,0}|$ be the open set of pairs $(C, H)$ such that 
$C$ and $H$ are smooth and transverse to each other. 
The 2-elementary $K3$ surface associated to the $-2K_{{\F}_1}$-curve $C+H$ has invariant $(g, k)=(5, 1)$. 
Thus we obtain a period map  
$\mathcal{P}\colon U/{\aut}({\F}_1)\dashrightarrow\mathcal{M}_{7,5,1}$. 

\begin{proposition}\label{birat (7,5,1)} 
The period map $\mathcal{P}$ is birational. 
\end{proposition}

\begin{proof}
As before, we consider an $\frak{S}_5$-cover $\widetilde{U}$ of $U$ 
whose fiber over a $(C, H)\in U$ corresponds to labelings of the five nodes $C\cap H$. 
Noticing that the blow-down $\phi\colon{\F}_1\to{\proj}^2$ contracts 
the $(-1)$-curve to the unique node of $\phi(C)$,  
one may proceed as in Example \ref{ex:2} to obtain a birational lift  
$\widetilde{U}/{\aut}({\F}_1)\dashrightarrow{\cove}_{7,5,1}$ of $\mathcal{P}$. 
The projection ${\cove}_{7,5,1}\dashrightarrow \mathcal{M}_{7,5,1}$ 
is an ${\Or}(D_{L_+})$-covering for the lattice $L_+=U\oplus A_1^5$. 
By \cite{M-S} we have $|{\Or}(D_{L_+})|=|{\Or}^-(4, 2)|=5!$, so that $\mathcal{P}$ is birational. 
\end{proof}

\begin{proposition}\label{rational (7,5,1)} 
The quotient $U/{\aut}({\F}_1)$ is rational. 
Hence $\mathcal{M}_{7,5,1}$ is rational. 
\end{proposition}

\begin{proof}
By Proposition \ref{moduli of trigonal}, 
$|L_{3,2}|/{\aut}({\F}_1)$ is canonically birational to the moduli $\mathcal{T}_5$ of trigonal curves of genus $5$. 
Since $L_{1,0}$ is ${\aut}({\F}_1)$-linearized, 
we may apply the no-name lemma \ref{no-name 2} to the projection $|L_{3,2}|\times|L_{1,0}|\to|L_{3,2}|$.  
Then we have $U/{\aut}({\F}_1)\sim{\proj}^2\times\mathcal{T}_5$. 
The space $\mathcal{T}_5$ is rational by \cite{Ma2}. 
\end{proof} 

For every smooth $C\in|L_{3,2}|$, 
the restriction of $|L_{1,0}|$ to $C$ is identified with the linear system $|K_C-T|$ 
where $T$ is the trigonal bundle. 
Therefore 

\begin{corollary}\label{fixed curve (7,5,1)}
The fixed curve map for $\mathcal{M}_{7,5,1}$ gives 
a dominant map $\mathcal{M}_{7,5,1}\dashrightarrow \mathcal{T}_5$  
whose general fibers are birationally identified with the residuals of the $g^1_3$. 
\end{corollary}

\subsection{The rationality of $\mathcal{M}_{6+k,6-k,1}$ with $k>1$}\label{ssec:(6+k,6-k),k>1}

We consider curves on ${\F}_1$. 
For $2\leq k\leq5$ let $U_k\subset|L_{3,2}|\times|L_{0,1}|$ be the locus of pairs $(C, F)$ such that 
$({\rm i})$ $C$ is smooth and transverse to the $(-1)$-curve $\Sigma$, 
$({\rm ii})$ $F$ intersects with $C$ at one of $C\cap\Sigma$ with multiplicity $k-2$, and 
$({\rm iii})$ $|C\cap F\backslash\Sigma|=5-k$. 
When $k=2$, the conditions $({\rm ii})$ and $({\rm iii})$ simply mean that $F$ is transverse to $C+\Sigma$.  
In particular, $U_2$ is open in $|L_{3,2}|\times|L_{0,1}|$. 
The projection $U_k\to|L_{3,2}|$ is dominant for $k=2, 3$, and generically injective for $k=4, 5$. 
As in the proof of Lemma \ref{lemma (5+k,5-k)}, one checks that $U_k$ has the expected dimension $20-k$. 

For a $(C, F)\in U_k$, the curve $B=C+F+\Sigma$ belongs to $|-2K_{{\F}_1}|$. 
When $k=2$, $B$ has only nodes (the intersections of the components) as the sigularities. 
When $k\geq3$, $B$ has the $D_{2k-2}$-point $F\cap\Sigma$, 
the $5-k$ nodes $C\cap F\backslash\Sigma$, the one node $C\cap\Sigma\backslash F$, 
and no other singularity. 
Thus the 2-elementary $K3$ surface associated to $B$ has invariant $(r, a)=(6+k, 6-k)$. 
When $k=4$, we have parity $\delta=1$ by Lemma \ref{delta=1} $(3)$. 
Hence we obtain a period map  
$\mathcal{P}_k\colon U_k/{\aut}({\F}_1)\dashrightarrow\mathcal{M}_{6+k,6-k,1}$. 

\begin{proposition}\label{birational (6+k,6-k)} 
The map $\mathcal{P}_k$ is birational. 
\end{proposition}

\begin{proof} 
When $k=4, 5$, the singularities of the curves $B=C+F+\Sigma$ are a priori distinguished 
by their type and by the irreducible decomposition of $B$. 
Also the three branches at the $D_{2k-2}$-point are distinguished by the decomposition of $B$. 
Therefore $\mathcal{P}_k$ lifts to a birational map 
$U_k/{\aut}({\F}_1)\dashrightarrow{\cove}_{6+k,6-k,1}$ by the recipe in \S \ref{ssec: recipe}. 
One checks that ${\cove}_{6+k,6-k,1}=\mathcal{M}_{6+k,6-k,1}$ in these cases. 

The case $k=3$ is treated in Example \ref{ex:2}. 

For $k=2$, we label the three nodes $C\cap F$ and the two nodes $C\cap\Sigma$ independently. 
This is realized by an $\frak{S}_3\times\frak{S}_2$-cover $\widetilde{U}_2$ of $U_2$. 
The rest node $F\cap\Sigma$ of $B$ is distinguished from those five. 
Then we will obtain a birational lift  
$\widetilde{U}_2/{\aut}({\F}_1) \dashrightarrow {\cove}_{8,4,1}$ of $\mathcal{P}_2$. 
By \cite{M-S} we have $|{\Or}(D_{L_+})|=2\cdot|{\rm Sp}(2, 2)|=2\cdot3!$ 
for the lattice $L_+=U\oplus D_4\oplus A_1^2$. 
Therefore $\mathcal{P}_2$ is birational. 
\end{proof}

\begin{proposition}\label{rational g=5, k>1}
The quotient $U_k/{\aut}({\F}_1)$ is rational. 
Therefore $\mathcal{M}_{6+k,6-k,1}$ with $k>1$ is rational. 
\end{proposition} 

\begin{proof} 
The same proof as for Proposition \ref{rational (5+k,5-k)} works. 
\end{proof}

For completeness, we briefly explain the birational period map for $\mathcal{M}_{10,2,0}$ 
constructed by Kond\=o \cite{Ko}. 
Let $U\subset|L_{3,2}|$ be the locus of smooth curves $C$ which are tangent to $\Sigma$. 
The quotient $U/{\aut}({\F}_1)$ is identified with the theta-null divisor $\mathcal{T}_5'$ in the trigonal locus. 
For a $C\in U$ let $F$ be the $\pi$-fiber through $C\cap\Sigma$. 
The 2-elementary $K3$ surface associated to the $-2K_{{\F}_1}$-curve $C+F+\Sigma$ 
belongs to $\mathcal{M}_{10,2,0}$. 
Then the period map $U/{\aut}({\F}_1)\to\mathcal{M}_{10,2,0}$ is birational. 
Kond\=o's proof is essentially along the line of \S \ref{ssec: recipe}. 
It can also be deduced using the fixed curve map. 

We describe the fixed curve maps as rational maps 
$F_k:\mathcal{M}_{6+k,6-k,1}\dashrightarrow\mathcal{T}_5$ to the trigonal locus. 
By Proposition \ref{birational (6+k,6-k)} we have 

\begin{corollary}\label{fixed curve (6+k,6-k) ver1}
The map $F_k$ is dominant for $k=2, 3$, and is generically injective for $k=4, 5$. 
The general fibers of $F_2$ are birationally identified with the $g^1_3$, 
and the general fibers of $F_3$ are identified with the two points ${\rm Sing}W^1_4(C)$, $C\in\mathcal{T}_5$. 
\end{corollary}

\begin{proof}
The $F_3$-fiber over a general $C\in|L_{3,2}|$ is 
identified with the two points $\{ p_1, p_2\}=\Sigma\cap C$. 
Note that $K_C\sim2T+p_1+p_2$ for the trigonal bundle $T$. 
By \cite{M-Sc}, $W^1_4(C)$ consists of two residual components, 
$W_+=T+W_1(C)$ and $W_-=K_C-W_+$. 
Then ${\rm Sing}W^1_4(C)$ is the intersection $W_+\cap W_-=\{ T+p_1, T+p_2\}$. 
\end{proof} 

A generic image of $F_4$ (resp. $F_5$) is the locus where 
the trigonal map ramifies (resp. totally ramifies) at the base point of one of ${\rm Sing}W^1_4(C)$. 
We also note that the theta-null divisor $\mathcal{T}_5'$ is exactly the locus where 
${\rm Sing}W^1_4(C)$ is one point. 
Thus the double covering $F_3\colon \mathcal{M}_{9,3,1}\dashrightarrow\mathcal{T}_5$ 
is the quotient by the residuation and is ramified at $\mathcal{M}_{10,2,0}$ over $\mathcal{T}_5'$

\section{The case $g=4$}\label{sec:g=4} 

In this section we study the case $g=4$, $k>0$. 
Our constructions are related to canonical models of genus $4$ curves, i.e., 
curves on quadratic surfaces cut out by cubics. 
Except for \S \ref{ssec:(10,4,0)} we shall use the following notation: 
$Q$ is the surface ${\proj}^1\times{\proj}^1$, 
$L_{a,b}$ is the bundle ${\sheaf}_Q(a, b)$, 
and ${\aut}(Q)_0={\PGL}_2\times{\PGL}_2$ is the identity component of ${\aut}(Q)$. 
 
\subsection{The rationality of $\mathcal{M}_{8,6,1}$}\label{ssec:(8,6,1)}
 
Let $U\subset|L_{3,3}|\times|L_{1,1}|$ be the open set of pairs $(C, H)$ 
such that $C$ and $H$ are smooth and transverse to each other. 
Considering the 2-elementary $K3$ surfaces associated to the $-2K_Q$-curves $C+H$, 
we obtain a period map $\mathcal{P}\colon U/{\aut}(Q)\dashrightarrow\mathcal{M}_{8,6,1}$. 
In Example \ref{ex:1} we proved that $\mathcal{P}$ is birational. 

\begin{proposition}\label{rational (8,6)}
The quotient $U/{\aut}(Q)$ is rational. 
Hence $\mathcal{M}_{8,6,1}$ is rational. 
\end{proposition}

\begin{proof}
Recall that the quotient $|L_{3,3}|/{\aut}(Q)$ is naturally birational to 
the moduli $\mathcal{M}_4$ of genus $4$ curves. 
This is a consequence of the fact that a general canonical genus $4$ curve 
is a complete intersection of a cubic and a unique smooth quadric.   
For a smooth $C\in|L_{3,3}|$ the linear system $|L_{1,1}|$ is identified with $|K_C|$ by restriction. 
Then let $\pi\colon\mathcal{X}_4\to\mathcal{M}_4$ be the universal genus $4$ curve (over an open locus),  
and $\mathcal{E}$ be the bundle $\pi_{\ast}K_{\mathcal{X}_4/\mathcal{M}_4}$ over $\mathcal{M}_4$. 
The above remark implies that $(|L_{3,3}|\times|L_{1,1}|)/{\aut}(Q)$ is birational to ${\proj}\mathcal{E}$. 
Since $\mathcal{M}_4$ is rational (\cite{SB1}), so is ${\proj}\mathcal{E}$. 
\end{proof}

\begin{corollary}
The fixed curve map $\mathcal{M}_{8,6,1}\to\mathcal{M}_4$ is dominant 
with the general fibers birationally identified with the canonical systems. 
\end{corollary}

\subsection{The rationality of $\mathcal{M}_{9,5,1}$}\label{ssec:(9,5,1)}

For a point $p\in Q$ we denote by $D_p$ the union of the two ruling fibers meeting at $p$. 
Let $U\subset|L_{3,3}|\times Q$ be the open set of pairs $(C, p)$ such that 
$C$ is smooth and transverse to $D_p$. 
Taking the right resolution of the $-2K_Q$-curves $C+D_p$, 
we obtain a period map $\mathcal{P}\colon U/{\aut}(Q)\dashrightarrow\mathcal{M}_{9,5,1}$. 

\begin{proposition}\label{birat (9,5)}
The map $\mathcal{P}$ is birational. 
\end{proposition}

\begin{proof}
Let $\widetilde{U}$ be the space of those $(C, p, p_1,\cdots, p_6)\in U\times Q^6$ 
such that $C\cap D_p=\{p_i\}_{i=1}^6$ and that 
$p_1, p_2, p_3$ lie on the $(1, 0)$-component of $D_p$. 
The space $\widetilde{U}$ parametrizes the curves $C+D_p$ endowed with 
labelings of the six nodes $C\cap D_p$ which take into account the decomposition of $D_p$. 
The rest node of $C+D_p$, the point $p$, is clearly distinguished from those six. 
Note that ${\aut}(Q)$ does not act on $\widetilde{U}$, 
for the definition of $\widetilde{U}$ involves the distinction of the two rulings. 
Rather $\widetilde{U}$ is acted on by ${\aut}(Q)_0$. 
As in Example \ref{ex:1}, $\mathcal{P}$ lifts to 
a birational map $\widetilde{U}/{\aut}(Q)_0\dashrightarrow{\cove}_{9,5,1}$. 
The projection $\widetilde{U}/{\aut}(Q)_0\dashrightarrow U/{\aut}(Q)$ has degree $2\cdot|\frak{S}_3\times\frak{S}_3|$, 
while ${\cove}_{9,5,1}$ is an ${\Or}(D_{L_+})$-cover of $\mathcal{M}_{9,5,1}$ 
for the lattice $L_+=U\oplus D_4\oplus A_1^3$. 
By \cite{M-S} we have $|{\Or}(D_{L_+})|=|{\Or}^+(4, 2)|=72$. 
\end{proof}

\begin{proposition}\label{rational (9,5)}
The quotient $(|L_{3,3}|\times Q)/{\aut}(Q)$ is rational. 
Therefore $\mathcal{M}_{9,5,1}$ is rational. 
\end{proposition}

\begin{proof}
Applying the slice method to the projection $|L_{3,3}|\times Q \to Q$, 
we have 
\begin{equation*}
(|L_{3,3}|\times Q)/{\aut}(Q)\sim |L_{3,3}|/G
\end{equation*} 
where $G\subset{\aut}(Q)$ is the stabilizer of a point $p\in Q$. 
Let $F_1, F_2$ be respectively the $(1, 0)$- and the $(0, 1)$-fiber through $p$.  
We denote $W=|\mathcal{O}_{F_1}(3)|\times|\mathcal{O}_{F_2}(3)|$. 
We want to apply the no-name lemma to the $G$-equivariant map 
$|L_{3,3}|\dashrightarrow W$, $C\mapsto(C|_{F_1}, C|_{F_2})$. 
It is birationally the projectivization of a subbundle $\mathcal{E}$ of the vector bundle 
$H^0(L_{3,3})\times W$ over $W$. 
Let $\widetilde{G}$ be the preimage of $G\subset{\aut}(Q)$ in $\frak{S}_2\ltimes({\SL}_2)^2$. 
Since $\frak{S}_2\ltimes({\SL}_2)^2$ acts on $L_{3,3}$, the bundle $\mathcal{E}$ is $\widetilde{G}$-linearized. 
On the other hand, consider the natural line bundle $\mathcal{L}=\mathcal{O}(1)\boxtimes\mathcal{O}(1)$ 
over $W$ which is also $\widetilde{G}$-linearized. 
The kernel of $\widetilde{G}\to G$ is generated by the elements $(1, -1), (-1, 1)$ of $({\SL}_2)^2$, 
which act on both $\mathcal{E}$ and $\mathcal{L}$ by the multiplication by $-1$. 
Thus the bundle $\mathcal{E}\otimes\mathcal{L}$ is $G$-linearized. 
The group $G$ acts on $W$ almost freely: 
this follows by noticing that for a general point $(D_1, D_2)$ of $W$  
there is no isomorphism $F_1\to F_2$ mapping $(p, D_1)$ to $(p, D_2)$, 
and that ${\PGL}_2$ acts on ${\proj}^1\times|{\Oline}(3)|$ almost freely. 
Therefore we may apply the no-name lemma to $\mathcal{E}\otimes\mathcal{L}$ to deduce that 
\begin{equation*}
|L_{3,3}|/G\sim {\proj}\mathcal{E}/G \sim {\proj}(\mathcal{E}\otimes\mathcal{L})/G \sim {\proj}^9\times(W/G). 
\end{equation*}
The quotient $W/G$ is rational because ${\dim}(W/G)=2$. 
\end{proof}

\begin{corollary}\label{fixed curve (9,5)} 
The fixed curve map $\mathcal{M}_{9,5,1}\to\mathcal{M}_4$ is dominant 
with the general fibers birationally identified with the products of the two trigonal pencils.  
\end{corollary}

\subsection{$\mathcal{M}_{10,4,1}$ and the universal genus $4$ curve}\label{ssec:(10,4,1)}

As in \S \ref{ssec:(9,5,1)}, 
for a point $p\in Q$ we denote by $D_p$ the reducible bidegree $(1, 1)$ curve singular at $p$. 
Let $U\subset|L_{3,3}|\times Q$ be the locus of pairs $(C, p)$ such that 
$C$ is smooth, passes $p$, and is transverse to each component of $D_p$. 
The space $U$ is an open set of the universal bidegree $(3, 3)$ curve. 
The bidegree $(4, 4)$ curve $C+D_p$ has the $D_4$-point $p$, the four nodes $C\cap D_p\backslash p$, 
and no other singularity.  
The associated 2-elementary $K3$ surface has invariant $(g, k)=(4, 3)$. 
It has parity $\delta=1$ by Lemma \ref{delta=1} $(2)$. 
Thus we obtain a period map $\mathcal{P}\colon U/{\aut}(Q)\dashrightarrow\mathcal{M}_{10,4,1}$. 

\begin{proposition}\label{birat (10,4,1)}
The map $\mathcal{P}$ is birational. 
\end{proposition}

\begin{proof}
As in \S \ref{ssec:(9,5,1)}, 
we consider the locus $\widetilde{U}\subset U\times Q^4$ of 
those $(C, p, p_1,\cdots, p_4)$ such that $C\cap D_p\backslash p=\{p_i\}_{i=1}^4$ and 
that $p_1, p_2$ lie on the $(1, 0)$-component of $D_p$. 
Since the three branches of $C+D_p$ at $p$ are distinguished by the irreducible decomposition of $C+D_p$,  
then $\mathcal{P}$ lifts to a birational map $\widetilde{U}/{\aut}(Q)_0\dashrightarrow{\cove}_{10,4,1}$. 
The space $\widetilde{U}/{\aut}(Q)_0$ is an $\frak{S}_2\ltimes(\frak{S}_2)^2$-cover of $U/{\aut}(Q)$, 
while ${\cove}_{10,4,1}$ is an ${\Or}(D_{L_+})$-cover of $\mathcal{M}_{10,4,1}$ 
for the lattice $L_+=U(2)\oplus E_7\oplus A_1$. 
By \cite{M-S} we have $|{\Or}(D_{L_+})|=2^2\cdot|{\Or}^+(2, 2)|=8$. 
\end{proof} 

Since $|L_{3,3}|/{\aut}(Q)$ is naturally biational to $\mathcal{M}_4$, 
the quotient $U/{\aut}(Q)$ is birational to the universal genus $4$ curve, 
which is rational by Catanese \cite{Ca}. 
Therefore 

\begin{proposition}\label{rational (10,4,1)} 
The moduli space $\mathcal{M}_{10,4,1}$ is rational. 
\end{proposition}

\begin{corollary}\label{fixed curve (10,4,1)}
The fixed curve map $\mathcal{M}_{10,4,1}\to\mathcal{M}_4$ 
identifies $\mathcal{M}_{10,4,1}$ birationally with the universal genus $4$ curve.  
\end{corollary}

\subsection{$\mathcal{M}_{10,4,0}$ and genus $4$ curves with vanishing theta-null}\label{ssec:(10,4,0)}

We consider curves on ${\F}_2$, keeping the notation of \S \ref{Sec:Hirze}. 
Let $U\subset |L_{3,0}| \times |L_{0,2}|$ be the open set of pairs $(C, D)$ such that 
$C$ and $D$ are smooth and transverse to each other. 
In Example \ref{ex:3}, we showed that the 2-elementary $K3$ surfaces 
associated to the $-2K_{{\F}_2}$-curves $C+D+\Sigma$ have main invariant $(10, 4, 0)$, 
and that the induced period map 
$U/{\aut}({\F}_2)\dashrightarrow \mathcal{M}_{10,4,0}$ is birational.

The quotient $|L_{3,0}|/{\aut}({\F}_2)$ is naturally birational to 
the theta-null divisor $\mathcal{M}_4'$ in $\mathcal{M}_4$. 
Indeed, recall that the morphism $\phi_{L_{1,0}}\colon{\F}_2\to{\proj}^3$ associated to the bundle $L_{1,0}$ 
is the minimal desingularization of the quadratic cone $Q=\phi_{L_{1,0}}({\F}_2)$,  
and that the restriction of $\phi_{L_{1,0}}$ to each smooth $C\in|L_{3,0}|$ is a canonical map of $C$. 
Then our claim follows from the fact that a non-hyperelliptic genus $4$ curve 
has an effective even theta characteristic if and only if its canonical model lies on a singular quadric. 
In that case the half canonical pencil is given by the pencil of lines on $Q$, 
or equivalently, the pencil $|L_{0,1}|$. 

\begin{proposition}\label{rational (10,4,0)}
The quotient $(|L_{3,0}|\times|L_{0,2}|)/{\aut}({\F}_2)$ is rational. 
Hence $\mathcal{M}_{10,4,0}$ is rational. 
\end{proposition}

\begin{proof}
Since both $L_{3,0}$ and $L_{0,2}$ are ${\aut}({\F}_2)$-linearized (Proposition \ref{linearization Hirze}) 
and since a general $C\in\mathcal{M}_4'$ has no automorphism, 
we may apply the no-name lemma \ref{no-name 2} to the projection $|L_{3,0}|\times|L_{0,2}| \to |L_{3,0}|$ 
to see that 
\begin{equation*}
(|L_{3,0}|\times|L_{0,2}|)/{\aut}({\F}_2) \sim {\proj}^2\times(|L_{3,0}|/{\aut}({\F}_2)) \sim {\proj}^2\times\mathcal{M}_4'. 
\end{equation*}
The space $\mathcal{M}_4'$ is rational by Dolgachev \cite{Do}. 
\end{proof}

\begin{corollary}\label{fixed curve (10,4,0)}
The fixed curve map for $\mathcal{M}_{10,4,0}$ is a dominant map onto $\mathcal{M}_4'$ 
whose general fibers are birationally identified with the symmetric products of the half-canonical pencils. 
\end{corollary}

\subsection{The rationality of $\mathcal{M}_{11,3,1}$ and $\mathcal{M}_{12,2,1}$}\label{ssec:(11,3)&(12,2)}

Let $Q={\proj}^1\times{\proj}^1$. 
For a point $p\in Q$ we denote by $D_p$ the reducible bidegree $(1, 1)$ curve singular at $p$. 
For $k=4, 5$ let $U_k\subset|L_{3,3}|\times Q$ be the locus of pairs $(C, p)$ such that 
$({\rm i})$ $C$ is smooth with $p\in C$, 
$({\rm ii})$ the $(1, 0)$-component of $D_p$ is tangent to $C$ at $p$ with multiplicity $k-2$, and 
$({\rm iii})$ the $(0, 1)$-component of $D_p$ is transverse to $C$. 
The space $U_k$ is acted on by ${\aut}(Q)_0$ (but not by ${\aut}(Q)$).  
The bidegree $(4, 4)$ curve $C+D_p$ has the $D_{2k-2}$-singularity $p$, 
the $7-k$ nodes $C\cap D_p\backslash p$, and no other singularity. 
Taking the right resolution of $C+D_p$, 
we obtain a period map 
$\mathcal{P}_k\colon U_k/{\aut}(Q)_0\dashrightarrow\mathcal{M}_{7+k,7-k,1}$. 

\begin{proposition}\label{birational (11,3)&(12,2)}
The map $\mathcal{P}_k$ is birational. 
\end{proposition}

\begin{proof}
We only have to distinguish the two intersection points other than $p$ 
of $C$ and the $(0, 1)$-component of $D_p$. 
This defines a double cover $\widetilde{U}_k\to U_k$. 
The rest singularities of $C+D_p$ and the branches of $C+D_p$ at $p$ are a priori labeled as before. 
Checking ${\dim}(U_k/{\aut}(Q)_0)=13-k$, 
we see that $\mathcal{P}_k$ lifts to a birational map 
$\widetilde{U}_k/{\aut}(Q)_0 \dashrightarrow {\cove}_{7+k,7-k,1}$. 
The variety ${\cove}_{7+k,7-k,1}$ is an ${\Or}(D_{L_-})$-cover of $\mathcal{M}_{7+k,7-k,1}$ 
for the lattice $L_-=\langle2\rangle^2\oplus E_8\oplus A_1^{5-k}$. 
Then ${\Or}(D_{L_-})\simeq\frak{S}_2$ for both $k=4, 5$, 
so that $\mathcal{P}_k$ has degree $1$. 
\end{proof} 

\begin{proposition}\label{rational (11,3)&(12,2)} 
The quotient $U_k/{\aut}(Q)_0$ is rational. 
Therefore $\mathcal{M}_{11,3,1}$ and $\mathcal{M}_{12,2,1}$ are rational. 
\end{proposition}

\begin{proof}
This is a consequence of the slice method for the projection $U_k\to Q$
and Miyata's theorem. 
The stabilizer of a point $p\in Q$ in ${\aut}(Q)_0$ is isomorphic to $({\C}^{\times}\ltimes{\C})^2$, 
which is connected and solvable. 
\end{proof}

\begin{corollary}\label{fixed curve (11,3)}
The fixed curve map $\mathcal{M}_{11,3,1}\to\mathcal{M}_4$ is finite and dominant, 
with the fiber over a general $C\in\mathcal{M}_4$ being the ramification 
points of the two trigonal maps $C\to{\proj}^1$. 
\end{corollary} 

For $k=5$ the image of the natural map $U_5\to\mathcal{M}_4$ 
consists of curves such that one of its trigonal maps has a total ramification point. 
Such a point is nothing but a Weierstrass point whose first non-gap is $3$. 
Therefore  

\begin{corollary}\label{fixed curve (12,2)}
The fixed curve map for $\mathcal{M}_{12,2,1}$ is generically injective 
with a generic image being the locus of curves having a Weierstrass point 
whose first non-gap is $3$. 
\end{corollary}


\section{The case $g=3$}\label{sec:g=3} 

In this section we study the case $g=3$, $k>0$. 
When $(k, \delta)\ne(2, 0)$, we use 
plane quartics to construct general members of ${\moduli}$. 
For $(k, \delta)=(2, 0)$ we find hyperelliptic curves as the main fixed curves.

\subsection{The rationality of $\mathcal{M}_{9,7,1}$}\label{ssec:(9,7)}

Let $U\subset|{\Oplane}(4)|\times|{\Oplane}(2)|$ be the open locus of pairs $(C, Q)$ such that 
$C$ and $Q$ are smooth and transverse to each other. 
The 2-elementary $K3$ surfaces associated to the sextics $C+Q$ have invariant $(g, k)=(3, 1)$, 
and we obtain a period map 
$\mathcal{P}\colon U/{\PGL}_3\dashrightarrow\mathcal{M}_{9,7,1}$. 

\begin{proposition}\label{birat (9,7)}
The map $\mathcal{P}$ is birational. 
\end{proposition}

\begin{proof}
We consider an $\frak{S}_8$-covering $\widetilde{U}\to U$ whose fiber over a $(C, Q)\in U$ 
corresponds to labelings of the eight nodes $C\cap Q$ of $C+Q$.  
As before, we see that $\mathcal{P}$ lifts to a birational map 
$\widetilde{U}/{\PGL}_3\dashrightarrow{\cove}_{9,7,1}$. 
The projection ${\cove}_{9,7,1}\dashrightarrow\mathcal{M}_{9,7,1}$ 
is an ${\Or}(D_{L_+})$-covering for the lattice $L_+=U\oplus A_1^7$. 
By \cite{M-S} we calculate $|{\Or}(D_{L_+})|=|{\Or}^+(6, 2)|=8!$. 
\end{proof}

Recall that we have a natural birational equivalence $|{\Oplane}(4)|/{\PGL}_3 \sim \mathcal{M}_3$, 
for a canonical model of a non-hyperelliptic genus $3$ curve is a plane quartic. 

\begin{proposition}\label{rational (9,7)} 
The quotient $(|{\Oplane}(4)|\times|{\Oplane}(2)|)/{\PGL}_3$ is rational. 
Therefore $\mathcal{M}_{9,7,1}$ is rational. 
\end{proposition}

\begin{proof}
Let $\pi\colon\mathcal{X}_3\to\mathcal{M}_3$ be the universal curve (over an open set), 
$K_{\pi}$ be the relative canonical bundle for $\pi$, 
and $\mathcal{E}$ be the bundle $\pi_{\ast}K_{\pi}^2$. 
As the restriction map $|{\Oplane}(2)|\to|2K_C|$ for a smooth quartic $C$ is isomorphic, 
the quotient $(|{\Oplane}(4)|\times|{\Oplane}(2)|)/{\PGL}_3$ is birational to ${\proj}\mathcal{E}$. 
Since $\mathcal{M}_3$ is rational by Katsylo \cite{Ka3}, so is ${\proj}\mathcal{E}$. 
\end{proof} 

\begin{corollary}\label{fixed curve (9,7)} 
The fixed curve map $\mathcal{M}_{9,7,1}\to\mathcal{M}_3$ is dominant 
with general fibers being birationally identified with the bi-canonical systems. 
\end{corollary}

\subsection{The rationality of $\mathcal{M}_{10,6,1}$}\label{ssec:(10,6,1)}

Let $U\subset|{\Oplane}(4)|\times|{\Oplane}(2)|$ be the locus of pairs $(C, Q)$ 
such that $Q$ is a union of two distinct lines, and $C$ is smooth and transverse to $Q$. 
The 2-elementary $K3$ surfaces associated to the nodal sextics $C+Q$ 
have invariant $(g, k)=(3, 2)$, and parity $\delta=1$ by Lemma \ref{delta=1} (1). 
Hence we obtain a period map 
$\mathcal{P}\colon U/{\PGL}_3\dashrightarrow\mathcal{M}_{10,6,1}$. 

\begin{proposition}\label{birat (10,6,1)}
The period map $\mathcal{P}$ is birational. 
\end{proposition}

\begin{proof}
We define an $\frak{S}_2\ltimes(\frak{S}_4)^2$-cover of $U$ as the locus 
$\widetilde{U}\subset U\times({\proj}^2)^4\times({\proj}^2)^4$ of 
those $(C, Q, p_{1+}, p_{2+}, \cdots, p_{4-})$ such that 
$C\cap Q=\{ p_{1+}, p_{2+}, \cdots, p_{4-}\}$ and 
that $\{ p_{i+}\}_{i=1}^{4}$ belong to a same component of $Q$. 
By considering $\widetilde{U}$, the eight nodes $C\cap Q$ and the two components of $Q$ are 
labelled compatibly (cf. Example \ref{ex:3}). 
The rest node of $C+Q$, ${\rm Sing}(Q)$, is distinguished from those eight by the decomposition of $C+Q$. 
Then $\mathcal{P}$ will lift to a birational map 
$\widetilde{U}/{\PGL}_3\dashrightarrow{\cove}_{10,6,1}$. 
The projection $\widetilde{U}/{\PGL}_3\to U/{\PGL}_3$ is an $\frak{S}_2\ltimes(\frak{S}_4)^2$-covering, 
while the Galois group of ${\cove}_{10,6,1}\to\mathcal{M}_{10,6,1}$ is 
${\Or}(D_{L_+})$ for the lattice $L_+=U\oplus D_4\oplus A_1^4$. 
Since $|{\Or}(D_{L_+})|=2^4\cdot|{\Or}^+(4, 2)|=2^4\cdot72$ by \cite{M-S}, 
$\mathcal{P}$ is then birational. 
\end{proof}

\begin{proposition}\label{rational (10,6,1)}
The quotient $U/{\PGL}_3$ is rational. 
Hence $\mathcal{M}_{10,6,1}$ is rational. 
\end{proposition}

\begin{proof}
We keep the notation in the proof of Proposition \ref{rational (9,7)}. 
Let $\mathcal{F}$ be the bundle $\pi_{\ast}K_{\pi}$ over $\mathcal{M}_3$. 
By restriction, the space of singular plane conics is identified with 
the symmetric product of the canonical system of every smooth quartic. 
This implies that $U/{\PGL}_3$ is birational to the symmetric product of 
${\proj}\mathcal{F}$ over $\mathcal{M}_3$. 
Since ${\proj}\mathcal{F}\sim\mathcal{M}_3\times{\proj}^2$, 
we have 
$U/{\PGL}_3 \sim \mathcal{M}_3\times S^2{\proj}^2$. 
Then $S^2{\proj}^2$ 
is birational to the quotient of ${\C}^2\times{\C}^2$ by the permutation, and so is rational. 
Since $\mathcal{M}_3$ is rational (\cite{Ka3}), our assertion is proved. 
\end{proof}

\begin{corollary}\label{fixed curve (10,6,1)} 
The fixed curve map $\mathcal{M}_{10,6,1}\to\mathcal{M}_3$ is dominant 
with general fibers being birationally identified with the symmetric products of the canonical systems. 
\end{corollary}

\subsection{$\mathcal{M}_{10,6,0}$ and hyperelliptic curves}\label{ssec:(10,6,0)}

We construct 2-elementary $K3$ surfaces using curves on the Hirzebruch surface ${\F}_4$. 
We keep the notation of \S \ref{Sec:Hirze}. 
Let $U\subset|L_{2,0}|\times |L_{1,0}|$ be the open set of pairs $(C, H)$ such that 
$C$ and $H$ are smooth and transverse to each other. 
The 2-elementary $K3$ surface $(X, \iota)$ associated to the nodal $-2K_{{\F}_4}$-curve $C+H+\Sigma$ 
has invariant $(g, k)=(3, 2)$. 
Since the strict transform of $C-H-\Sigma$ in $Y'=X/\iota$ belongs to $4NS_{Y'}$, 
$(X, \iota)$ has parity $\delta=0$. 
Thus we obtain a period map 
$\mathcal{P}\colon U/{\aut}({\F}_4)\dashrightarrow\mathcal{M}_{10,6,0}$. 

\begin{proposition}\label{birat (10,6,0)} 
The map $\mathcal{P}$ is birational. 
\end{proposition}

\begin{proof}
We consider an $\frak{S}_8$-covering $\widetilde{U}\to U$ whose fiber over a $(C, H)\in U$ 
corresponds to labelings of the eight nodes $C\cap H$ of $C+H+\Sigma$.  
In the familiar way we will obtain a birational lift 
$\widetilde{U}/{\aut}({\F}_4)\dashrightarrow{\cove}_{10,6,0}$ of $\mathcal{P}$. 
The variety ${\cove}_{10,6,0}$ is an ${\Or}(D_{L_+})$-cover of $\mathcal{M}_{10,6,0}$ 
for the lattice $L_+=U(2)\oplus D_4^2$. 
By \cite{M-S} we have $|{\Or}(D_{L_+})|=|{\Or}^+(6, 2)|=8!$. 
\end{proof}

\begin{proposition}\label{rational (10,6,0)}
The quotient $(|L_{2,0}|\times|L_{1,0}|)/{\aut}({\F}_4)$ is rational. 
Therefore $\mathcal{M}_{10,6,0}$ is rational. 
\end{proposition}

\begin{proof}
This is a consequence of the slice method for the projection $|L_{2,0}|\times|L_{1,0}|\to|L_{1,0}|$, 
Proposition \ref{stabilizer of section}, and Katsylo's theorem \ref{Katsylo}. 
\end{proof}

Recall that the quotient $|L_{2,0}|/{\aut}({\F}_4)$ is birational to 
the moduli $\mathcal{H}_3$ of genus $3$ hyperelliptic curves (Proposition \ref{HE moduli}), 
and that the stabilizer in ${\aut}({\F}_4)$ of a general $C\in|L_{2,0}|$ 
is generated by its hyperelliptic involution (Corollary \ref{stab HE}). 
Since $\Sigma\cap C=\emptyset $, we have $L_{1,0}|_C \simeq L_{0,4}|_C \simeq 2K_C$. 
Then the restriction map $|L_{1,0}|\to|2K_C|$ is isomorphic because $h^0(L_{-1,0})=h^1(L_{-1,0})=0$. 
These show that 

\begin{corollary}\label{fixed curve (10,6,0)} 
The fixed curve map for $\mathcal{M}_{10,6,0}$ gives a dominant morphism 
$\mathcal{M}_{10,6,0}\to\mathcal{H}_3$ whose general fibers are birationally identified with 
the quotients of the bi-canonical systems by the hyperelliptic involutions. 
\end{corollary} 

We note that the hyperelliptic locus $\mathcal{H}_3$ is the theta-null divisor of $\mathcal{M}_3$.

\subsection{$\mathcal{M}_{11,5,1}$ and the universal genus $3$ curve}\label{ssec:(11,5)} 

Let $U\subset|{\Oplane}(4)|\times|{\Oplane}(2)|$ be the locus of pairs $(C, Q)$ such that 
$C$ is smooth, 
$Q$ is the union of distinct lines with $p={\rm Sing}(Q)$ lying on $C$, 
and each component of $Q$ is transverse to $C$. 
The point $p$ is a $D_4$-singularity of the sextic $C+Q$, 
and the rest singularities of $C+Q$ are the six nodes $C\cap Q\backslash p$. 
The 2-elementary $K3$ surface associated to $C+Q$ 
has invariant $(g, k)=(3, 3)$. 
Thus we obtain a period map $\mathcal{P}\colon U/{\PGL}_3\dashrightarrow \mathcal{M}_{11,5,1}$. 

\begin{proposition}\label{birat (11,5,1)}
The map $\mathcal{P}$ is birational. 
\end{proposition}

\begin{proof}
Let $\widetilde{U}\subset U\times({\proj}^2)^6$ be the locus of those $(C, Q, p_1\cdots, p_6)$ 
such that $C\cap Q\backslash {\rm Sing}(Q)=\{p_i\}_{i=1}^{6}$ 
and that $p_1, p_2, p_3$ lie on the same component of $Q$. 
For a point $(C,\cdots, p_6)$ of $\widetilde{U}$, 
the six nodes of $C+Q$ and the two components of $Q$ are labelled in a compatible way. 
In particular, the three tangents at the $D_4$-point of $C+Q$ are also distinguished. 
Thus $\mathcal{P}$ lifts to a birational map 
$\widetilde{U}/{\PGL}_3\dashrightarrow{\cove}_{11,5,1}$. 
The space $\widetilde{U}/{\PGL}_3$ is an $\frak{S}_2\ltimes(\frak{S}_3)^2$-cover of $U/{\PGL}_3$, 
while ${\cove}_{11,5,1}$ is an ${\Or}(D_{L_+})$-cover of $\mathcal{M}_{11,5,1}$ 
for the lattice $L_+=U\oplus D_4^2\oplus A_1$. 
By \cite{M-S} we calculate $|{\Or}(D_{L_+})|=|{\Or}^+(4, 2)|=72$. 
\end{proof}

\begin{proposition}\label{rational (11,5)}
The quotient $U/{\PGL}_3$ is rational. 
Hence $\mathcal{M}_{11,5,1}$ is rational. 
\end{proposition}

\begin{proof}
Let $\pi\colon\mathcal{X}_3\to\mathcal{M}_3$ be the universal genus $3$ curve (over an open locus), 
and let $\mathcal{E}$ be the subbundle of $\pi^{\ast}\pi_{\ast}K_{\pi}$ 
whose fiber over a $(C, p)$ is $H^0(K_C-p)$. 
The natural map $U\to\mathcal{X}_3$, $(C, Q)\mapsto (C, {\rm Sing}(Q))$, shows that 
$U/{\PGL}_3$ is birational to the symmetric product of ${\proj}\mathcal{E}$ over $\mathcal{X}_3$. 
In particular, $U/{\PGL}_3$ is birational to 
${\proj}^2\times\mathcal{X}_3$. 
The space $\mathcal{X}_3$ is known to be rational: 
see \cite{Do} where this fact is attributed to Shepherd-Barron. 
It is a consequence of the slice method and Miyata's theorem  
(associate to a pointed quartic $(C, p)$ the pointed tangent line $(T_pC, p)$).  
\end{proof}

\begin{corollary}\label{fixed curve (11,5)}
The fixed curve map for $\mathcal{M}_{11,5,1}$ lifts to a dominant map 
$\mathcal{M}_{11,5,1}\dashrightarrow\mathcal{X}_3$ whose fiber over 
a general $(C, p)$ is birationally identified with the symmetric product of the pencil $|K_C-p|$. 
\end{corollary}

\begin{remark}
One can also prove Propositions \ref{rational (10,6,1)} and \ref{rational (11,5)} 
by considering the configuration $C\cap Q\backslash {\rm Sing}(Q)$ of points 
and using the no-name lemma. 
This avoids resorting to the rationality of $\mathcal{M}_3$. 
\end{remark}

\subsection{The rationality of $\mathcal{M}_{8+k,8-k,\delta}$ with $k\geq4$}\label{ssec:(8+k,8-k)}

For $4\leq k\leq6$ 
let $U_k\subset|{\Oplane}(4)|\times|{\Oplane}(1)|^2$ be the locus of 
triplets $(C, L_1, L_2)$ such that 
$({\rm i})$ $C$ is smooth and passes $p=L_1\cap L_2$, 
$({\rm ii})$ $L_2$ is transverse to $C$, and 
$({\rm iii})$ $L_1$ is tangent to $C$ at $p$ with multiplicity $k-2$ and transverse to $C$ elsewhere. 
For $k=5, 6$ the point $p$ is an inflectional point of $C$ of order $k-2$. 
The sextic $B=C+L_1+L_2$ has the $D_{2k-2}$-point $p$, 
the $9-k$ nodes $(L_1+L_2)\cap C\backslash p$, and no other singularity. 
Taking the right resolution of $B$, 
we obtain a period map $\mathcal{P}_k\colon U_k/{\PGL}_3 \dashrightarrow \mathcal{M}_{8+k,8-k,\delta}$. 
Here $\delta=1$ for $k=4, 5$, and $\delta=0$ for $k=6$. 

\begin{proposition}\label{birat (8+k,8-k)} 
The map $\mathcal{P}_k$ is birational. 
\end{proposition}

\begin{proof}
See Example \ref{ex:4} for the case $k=6$. 

For $k=4$, we label the three nodes $L_2\cap C\backslash L_1$ and 
the two nodes $L_1\cap C\backslash L_2$ independently. 
This is realized by an $\frak{S}_3\times\frak{S}_2$-covering $\widetilde{U}_4\to U_4$ as before. 
Note that $L_1$ and $L_2$ are distinguished by their intersection with $C$. 
Then we will obtain a birational lift 
$\widetilde{U}_4/{\PGL}_3 \dashrightarrow {\cove}_{12,4,1}$ of $\mathcal{P}_4$. 
The invariant lattice $L_+$ is isometric to $U(2)\oplus A_1^2\oplus E_8$. 
By \cite{M-S} we have $|{\Or}(D_{L_+})|=2\cdot|{\rm Sp}(2, 2)|=12$. 
Therefore $\mathcal{P}_4$ is birational. 

For $k=5$, we label only the three nodes $L_2\cap C\backslash L_1$, 
which defines an $\frak{S}_3$-covering $\widetilde{U}_5\to U_5$. 
The rest one node $L_1\cap C\backslash L_2$ is obviously distinguished from those three. 
Therefore we will obtain a birational lift 
$\widetilde{U}_5/{\PGL}_3 \dashrightarrow {\cove}_{13,3,1}$ of $\mathcal{P}_5$. 
The invariant lattice $L_+$ is isometric to $U\oplus E_8\oplus  A_1^3$. 
Then we have ${\Or}(D_{L_+}) \simeq \frak{S}_3$, which proves the proposition. 
\end{proof}

\begin{proposition}\label{rational (8+k,8-k)}
The quotient $U_k/{\PGL}_3$ is rational. 
Therefore $\mathcal{M}_{12,4,1}$, $\mathcal{M}_{13,3,1}$, and $\mathcal{M}_{14,2,0}$ are rational. 
\end{proposition}

\begin{proof}
Consider the projection $\pi\colon U_k\to|{\Oplane}(1)|^2, (C, L_1, L_2)\mapsto(L_1, L_2)$. 
The group ${\PGL}_3$ acts almost transitively on $|{\Oplane}(1)|^2$ 
with the stabilizer of a general $(L_1, L_2)$ being isomorphic to $({\C}^{\times}\ltimes{\C})^2$. 
The $\pi$-fiber over $(L_1, L_2)$ is an open set of the linear system of quartics 
which are tangent to $L_1$ at $L_1\cap L_2$ with multiplicity $\geq k-2$. 
Thus our assertion follows from the slice method for $\pi$ and Miyata's theorem. 
\end{proof} 

The fixed curve maps are described as follows. 
Recall that the ordinary inflectional points of a smooth quartic $C\subset{\proj}^2$ 
are just the normal Weierstrass points of $C$, 
and the inflectional points of order $4$ are just the Weierstrass points of weight $2$. 
Then let $\mathcal{X}_3$ be the universal genus $3$ curve, 
$\mathcal{U}_5\subset\mathcal{X}_3$ the locus of normal Weierstrass points, 
and $\mathcal{U}_6\subset\mathcal{X}_3$ the locus of Weierstrass points of weight $2$. 

\begin{corollary}
The fixed curve map for $\mathcal{M}_{12,4,1}$ (resp. $\mathcal{M}_{8+k,8-k,\delta}$ with $k=5, 6$)  
lifts to a dominant map $\mathcal{M}_{12,4,1}\dashrightarrow \mathcal{X}_3$ 
(resp. $\mathcal{M}_{8+k,8-k,\delta} \dashrightarrow \mathcal{U}_k$)  
whose general fibers are birationally identified with the pencils $|K_C-p|$. 
\end{corollary}


\section{The case $g=2$}\label{sec:g=2} 

In this section we treat the case $g=2$, $k>0$. 
In view of the unirationality result \cite{Ma}, we may assume $k<9$. 
The case $(k, \delta)=(1, 0)$ is reduced to the rationality of $\mathcal{M}_{10,4,0}$ 
via the structure as an arithmetic quotient. 
The case $(k, \delta)=(1, 1)$ is settled by analyzing the quotient rational surfaces. 
The cases $2\leq k\leq4$ and $(k, \delta)=(5, 0)$ are studied using genus $2$ curves on ${\F}_3$, 
and the case $k\geq5$ with $\delta=1$ is studied using cuspidal plane quartics.

\subsection{The rationality of $\mathcal{M}_{10,8,0}$}\label{ssec:(10,8,0)} 

Here we may take the same approach as the one for $\mathcal{M}_{10,10,0}$ by Kond\=o \cite{Ko}. 
Recall that $\mathcal{M}_{10,8,0}$ is an open set of the modular variety $\mathcal{F}({\Or}(L_-))$ 
where $L_-$ is the 2-elementary lattice $U^2\oplus E_8(2)$. 
Since we have canonical isomorphisms 
${\Or}(L_-)\simeq {\Or}(L_-^{\vee})\simeq {\Or}(L_-^{\vee}(2))$, 
the variety $\mathcal{F}({\Or}(L_-))$ is isomorphic to $\mathcal{F}({\Or}(L_-^{\vee}(2)))$. 
The lattice $L_-^{\vee}(2)$ is isometric to $U(2)^2\oplus E_8$, 
so that $\mathcal{F}({\Or}(L_-^{\vee}(2)))$ is birational to $\mathcal{M}_{10,4,0}$. 
In \S \ref{ssec:(10,4,0)} we proved that $\mathcal{M}_{10,4,0}$ is rational, 
and hence $\mathcal{M}_{10,8,0}$ is rational.

\subsection{The rationality of $\mathcal{M}_{10,8,1}$}\label{ssec:(10,8,1)} 

We construct the corresponding right DPN pairs starting from the Hirzebruch surface ${\F}_2$. 
For a smooth curve $C\in|L_{2,2}|$ transverse to $\Sigma$, 
let $f\colon Y\to{\F}_2$ be the double cover branched along $C$. 
Since we have a conic fibration $Y\to{\F}_2\to{\proj}^1$, the surface $Y$ is rational. 
The curve $f^{\ast}\Sigma$ is a $(-4)$-curve on $Y$. 

\begin{lemma}\label{lemma 1 (10,8,1)}
We have $|\!-\!2K_Y|=f^{\ast}\Sigma+f^{\ast}|L_{1,0}|$. 
\end{lemma}

\begin{proof}
By the ramification formula we have $-2K_Y\simeq f^{\ast}L_{2,-2}$. 
Since $|L_{2,-2}|=\Sigma+|L_{1,0}|$ and ${\dim}\,|L_{1,0}|=3$, 
it suffices to show that ${\dim}\,|\!-\!2K_Y|=3$. 
We take a smooth curve $H\in|L_{1,0}|$ transverse to $C$. 
The inverse image $D=f^{\ast}H$ is a smooth genus $2$ curve disjoint from $f^{\ast}\Sigma$. 
Since ${\sheaf}_D(f^{\ast}\Sigma)\simeq {\sheaf}_D$, 
we have $-K_Y|_D\simeq K_D$ by the adjunction formula. 
Then the exact sequence 
\begin{equation*}
0 \to {\sheaf}_Y \to -2K_Y \to 2K_D\oplus {\sheaf}_{f^{\ast}\Sigma}(-4) \to 0 
\end{equation*}
shows that ${\dim}\,|\!-\!2K_Y|=h^0(2K_D)=3$. 
\end{proof}

Thus the resolution of the bi-anticanonical map $Y\dashrightarrow {\proj}^3$ of $Y$ 
is given by the composition of $f\colon Y\to{\F}_2$ and 
the morphism $\phi\colon{\F}_2\to{\proj}^3$ associated to $L_{1,0}$. 
The image $\phi({\F}_2)$ is a quadratic cone with vertex $\phi(\Sigma)$. 
In this way the quotient morphism $f$ is recovered from the bi-anticanonical map of $Y$.  
It follows that  

\begin{lemma}\label{lemma 2 (10,8,1)}
For two smooth curves $C, C'\in|L_{2,2}|$ transverse to $\Sigma$, 
the associated double covers $Y, Y'$ are isomorphic if and only if 
$C$ and $C'$ are ${\aut}({\F}_2)$-equivalent. 
\end{lemma}

Now we let $U\subset|L_{2,2}|\times |L_{1,0}|$ be the open locus of pairs $(C, H)$ 
such that $C$ and $H$ are smooth and that $C$ is transverse to $H+\Sigma$. 
To a $(C, H)\in U$ we associate the right DPN pair $(Y, B)$ where 
$f\colon Y\to{\F}_2$ is the double cover branched along $C$ and $B=f^{\ast}(H+\Sigma)$. 
Since $(Y, B)$ has invariant $(g, k)=(2, 1)$, 
we obtain a period map $U\to\mathcal{M}_{10,8,1}$. 
By Lemmas \ref{lemma 1 (10,8,1)} and \ref{lemma 2 (10,8,1)} 
the induced map $U/{\aut}({\F}_2)\dashrightarrow \mathcal{M}_{10,8,1}$ is generically injective. 
In view of the equality 
${\dim}(U/{\aut}({\F}_2))=10$, 
we have a birational equivalence 
$U/{\aut}({\F}_2)\sim \mathcal{M}_{10,8,1}$. 

\begin{proposition}\label{rational (10,8,1)}
The quotient $U/{\aut}({\F}_2)$ is rational. 
Hence $\mathcal{M}_{10,8,1}$ is rational. 
\end{proposition} 

\begin{proof}
This follows from the slice method for the projection $|L_{2,2}|\times |L_{1,0}|\to|L_{1,0}|$, 
Proposition \ref{stabilizer of section}, and Katsylo's theorem \ref{Katsylo}. 
\end{proof}


\begin{remark}\label{second proof (10,8,1)} 
The above $(Y, B)$ are (generically) right resolution of the plane sextics $C+Q$ 
where $C$ is a one-nodal quartic and $Q$ is a smooth conic transverse to $C$. 
Although the period map for this sextic model has degree $>1$, 
we can analyze its fibers to derive the rationality of $\mathcal{M}_{10,8,1}$. 
This alternative approach 
may also be of some interest. 
We here give an outline for future reference. 

Let $\widetilde{U}\subset|{\Oplane}(4)|\times({\proj}^2)^8$ be the locus of $(C, p_1,\cdots, p_8)$ 
such that $C$ is one-nodal and $\{p_i\}_{i=1}^8=C\cap Q$ for a smooth conic $Q$. 
Using the labeling $\{p_i\}_{i=1}^8$, 
we obtain a birational map 
$\mathcal{U}=\widetilde{U}/{\PGL}_3 \dashrightarrow {\cove}_{10,8,1}$
as before. 
The projection ${\cove}_{10,8,1}\dashrightarrow\mathcal{M}_{10,8,1}$ is 
an $\frak{S}_8\ltimes({\Z}/2)^6$-covering, 
which exceeds the the obvious $\frak{S}_8$-symmetry of $\mathcal{U}$.  
In order to find the rest $({\Z}/2)^6$-symmetry, 
let $H$ be the group of even cardinality subsets of $\{1,\cdots, 8\}$ 
with the symmetric difference operation. 
For $\{i, j\}\in H$ and $(C, p_1,\cdots, p_8)\in\mathcal{U}$, 
consider the quadratic transformation  
$\varphi\colon{\proj}^2\dashrightarrow{\proj}^2$ based at $p_i, p_j$, and $p_0={\rm Sing}(C)$. 
We set $C^+=\varphi(C), p_i^+=\varphi(\overline{p_0p_i}), p_j^+=\varphi(\overline{p_0p_j})$, 
and $p_k^+=\varphi(p_k)$ for $k\ne i, j$. 
Then we have $(C^+, p_1^+,\cdots, p_8^+)\in\mathcal{U}$, 
and this defines an action of $H$ on $\mathcal{U}$. 
The element $\{1,\cdots, 8\}\in H$ acts on $\mathcal{U}$ trivially 
(it 
gives the covering transformation of the above $Y\to{\F}_2$). 
Now the period map $\mathcal{U}\dashrightarrow\mathcal{M}_{10,8,1}$ is $\frak{S}_8\ltimes H$-invariant, 
so that $\mathcal{M}_{10,8,1}$ is birational to $\mathcal{U}/(\frak{S}_8\ltimes H)$ by a degree comparison. 
We can prove that $\mathcal{U}/(\frak{S}_8\ltimes H)$ is rational. 
\end{remark}

\subsection{$\mathcal{M}_{11,7,1}$ and genus $2$ curves on ${\F}_3$}\label{ssec:(11,7)}

We construct 2-elementary $K3$ surfaces using curves on ${\F}_3$. 
Let $U\subset|L_{2,0}|\times |L_{1,1}|$ be the open set of pairs $(C, D)$ such that 
$C$ and $D$ are smooth and transverse to each other. 
Then the curves $C+D+\Sigma$ belong to $|\!-\!2K_{{\F}_3}|$. 
Taking the right resolution of $C+D+\Sigma$, 
we obtain a period map $\mathcal{P}\colon U/{\aut}({\F}_3)\dashrightarrow\mathcal{M}_{11,7,1}$. 

\begin{proposition}
The map $\mathcal{P}$ is birational. 
\end{proposition}

\begin{proof}
We consider an $\frak{S}_8$-covering $\widetilde{U}\to U$ whose fiber over a $(C, D)\in U$ 
corresponds to labelings of the eight nodes $C\cap D$. 
The rest node of $C+D+\Sigma$ is the one point $D\cap\Sigma$, 
which is obviously distinguished from those eight. 
Thus we obtain a birational lift 
$\widetilde{U}/{\aut}({\F}_3)\dashrightarrow{\cove}_{11,7,1}$ of $\mathcal{P}$. 
The projection ${\cove}_{11,7,1}\dashrightarrow\mathcal{M}_{11,7,1}$ is an ${\Or}(D_{L_+})$-covering 
for the lattice $L_+=U\oplus D_4\oplus A_1^5$. 
By \cite{M-S} we have $|{\Or}(D_{L_+})|=|{\Or}^+(6, 2)|=8!$, 
which implies that $\mathcal{P}$ is birational.   
\end{proof}

\begin{proposition}\label{rational (11,7)}
The quotient $(|L_{2,0}|\times |L_{1,1}|)/{\aut}({\F}_3)$ is rational. 
Therefore $\mathcal{M}_{11,7,1}$ is rational. 
\end{proposition}

\begin{proof}
We apply the slice method to the projection $|L_{2,0}|\times |L_{1,1}| \to |L_{1,1}|$, 
and then use Proposition \ref{L_{1,1}} and Miyata's theorem. 
\end{proof}

For every smooth $C\in|L_{2,0}|$ the linear system $|L_{1,1}|$ is identified with $|4K_C|$ by restriction. 
Indeed, we have $4K_C\simeq L_{0,4}|_C$ by the adjunction formula, 
and $L_{0,4}|_C\simeq L_{1,1}|_C$ because $C\cap\Sigma=\emptyset$. 
Then the vanishings $h^0(L_{-1,1})=h^1(L_{-1,1})=0$ prove our claim. 
In view of Proposition \ref{HE moduli} and Corollary \ref{stab HE}, we have 

\begin{corollary}\label{fixed curve (11,7)}
The fixed curve map $\mathcal{M}_{11,7,1}\to\mathcal{M}_2$ is dominant 
with a general fiber being birationally identified with 
the quotient of $|4K_C|$ by the hyperelliptic involution. 
\end{corollary}

\subsection{$\mathcal{M}_{12,6,1}$, $\mathcal{M}_{13,5,1}$, $\mathcal{M}_{14,4,0}$ and 
                       genus $2$ curves on ${\F}_3$}\label{ssec:g=2 II}

We construct 2-elementary $K3$ surfaces using curves on ${\F}_3$. 
For $3\leq k\leq5$ let $U_k\subset|L_{2,0}|\times |L_{1,0}|\times |L_{0,1}|$ be the locus of triplets 
$(C, H, F)$ such that 
$({\rm i})$ $C$ and $H$ are smooth and transverse to each other,  
$({\rm ii})$ $F$ intersects with $C$ at $p=F\cap H$ with multiplicity $k-3$ in case $k=4, 5$, 
and is transverse to $C+H$ in case $k=3$.  
It is easy to calculate ${\dim}\, U_k=19-k$. 
For a $(C, H, F)\in U_k$ the curve $B=C+H+F+\Sigma$ belongs to $|\!-\!2K_{{\F}_3}|$. 
When $k=3$, $B$ has only nodes as the singularities. 
When $k=4, 5$, $B$ has the $D_{2k-4}$-singularity $p$, 
the nodes $C\cap H\backslash p$, $F\cap C\backslash p$, $F\cap\Sigma$, 
and no other sigularity. 
The 2-elementary $K3$ surface $(X, \iota)$ associated to $B$ has invariant $(r, a)=(9+k, 9-k)$. 
When $k=5$, $(X, \iota)$ has parity $\delta=0$. 
Indeed, let $(Y', B')$ be the corresponding right DPN pair. 
The curve $B'$ has two components over $p$, say $E_3$ and $E_5$, 
whose numbering corresponds to the one for the vertices of the $D_6$-graph in Figure \ref{dual graph}. 
Then the sum of $-E_3+E_5$ and the strict transform of $C+F-H-\Sigma$ belongs to $4NS_{Y'}$, 
which proves $\delta=0$.  
Thus we obtain a period map 
$\mathcal{P}_k\colon U_k/{\aut}({\F}_3)\dashrightarrow\mathcal{M}_{9+k,9-k,\delta}$ 
where $\delta=1$ if $k=3, 4$ and $\delta=0$ if $k=5$. 

\begin{proposition}\label{birat g=2 II}
The map $\mathcal{P}_k$ is birational. 
\end{proposition}

\begin{proof}
For $k=3$ we label the six nodes $C\cap H$ and the two nodes $C\cap F$ independently, 
which is realized by an $\frak{S}_6\times\frak{S}_2$-covering $\widetilde{U}_3\to U_3$. 
The rest two nodes, $F\cap H$ and $F\cap\Sigma$, are distinguished by the irreducible decomposition of $B$. 
This defines a birational lift $\widetilde{U}_3/{\aut}({\F}_3) \dashrightarrow {\cove}_{12,6,1}$ of $\mathcal{P}_3$. 
The variety ${\cove}_{12,6,1}$ is an ${\Or}(D_{L_+})$-cover of $\mathcal{M}_{12,6,1}$ 
for the lattice $L_+=U\oplus A_1^2\oplus D_4^2$. 
By \cite{M-S} we have $|{\Or}(D_{L_+})|=2\cdot|{\rm Sp}(4, 2)|=2\cdot6!$. 
Therefore $\mathcal{P}_3$ has degree $1$. 

For $k=4, 5$ we consider labelings the five nodes $C\cap H\backslash F$, 
which defines an $\frak{S}_5$-cover $\widetilde{U}_k$ of $U_k$. 
The rest singularities of $B$ and the branches at the $D_{2k-4}$-point are a priori distinguished.  
Thus we obtain a birational lift 
$\widetilde{U}_k/{\aut}({\F}_3) \dashrightarrow {\cove}_{9+k,9-k,\delta}$ 
of $\mathcal{P}_k$. 
The variety ${\cove}_{9+k,9-k,\delta}$ is an ${\Or}(D_{L_-})$-cover of $\mathcal{M}_{9+k,9-k,\delta}$ 
for the lattice $L_-=U\oplus U(2)\oplus D_4\oplus A_1^{5-k}$. 
We have $|{\Or}(D_{L_-})|=|{\Or}^-(4, 2)|=5!$ for both $k=4, 5$. 
Hence $\mathcal{P}_k$ is birational. 
\end{proof}

\begin{proposition}\label{rational g=2 II} 
The quotient $U_k/{\aut}({\F}_3)$ is rational. 
Therefore $\mathcal{M}_{12,6,1}$, $\mathcal{M}_{13,5,1}$, and $\mathcal{M}_{14,4,0}$ are rational. 
\end{proposition}

\begin{proof}
We apply the slice method to the projection $\pi_k\colon U_k\to|L_{1,0}|\times |L_{0,1}|$. 
By Propositions \ref{stabilizer of section} and \ref{stabilizer of fiber} 
the group ${\aut}({\F}_3)$ acts on $|L_{1,0}|\times |L_{0,1}|$ almost transitively 
with the stabilizer $G$ of a general $(H, F)$ being connected and solvable. 
The fiber $\pi_k^{-1}(H, F)$ is an open set of a linear subspace ${\proj}V_k$ of $|L_{2,0}|$. 
Then ${\proj}V_k/G$ is rational by Miyata's theorem. 
\end{proof}

As in the paragraph just before Corollary \ref{fixed curve (11,7)}, 
we see that for every smooth $C\in|L_{2,0}|$ the linear system $|L_{1,0}|$ is identified with $|3K_C|$ by restriction. 
When $k=3$, we obtain a pointed genus $2$ curve $(C, p)$ by considering either of the two points $F\cap C$. 
When $k\geq4$, the point $p=H\cap F$ determines $F$, and $H$ is a general member of $|3K_C-p|$. 
For $k=5$, $p$ is a Weierstrass point of $C$. 
These infer the following. 

\begin{corollary}\label{fixed curve g=2, 2<k<6}
Let $\mathcal{X}_2$ be the moduli of pointed genus $2$ curves $(C, p)$, 
and let $\mathcal{W}\subset\mathcal{X}_2$ be the divisor of Weierstrass points. 
The fixed curve maps for $\mathcal{M}_{12,6,1}$, $\mathcal{M}_{13,5,1}$, and $\mathcal{M}_{14,4,0}$
lift to rational maps $\widetilde{F}_k:\mathcal{M}_{9+k,9-k,\delta}\dashrightarrow\mathcal{X}_2$. 
Then 

$(1)$ $\widetilde{F}_3$ is dominant with a general fiber birationally identified with 
the quotient of $|3K_C|$ by the hyperelliptic involution. 
 
$(2)$ $\widetilde{F}_4$ is dominant with a general fiber birationally identified with $|3K_C-p|$. 

$(3)$ $\widetilde{F}_5$ is a dominant map onto $\mathcal{W}$ 
whose general fiber is birationally identified with the quotient of $|3K_C-p|$ by the hyperelliptic involution. 
\end{corollary}

\subsection{$\mathcal{M}_{9+k,9-k,1}$ with $k\geq5$ and cuspidal plane quartics}\label{ssec:g=2 III}

Let $U\subset|{\Oplane}(4)|$ be the locus of 
plane quartics with an ordinary cusp and with no other singularity. 
For $5\leq k\leq8$ 
we denote by $U_k\subset U\times|{\Oplane}(1)|$ the locus of pairs $(C, L)$ such that 
if $M\subset{\proj}^2$ is the tangent line of $C$ at the cusp,  
then $L$ intersects with $C$ at $C\cap M\backslash{\rm Sing}(C)$ with multiplicity $k-5$, 
and is transverse to $C$ elsewhere.  
The space $U_k$ is of dimension $19-k$. 
This is obvious for $5\leq k\leq7$. 
For $k=8$, if we take the homogeneous coordinate $[X,Y,Z]$ of ${\proj}^2$ and  
normalize $p=[0,0,1]$, $M=\{Y=0\}$, and $L=\{Z=0\}$, 
then quartics $C$ having cusp at $p$ with $(C.M)_p=3$ and $(C.L)_{L\cap M}=3$ 
are defined by the equations 
\begin{equation}\label{eqn: special cuspidal quartic}
a_{02}Y^2Z^2+\sum_{i+j=3}a_{ij}X^iY^jZ+a_{13}XY^3+a_{04}Y^4=0, 
\end{equation} 
in which the coefficients $a_{\ast}$ may be taken general. 
This shows that $U_8$ is of the expected dimension. 

For a $(C, L)\in U_k$ the sextic $B=C+L+M$ has an $E_7$-singularity at $p={\rm Sing}(C)$, 
$9-k$ nodes at $C\cap L\backslash M$, 
a $D_{2k-8}$-singularity at $L\cap M$ (resp. two nodes at $L\cap M$ and $C\cap M\backslash p$) 
in case $k\geq6$ (resp. $k=5$), 
and no other singularity. 
Hence the 2-elementary $K3$ surface $(X, \iota)$ associated to $B$ 
has invariant $(r, a)=(9+k, 9-k)$. 
When $k=5$, $(X, \iota)$ has parity $\delta=1$ by Lemma \ref{delta=1} $(1)$. 
Thus we obtain a period map 
$\mathcal{P}_k\colon U_k\dashrightarrow\mathcal{M}_{9+k,9-k,1}$. 

\begin{proposition}\label{birat g=2 III}
The map $\mathcal{P}_k$ is birational. 
\end{proposition}

\begin{proof}
We consider an $\frak{S}_{9-k}$-covering $\widetilde{U}_k\to U_k$ 
whose fiber over a $(C, L)\in U_k$ corresponds to labelings the $9-k$ nodes $C\cap L\backslash M$. 
As described above, the rest singular points of $B$ are a priori distinguished 
by the type of singularity and by the irreducible decomposition of $B$. 
Also note that the two lines $L$ and $M$ are distinguished by their intersection with $C$. 
Therefore we will obtain a birational lift 
$\widetilde{U}_k/{\PGL}_3 \dashrightarrow {\cove}_{9+k,9-k,1}$
of $\mathcal{P}_k$. 
The anti-invariant lattice $L_-$ is isometric to $U^2\oplus A_1^{9-k}$. 
It is easy to check that ${\Or}(D_{L_-})\simeq\frak{S}_{9-k}$. 
Since ${\PGL}_3$ acts on $U$ almost freely, this shows that $\mathcal{P}_k$ has degree $1$. 
\end{proof}

\begin{proposition}\label{rational g=2 III}
The quotient $U_k/{\PGL}_3$ is rational. 
Therefore $\mathcal{M}_{9+k,9-k,1}$ with $5\leq k\leq8$ is rational. 
\end{proposition}

\begin{proof}
Let $V\subset{\proj}^2\times|{\Oplane}(1)|^2$ be the locus of triplets $(p, M, L)$ such that $p\in M$. 
We have the ${\PGL}_3$-equivariant map 
$\pi_k\colon U_k \to V$, $(C, L)\mapsto ({\rm Sing}(C), M, L)$, 
where $M$ is the tangent line of $C$ at the cusp. 
The group ${\PGL}_3$ acts on $V$ almost transitively with the stabilizer of a general point 
isomorphic to ${\C}^{\times}\times({\C}^{\times}\ltimes{\C})$. 
Since a general $\pi_k$-fiber is an open set of a sub-linear system of $|{\Oplane}(4)|$, 
the assertion follows from the slice method for $\pi_k$ and Miyata's theorem. 
\end{proof}

Let $\mathcal{X}_2$ be the moduli of pointed genus $2$ curves $(C, p)$. 
Recall (cf. \cite{Do}) that we have a birational map 
$\mathcal{X}_2\dashrightarrow U/{\PGL}_3$ 
by associating to a pointed curve $(C, p)$ the image $\phi(C)\subset{\proj}^2$ by the linear system $|K_C+2p|$. 
When $p$ is not a Weierstrass point, $\phi(C)$ is a quartic with a cusp $\phi(p)$, 
and the projection from $\phi(p)$ gives the hyperelliptic map of $C$. 
Thus we see the following.  

\begin{corollary}
The fixed curve map for $\mathcal{M}_{9+k,9-k,1}$ with $5\leq k\leq8$ 
lifts to a rational map $\widetilde{F}_k:\mathcal{M}_{9+k,9-k,1}\dashrightarrow\mathcal{X}_2$.
Then

$(1)$ $\widetilde{F}_5$ is dominant with general fibers birationally identified with $|K_C+2p|$.

$(2)$ $\widetilde{F}_6$ is dominant with general fibers birationally identified with $|3p|$. 

$(3)$ $\widetilde{F}_7$ is birational. 

$(4)$ $\widetilde{F}_8$ is generically injective 
          with a generic image being the divisor of those $(C, p)$ with $5p-2K_C$ effective 
          minus the divisor of Weierstrass points. 
\end{corollary}


\section{The case $g=1$ (I)}\label{sec:g=1 (1)} 

In this section we study the case $g=1$, $1\leq k\leq4$, $\delta=1$. 
It seems (at least to the author) difficult to use the previous methods for this case. 
We obtain general members of ${\moduli}$ from plane sextics of the form $C_1+C_2$ 
where $C_1$ is a nodal cubic and $C_2$ is a smooth cubic intersecting with $C_1$ 
at the node with suitable multiplicity. 
But the period map for this sextic model has degree $>1$. 
In order to pass to a ``canonical'' construction, 
we find a Weyl group symmetry in this sextic model, 
which reflects the fact that the Galois group of ${\cover}\dashrightarrow{\moduli}$ is 
the Weyl group or its central quotient. 
This leads to regard the cubics $C_i$ as anticanonical curves on del Pezzo surfaces of degree $k$. 

\subsection{Weyl group actions on universal families}\label{ssec:equiv Weyl}

We begin with constructing a universal family of marked del Pezzo surfaces 
with an equivariant action by the Weyl group. 
For $5\leq d\leq8$ let $U^d\subset({\proj}^2)^d$ be the open set of ordered $d$ points 
in general position in the sense of \cite{De}. 
There exists a geometric quotient $\mathcal{U}^d=U^d/{\PGL}_3$ of $U^d$ by ${\PGL}_3$ (see \cite{D-O}). 
Since any $\mathbf{p}=(p_1,\cdots, p_d)\in U^d$ has trivial stabilizer, 
the quotient map $U^d\to\mathcal{U}^d$ is a principal ${\PGL}_3$-bundle by Luna's etale slice theorem. 
Hence by \cite{GIT} Proposition 7.1 the projection ${\proj}^2\times U^d\to U^d$ 
descends to a smooth morphism $\mathcal{X}^d\to\mathcal{U}^d$ where 
$\mathcal{X}^d$ is a geometric quotient of ${\proj}^2\times U^d$ by ${\PGL}_3$. 
Here ${\PGL}_3$ acts on ${\proj}^2$ in the natural way. 
For $1\leq i\leq d$ we have the section $s_i$ of $\mathcal{X}^d\to\mathcal{U}^d$ 
induced by the $i$-th projection $U^d\to{\proj}^2$, $(p_1,\cdots, p_d)\mapsto p_i$. 
Blowing-up the $d$ sections $s_i$, 
we obtain a family $f\colon\mathcal{Y}^d\to\mathcal{U}^d$ of del Pezzo surfaces of degree $9-d$. 
The exceptional divisor over $s_i$, denoted by $\mathcal{E}_i$, is a family of $(-1)$-curves. 

Let $\Lambda_d$ be the lattice $\langle 1\rangle \oplus \langle -1\rangle^d$ with 
a natural orthogonal basis $h, e_1,\cdots, e_d$. 
For each $\mathbf{p}\in\mathcal{U}^d$, 
the Picard lattice of the $f$-fiber $\mathcal{Y}^d_{\mathbf{p}}$ over $\mathbf{p}$ 
is isometric to $\Lambda_d$ by associating $h$ to the pullback of ${\Oplane}(1)$ 
and $e_i$ to the class of the $(-1)$-curve $(\mathcal{E}_i)_{\mathbf{p}}$. 
This gives a trivialization $\varphi\colon\mathcal{U}^d\times\Lambda_d\to R^2f_{\ast}{\Z}$ of the local system. 
Thus we have a universal family 
\begin{equation}
(f:\mathcal{Y}^d\to\mathcal{U}^d, \: \: \varphi)
\end{equation} 
of marked del Pezzo surfaces. 

Recall that the Weyl group $W_d$ is the group of isometries of $\Lambda_d$ 
which fix the vector $3h-\sum_{i=1}^{d}e_i$. 
Let $w\in W_d$. 
For each $\mathbf{p}\in\mathcal{U}^d$ 
the marked del Pezzo surface $(\mathcal{Y}^d_{\mathbf{p}}, \varphi_{\mathbf{p}})$ 
is transformed by $w$ to $(\mathcal{Y}^d_{\mathbf{p}}, \varphi_{\mathbf{p}}\circ w^{-1})$, 
which is isomorphic to $(\mathcal{Y}^d_{\mathbf{q}}, \varphi_{\mathbf{q}})$ for a $\mathbf{q}\in \mathcal{U}^d$. 
Indeed, the classes $\varphi_{\mathbf{p}}\circ w^{-1}(e_i)$ are represented by 
disjoint $(-1)$-curves $E_i'$ on $\mathcal{Y}^d_{\mathbf{p}}$ 
which define a blow-down $\mathcal{Y}^d_{\mathbf{p}}\to{\proj}^2$. 
Then $\mathbf{q}$ is the blown-down points of $E_1',\cdots, E_d'$. 
We set $w(\mathbf{p})=\mathbf{q}$. 
By construction, we have an isomorphism 
$\mu_w\colon\mathcal{Y}^d_{\mathbf{p}}\to\mathcal{Y}^d_{w(\mathbf{p})}$ with 
$(\mu_w)_{\ast}\circ\varphi_{\mathbf{p}}\circ w^{-1}=\varphi_{w(\mathbf{p})}$. 
The last equality characterizes $\mu_w$ uniquely 
because the cohomological representation 
${\aut}(\mathcal{Y}^d_{\mathbf{p}})\to {\Or}({\rm Pic}(\mathcal{Y}^d_{\mathbf{p}}))$ is injective. 
This ensures that 
$(w'w)(\mathbf{p})=w'(w(\mathbf{p}))$ and that 
$\mu_{w'}\circ\mu_w=\mu_{w'w}$ for $w, w'\in W_d$. 
In this way we obtain an equivariant action of $W_d$ on the family $f\colon\mathcal{Y}^d\to\mathcal{U}^d$. 
The $W_d$-action on $\mathcal{U}^d$ is the Cremona representation. 
In the following we will refer to \cite{D-O} for the basic properties of the Cremona representation for each $d$. 
 
On $\mathcal{Y}^d$ we have two natural $W_d$-linearized vector bundles: 
the relative tangent bundle $T_f$ and the relative anticanonical bundle $K_f^{-1}$. 
The direct image $f_{\ast}K_f^{-1}$ is a $W_d$-linearized bundle over $\mathcal{U}^d$. 
The fiber of the ${\proj}^{9-d}$-bundle ${\proj}(f_{\ast}K_f^{-1})$ 
over a point $\mathbf{p}=[(p_1,\cdots, p_d)]$ of $\mathcal{U}^d$ 
is identified with the linear system of plane cubics through the $d$ points $p_1,\cdots, p_d$ of ${\proj}^2$.

\subsection{$\mathcal{M}_{11,9,1}$ and del Pezzo surfaces of degree $1$}\label{ssec:(11,9)}

\subsubsection{A period map}

Let $f\colon\mathcal{Y}^8\to\mathcal{U}^8$ be the family of marked del Pezzo surfaces of degree $1$ 
constructed above. 
The Cremona representation of the Weyl group $W_8$ has kernel ${\Z}/2$ 
whose generator $w_0$ acts on the $f$-fibers by the Bertini involutions. 
The quotient $W_8/\langle w_0\rangle$ is isomorphic to ${\Or}^+(8, 2)$ and 
acts on $\mathcal{U}^8$ almost freely. 

We let $\mathcal{V}^8\subset{\proj}(f_{\ast}K_f^{-1})$ be the locus of singular anticanonical curves. 
The locus $\mathcal{V}^8$ is invariant under the $W_8$-action on ${\proj}(f_{\ast}K_f^{-1})$, 
and the projection $\mathcal{V}^8\to\mathcal{U}^8$ is of degree $12$. 
We denote by $\mathcal{E}$ the pullback of the bundle $f_{\ast}K_f^{-1}$ by $\mathcal{V}^8\to\mathcal{U}^8$.  
An open set of the variety ${\proj}\mathcal{E}$ parametrizes 
triplets $(\mathcal{Y}^8_{\mathbf{p}}, C_1, C_2)$ of 
a marked del Pezzo surface $\mathcal{Y}^8_{\mathbf{p}}=Y$, 
a nodal $-K_Y$-curve $C_1$, and a smooth $-K_Y$-curve $C_2$. 
Notice that $C_1$ is irreducible and transverse to $C_2$ because $(C_1.C_2)=1$. 
Taking a right resolution of the DPN pair $(Y, C_1+C_2)$, 
we obtain a 2-elementary $K3$ surface of invariant $(g, k)=(1, 1)$. 
This defines a period map $\mathcal{P}\colon{\proj}\mathcal{E}\dashrightarrow\mathcal{M}_{11,9,1}$. 

\begin{proposition}\label{birat (11,9)}
The map $\mathcal{P}$ is $W_8$-invariant and 
descends to a birational map ${\proj}\mathcal{E}/W_8 \dashrightarrow\mathcal{M}_{11,9,1}$. 
\end{proposition}

\begin{proof}
Two $W_8$-equivalent points of ${\proj}\mathcal{E}$ give rise to isomorphic DPN pairs
so that $\mathcal{P}$ is $W_8$-invariant. 
We shall show that $\mathcal{P}$ lifts to a birational map 
${\proj}\mathcal{E} \dashrightarrow {\cove}_{11,9,1}$. 
For a point $(\mathbf{p}, C_1, C_2)$ of ${\proj}\mathcal{E}$, 
we have the marking $\varphi_{\mathbf{p}}$ of the Picard lattice of 
$\mathcal{Y}^8_{\mathbf{p}}$ induced by $\mathbf{p}$. 
The curve $C_1+C_2$ has two nodes, namely the intersection point $C_1\cap C_2$ and the node of $C_1$, 
which are clearly distinguished. 
This induces a marking of the invariant lattice of $(X, \iota)=\mathcal{P}(\mathbf{p}, C_1, C_2)$, 
which defines a lift ${\lift}\colon{\proj}\mathcal{E} \dashrightarrow {\cove}_{11,9,1}$ of $\mathcal{P}$. 
In order to show that ${\lift}$ is birational, 
the key point is that the blow-down $\pi\colon\mathcal{Y}^8_{\mathbf{p}}\to{\proj}^2$ defined by $\mathbf{p}$ 
translates the triplet $(\mathbf{p}, C_1, C_2)$ into the plane sextic $\pi(C_1+C_2)$ 
endowed with a labeling of the nine nodes $\pi(C_1)\cap\pi(C_2)$. 
Indeed, $\pi(C_1)\cap\pi(C_2)$ consists of the eight ordered points $\mathbf{p}$ 
and the rest one point $\pi(C_1\cap C_2)$. 
The first eight nodes recover the marked del Pezzo surface, 
and the ninth is determined as the base point of the associated cubic pencil. 
Note also that the unlabeled node ${\rm Sing}(\pi(C_1))$ of $\pi(C_1+C_2)$ is clearly 
distinguished from those nine. 
In this way ${\proj}\mathcal{E}$ is birationally identified with 
the ${\PGL}_3$-quotient of the space of such nodal sextics with labelings. 
Now one may follow the recipe in \S \ref{ssec: recipe} to see that ${\lift}$ is birational. 

We compare the two projections ${\proj}\mathcal{E}\to{\proj}\mathcal{E}/W_8$ 
and ${\cove}_{11,9,1}\dashrightarrow\mathcal{M}_{11,9,1}$. 
The Bertini involution $w_0\in W_8$ acts trivially on ${\proj}\mathcal{E}$ 
because it acts trivially on the anticanonical pencils $|\!-\! K_Y|$. 
Hence ${\proj}\mathcal{E}\to{\proj}\mathcal{E}/W_8$ is an ${\Or}^+(8, 2)$-covering. 
On the other hand, ${\cove}_{11,9,1}$ is an ${\Or}(D_{L_+})$-cover of 
$\mathcal{M}_{11,9,1}$ for the lattice $L_+=U\oplus A_1^9$.  
Since ${\Or}(D_{L_+})\simeq {\Or}^+(8, 2)$, this finishes the proof. 
\end{proof}

\subsubsection{The rationality} 

The $W_8$-action on ${\proj}\mathcal{E}$ gets rid of the markings of del Pezzo surfaces. 
This implies that the quotient ${\proj}\mathcal{E}/W_8$ is birationally a moduli of triplets $(Y, C_1, C_2)$ 
where $Y$ is an (unmarked) del Pezzo surface of degree $1$, and 
$C_1$ (resp. $C_2$) is a singular (resp. smooth) $-K_Y$-curve on $Y$. 
We consider the blow-up $\widehat{Y}\to Y$ at the base point $C_1\cap C_2$ of $|\!-\! K_Y|$. 
The quotient of $\widehat{Y}$ by the Bertini involution is the Hirzebruch surface ${\F}_2$, 
and the quotient morphism $\phi\colon\widehat{Y}\to{\F}_2$ is branched over the $(-2)$-curve $\Sigma$ 
and over a smooth $L_{3,0}$-curve $\Gamma$. 
Note that $\phi^{-1}(\Sigma)$ is the exceptional curve of $\widehat{Y}\to Y$. 
The pencil $|L_{0,1}|$ on ${\F}_2$ is pulled-back by $\phi$ to the pencil $|\!-\! K_{\widehat{Y}}|=|\!-\! K_Y|$. 
Then $\phi(C_1)$ is an $L_{0,1}$-fiber tangent to $\Gamma$, 
and $\phi(C_2)$ is an $L_{0,1}$-fiber transverse to $\Gamma$. 
If we let $U\subset|L_{3,0}|\times|L_{0,1}|^2$ be the locus of triplets $(\Gamma, F_1, F_2)$ such that 
$\Gamma$ is smooth and $F_1$ (resp. $F_2$) is tangent (resp. transverse) to $\Gamma$, 
we thus obtain a rational map 
\begin{equation*}
{\proj}\mathcal{E}/W_8 \dashrightarrow U/{\aut}({\F}_2). 
\end{equation*}
Since this construction may be reversed, the map is birational. 

\begin{proposition}\label{rational (11,9)}
The quotient $U/{\aut}({\F}_2)$ is rational. 
Thus $\mathcal{M}_{11,9,1}$ is rational. 
\end{proposition}

\begin{proof}
Let $V$ be the fiber product ${\F}_2\times_{{\proj}^1}{\F}_2$, 
where $\pi\colon{\F}_2\to{\proj}^1$ is the natural projection.  
We consider the ${\aut}({\F}_2)$-equivariant map 
\begin{equation*}
\varphi: U \to V\times |L_{0,1}|, \qquad  (\Gamma, F_1, F_2)\mapsto((p, q), F_2), 
\end{equation*}
where $p$ is the point of tangency of $\Gamma$ and $F_1$, 
and $q=\Gamma|_{F_1}-2p$. 
By \S \ref{Sec:Hirze}, the group ${\aut}({\F}_2)$ acts on $V\times |L_{0,1}|$ almost transitively. 
If we normalize $F_1=\pi^{-1}([0,1])$ and  $F_2=\pi^{-1}([1,0])$, 
the stabilizer $G$ of a general point $((p, q), F_2)$ with $p, q\in F_1$ is described by the exact sequence  
$0\to H \to G \to {\C}^{\times} \to 1$, 
where $H=\{ s\in H^0({\Oline}(2)), s([0,1])=0\}$. 
In particular, $G$ is connected and solvable. 
By the slice method we have $U/{\aut}({\F}_2)\sim \varphi^{-1}((p, q), F_2)/G$. 
The fiber $\varphi^{-1}((p, q), F_2)$ is an open set of 
the linear system ${\proj}W\subset|L_{3,0}|$ of curves $\Gamma$ with $\Gamma|_{F_1}=2p+q$. 
By Miyata's theorem the quotient ${\proj}W/G$ is rational. 
\end{proof}

\subsection{$\mathcal{M}_{12,8,1}$ and quadric del Pezzo surfaces}\label{ssec:(12,8)}

\subsubsection{A period map}

Let us begin with few remarks. 
Let $Y$ be a quadric del Pezzo surface and let $\phi\colon Y\to{\proj}^2$ be its anticanonical map, 
which is branched along a smooth quartic $\Gamma$. 
The pullback $\phi^{\ast}L$ of a line $L\subset{\proj}^2$ is singular if and only if $L$ is tangent to $\Gamma$. 
Therefore the reduced curve $\phi^{-1}(\Gamma)$, the fixed locus of the Geiser involution, 
is the locus of singular points of anticanonical curves on $Y$. 
For every $p\in\phi^{-1}(\Gamma)$ we have a unique singular $-K_Y$-curve $C_p$ with $p\in{\rm Sing}(C_p)$, 
which is the pullback of the tangent line $L_p$ of $\Gamma$ at $\phi(p)$. 
The singular curve $C_p$ is irreducible and nodal if and only if $L_p$ is an ordinary tangent line. 
In this case, the Geiser involution exchanges the two tangents of $C_p$ at $p$. 
On the other hand, $C_p$ has only one tangent at $p$ if and only if 
$\phi(p)$ is an inflectional point, i.e., a Weierstrass point of $\Gamma$. 

Now let $f\colon\mathcal{Y}^7\to\mathcal{U}^7$ be the family of marked quadric del Pezzo surfaces 
constructed in \S \ref{ssec:equiv Weyl}. 
The Cremona representation of the Weyl group $W_7$ has kernel ${\Z}/2$ 
whose generator $w_0$ acts on the $f$-fibers by the Geiser involutions. 
The quotient $W_7/\langle w_0\rangle$ is isomorphic to ${\rm Sp}(6, 2)$ and acts on $\mathcal{U}^7$ almost freely. 

Let $\mathcal{C}\subset \mathcal{Y}^7$ be the fixed locus of $w_0$. 
We define a double cover $\widetilde{\mathcal{C}}$ of $\mathcal{C}$ 
as the locus in $({\proj}T_f)|_{\mathcal{C}}$ of triplets $(\mathbf{p}, p, v)$ such that 
$v\in{\proj}(T_p\mathcal{Y}^7_{\mathbf{p}})$ is a tangent at $p$ of 
the anticanonical curve on $\mathcal{Y}^7_{\mathbf{p}}$ singular at $p$. 
The locus $\widetilde{\mathcal{C}}$ is invariant under the $W_7$-action on $({\proj}T_f)|_{\mathcal{C}}$. 
In particular, the Geiser involution $w_0$ acts on $\widetilde{\mathcal{C}}$ by the covering transformation of 
$\widetilde{\mathcal{C}}\to \mathcal{C}$. 
The branch divisor of $\widetilde{\mathcal{C}}\to\mathcal{C}$ is the family of the Weierstrass points. 

We pull back the bundle $f_{\ast}K_f^{-1}$ on $\mathcal{U}^7$ by the projection 
$\widetilde{\mathcal{C}}\to \mathcal{U}^7$, 
and consider its subbundle $\mathcal{E}$ whose fiber over a $(\mathbf{p}, p, v)\in\widetilde{\mathcal{C}}$ 
is the vector space of anticanonical forms on $\mathcal{Y}^7_{\mathbf{p}}$ vanishing at $p$. 
An open set of the variety ${\proj}\mathcal{E}$ parametrizes quadruples 
$(\mathcal{Y}^7_{\mathbf{p}}, C_1, v, C_2)$ such that 
$\mathcal{Y}^7_{\mathbf{p}}=Y$ is a marked quadric del Pezzo surface, 
$C_1$ is an irreducible nodal $-K_Y$-curve, 
$v$ is one of the tangents of $C_1$ at the node, and 
$C_2$ is a smooth $-K_Y$-curve passing through the node of $C_1$. 
Then $C_1+C_2$ is a $-2K_Y$-curve with the $D_4$-singularity $C_1\cap C_2$ and with no other singularity. 
Its two tangents given by $C_1$ at the $D_4$-point are distinguished by $v$. 
The 2-elementary $K3$ surface associated to the DPN pair $(Y, C_1+C_2)$ has invariant $(g, k)=(1, 2)$. 
Thus we obtain a period map $\mathcal{P}\colon{\proj}\mathcal{E}\dashrightarrow\mathcal{M}_{12,8,1}$. 

\begin{proposition}\label{birat (12,8)}
The map $\mathcal{P}$ is $W_7$-invariant 
and descends to a birational map ${\proj}\mathcal{E}/W_7\dashrightarrow\mathcal{M}_{12,8,1}$. 
\end{proposition}

\begin{proof}
The $W_7$-invariance of $\mathcal{P}$ is straightforward.  
The marking $\varphi_{\mathbf{p}}$ of ${\rm Pic}(\mathcal{Y}^7_{\mathbf{p}})$ 
and the labeling $v$ for the $D_4$-point $C_1\cap C_2$ induce a marking of 
the invariant lattice of $(X, \iota)=\mathcal{P}(\mathbf{p}, C_1, v, C_2)$. 
This defines a lift ${\lift}\colon{\proj}\mathcal{E} \dashrightarrow {\cove}_{12,8,1}$ of $\mathcal{P}$. 
As in \S \ref{ssec:(11,9)}, 
the blow-down $\pi\colon\mathcal{Y}^7_{\mathbf{p}}\to {\proj}^2$ determined by $\mathbf{p}$ 
translates the quadruple $(\mathbf{p}, C_1, v, C_2)$ into 
the plane sextic $\pi(C_1+C_2)$ endowed with a labeling of its seven nodes, $\mathbf{p}$,  
and of the two tangents at its $D_4$-point $\pi(C_1\cap C_2)$ given by $\pi(C_1)$. 
Applying the recipe in \S \ref{ssec: recipe} for such sextics with labelings, we see that ${\lift}$ is birational. 
The projection ${\cove}_{12,8,1}\dashrightarrow\mathcal{M}_{12,8,1}$ 
is an ${\Or}(D_{L_-})$-covering for the lattice $L_-=U\oplus U(2)\oplus A_1^6$. 
By \cite{M-S} we have $|{\Or}(D_{L_-})|=2\cdot|{\rm Sp}(6, 2)|=|W_7|$. 
Since $W_7$ acts on ${\proj}\mathcal{E}$ almost freely, 
this concludes that the induced map 
${\proj}\mathcal{E}/W_7\dashrightarrow\mathcal{M}_{12,8,1}$ 
is birational.  
\end{proof}

\subsubsection{The rationality} 

We shall prove that ${\proj}\mathcal{E}/W_7$ is rational. 
The Weyl group $W_7$ acts on the markings of del Pezzo surfaces, 
and the Geiser involution $w_0\in W_7$ is the covering transformation of $\widetilde{\mathcal{C}} \to \mathcal{C}$. 
These facts infer that ${\proj}\mathcal{E}/W_7$ is birationally 
a moduli space of triplets $(Y, p, C)$ where 
$Y$ is an (unmarked) quadric del Pezzo surface, $p\in Y$ is a fixed point of the Geiser involution, 
and $C$ is a $-K_Y$-curve passing through $p$. 
The anticanonical map $Y\to{\proj}^2$ translates $(Y, p, C)$ into the triplet $(\Gamma, q, L)$ such that 
$\Gamma$ is a smooth plane quartic, $q\in\Gamma$, and $L$ is a line through $q$. 
If $U\subset|{\Oplane}(4)|\times{\proj}^2\times|{\Oplane}(1)|$ denotes the space of such triplets $(\Gamma, q, L)$, 
we obtain a birational equivalence 
\begin{equation*}
{\proj}\mathcal{E}/W_7 \sim U/{\PGL}_3.  
\end{equation*} 
In Proposition \ref{rational (8+k,8-k)} (for $k=4$) we proved that this quotient $U/{\PGL}_3$ is rational. 
Therefore 

\begin{proposition}\label{rational (12,8)}
The moduli space $\mathcal{M}_{12,8,1}$ is rational. 
\end{proposition} 

\label{2nd map}
By the final step of the proof, we have a natural birational map 
$\mathcal{M}_{12,8,1}\dashrightarrow\mathcal{M}_{12,4,1}$ via $U/{\PGL}_3$. 
On the other hand, the referee suggested that we have another birational map 
$\mathcal{M}_{12,8,1}\dashrightarrow\mathcal{M}_{12,4,1}$ as follows (cf.\! Appendix). 
For the anti-invariant lattice $L_-=U\oplus U(2)\oplus A_1^6$ for $\mathcal{M}_{12,8,1}$, 
consider the odd lattice $L_-^{\vee}(2)=U(2)\oplus U\oplus \langle-1\rangle^6$. 
Its even part $M$ is isometric to $U(2)\oplus U\oplus D_6$ by Example \ref{I_n&D_n}. 
One checks that the element in $D_M$ giving the extension 
$M\subset L_-^{\vee}(2)$ is invariant under ${\Or}(D_M)$. 
Therefore we have an isomorphism 
$\mathcal{F}({\Or}(L_-))\simeq\mathcal{F}({\Or}(M))$ of modular varieties by Proposition \ref{even part}, 
and $\mathcal{F}({\Or}(M))$ is birational to $\mathcal{M}_{12,4,1}$. 

It would be interesting whether these two maps 
$\mathcal{M}_{12,8,1}\dashrightarrow\mathcal{M}_{12,4,1}$ coincide. 
Anyway, the second construction and the rationality of $\mathcal{M}_{12,4,1}$ (\S \ref{ssec:(8+k,8-k)}) 
provide another proof of the rationality of $\mathcal{M}_{12,8,1}$.

\subsection{$\mathcal{M}_{13,7,1}$ and cubic surfaces}\label{ssec:(13,7)}

\subsubsection{A period map}\label{sssec:(13,7)}

We begin with the remark that 
for every point $p$ of a cubic del Pezzo surface $Y$, 
there uniquely exists a $-K_Y$-curve $C_p$ singular at $p$. 
Indeed, if we embed $Y$ in ${\proj}^3$ naturally, 
$C_p$ is the intersection of $Y$ with the tangent plane of $Y$ at $p$. 
When $p$ is generic, $C_p$ is irreducible and nodal. 
For later use, we also explain an alternative construction. 
The blow-up $\widehat{Y}$ of $Y$ at a general $p\in Y$ is a quadric del Pezzo surface. 
Let $E\subset\widehat{Y}$ be the $(-1)$-curve over $p$, 
and let $\phi\colon\widehat{Y}\to{\proj}^2$ be the anticanonical map of $\widehat{Y}$ 
which is branched over a smooth quartic $\Gamma$. 
Then $\phi(E)$ is a bitangent of $\Gamma$, 
and $\phi^{\ast}\phi(E)=E+\iota(E)$ where $\iota$ is the Geiser involution of $\widehat{Y}$. 
Since $\phi^{\ast}\phi(E)$ is a $-K_{\widehat{Y}}$-curve, 
the image of $\iota(E)$ in $Y$ gives the desired curve $C_p$. 
The fact that $\phi$ and $\iota$ are given by the projection $Y\subset{\proj}^3\dashrightarrow{\proj}^2$ 
from $p$ connects these two constructions. 


Let $f\colon \mathcal{Y}^6 \to \mathcal{U}^6$ be the family of marked cubic surfaces 
constructed in \S \ref{ssec:equiv Weyl}. 
The Weyl group $W_6$ acts on $\mathcal{U}^6$ almost freely 
because a generic cubic surface has no automorphism. 
We shall denote a point of the variety ${\proj}T_f$ as $(\mathbf{p}, p, v)$ where 
$\mathbf{p}\in\mathcal{U}^6$, $p\in\mathcal{Y}^6_{\mathbf{p}}$, and $v\in{\proj}(T_p\mathcal{Y}^6_{\mathbf{p}})$. 
Then we let $\mathcal{Z}\subset{\proj}T_f$ be the locus of $(\mathbf{p}, p, v)$ such that 
$v$ is one of the tangents of the anticanonical curve $C_p$ 
singular at $p$. 
The locus $\mathcal{Z}$ is $W_6$-invariant and is a double cover of $\mathcal{Y}^6$. 
The branch divisor of $\mathcal{Z}\to\mathcal{Y}^6$ is the family of the Hessian quartics 
restricted to the marked cubic surfaces. 

We pull-back the vector bundle $f_{\ast}K_f^{-1}$ on $\mathcal{U}^6$ by the projection $\mathcal{Z} \to \mathcal{U}^6$, 
and consider its subbundle $\mathcal{E}$ whose fiber over a $(\mathbf{p}, p, v)\in \mathcal{Z}$ 
is the space of anticanonical forms on $\mathcal{Y}^6_{\mathbf{p}}$ which vanish at $p$ and 
whose first derivatives at $p$ vanish at $v$. 
Then $\mathcal{E}$ is a $W_6$-linearized vector bundle over $\mathcal{Z}$. 
An open set of the variety ${\proj}\mathcal{E}$ parametrizes triplets 
$(\mathcal{Y}^6_{\mathbf{p}}, p, C)$ where 
$\mathcal{Y}^6_{\mathbf{p}}=Y$ is a marked cubic surface, 
$p\in Y$ is such that the singular $-K_Y$-curve $C_p$ is irreducible and nodal,  
and $C$ is a smooth $-K_Y$-curve with $(C.C_p)_p=3$. 
By $v$ is chosen which branch of $C_p$ at $p$ is tangent to $C$. 
The curve $C_p+C$ has the $D_6$-singularity $p$ as its only singularity. 
Thus, taking the right resolution of the DPN pairs $(Y, C_p+C)$, 
we obtain a period map $\mathcal{P}\colon {\proj}\mathcal{E} \dashrightarrow \mathcal{M}_{13,7,1}$. 

\begin{proposition}\label{birat (13,7)}
The map $\mathcal{P}$ is $W_6$-invariant and descends to 
a birational map ${\proj}\mathcal{E}/W_6\dashrightarrow\mathcal{M}_{13,7,1}$. 
\end{proposition}

\begin{proof}
This is analogous to Propositions \ref{birat (11,9)} and \ref{birat (12,8)}: 
for a triplet $(\mathbf{p}, C_p, C)$ in ${\proj}\mathcal{E}$, 
the blow-down $\pi\colon \mathcal{Y}^6_{\mathbf{p}} \to {\proj}^2$ determined by $\mathbf{p}$ 
translates $(\mathbf{p}, C_p, C)$ into the plane sextic $\pi(C_p+C)$ 
endowed with the labeling $\mathbf{p}$ of its six nodes. 
The rest singularity of $\pi(C_p+C)$ is the $D_6$-point $\pi(p)$, 
at which the three branches of $\pi(C_p+C)$ are a priori distinguished 
by the irreducible decomposition of $\pi(C_p+C)$ and by the intersection multiplicity at $\pi(p)$.   
By the recipe in \S \ref{ssec: recipe} we see that    
$\mathcal{P}$ lifts to a birational map ${\proj}\mathcal{E}\dashrightarrow{\cove}_{13,7,1}$. 
The variety ${\cove}_{13,7,1}$ is an ${\Or}(D_{L_-})$-cover of $\mathcal{M}_{13,7,1}$ 
for the lattice $L_-=U\oplus U(2)\oplus A_1^5$. 
Since ${\Or}(D_{L_-})\simeq{\Or}^-(6, 2)\simeq W_6$, 
our assertion is proved. 
\end{proof}

\subsubsection{The rationality} 

We shall prove that ${\proj}\mathcal{E}/W_6$ is rational. 
First we apply the no-name method to the $W_6$-linearized bundle $\mathcal{E}$ 
over $\mathcal{Z}$ to see that 
\begin{equation*}
{\proj}\mathcal{E}/W_6 \sim {\proj}^1\times(\mathcal{Z}/W_6). 
\end{equation*}
The variety $\mathcal{Z}$ is a moduli of triplets $(\mathcal{Y}^6_{\mathbf{p}}, p, v)$ 
where $\mathcal{Y}^6_{\mathbf{p}}=Y$ is a marked cubic surface, $p\in Y$, 
and $v\in{\proj}(T_pY)$ is one of the tangents of the singular $-K_Y$-curve $C_p$. 
The $W_6$-action takes off the markings of surfaces. 
Let $\widehat{Y}\to Y$ be the blow-up at $p$, and let $E$ be the $(-1)$-curve over $p$. 
When $p$ is generic, $\widehat{Y}$ is a quadric del Pezzo surface.  
As explained in the beginning of \S \ref{sssec:(13,7)}, 
if we regard $v$ as a point $q$ of $E$, 
then $q$ is contained in $E\cap\iota(E)$ where $\iota$ is the Geiser involution of $\widehat{Y}$. 
By applying the anticanonical map $\widehat{Y}\to{\proj}^2$ of $\widehat{Y}$, 
the new triplet $(\widehat{Y}, E, q)$ is then translated into a triplet $(\Gamma, L, P)$ of 
a smooth quartic $\Gamma$, a bitangent $L$ of $\Gamma$, and a point $P$ in $\Gamma\cap L$. 
Therefore, if $U\subset|{\Oplane}(4)|\times{\proj}^2$ denotes the space of pairs $(\Gamma, P)$ of 
a smooth quartic $\Gamma$ and a tangent point $P$ of a bitangent of $\Gamma$, 
we have a rational map 
\begin{equation}\label{transform (13,7)}
\mathcal{Z}/W_6 \dashrightarrow U/{\PGL}_3. 
\end{equation}
Conversely, for a point $(\Gamma, P)$ of $U$, 
we take the double cover $\phi\colon\widehat{Y}\to{\proj}^2$ branched along $\Gamma$. 
If $L$ is the tangent line of $\Gamma$ at $P$, we have $\phi^{\ast}L=E+\iota(E)$ 
for a $(-1)$-curve $E$ and the Geiser involution $\iota$ of $\widehat{Y}$. 
Since the triplet $(\widehat{Y}, E, \phi^{-1}(P))$ is isomorphic to $(\widehat{Y}, \iota(E), \phi^{-1}(P))$ by $\iota$, 
we may contract either $E$ or $\iota(E)$ to obtain a well-defined inverse map of \eqref{transform (13,7)}. 
Thus $\mathcal{Z}/W_6$ is birational to $U/{\PGL}_3$. 

\begin{proposition}\label{rational (13,7)}
The quotient $U/{\PGL}_3$ is rational. 
Hence $\mathcal{M}_{13,7,1}$ is rational. 
\end{proposition}

\begin{proof}
For a point $(\Gamma, P)$ of $U$, 
let $Q$ be the another tangent point of the bitangent at $P$. 
This defines a ${\PGL}_3$-equivariant map 
$\varphi\colon U\to{\proj}^2\times{\proj}^2$, $(\Gamma, P)\mapsto(P, Q)$.  
The ${\PGL}_3$-action on ${\proj}^2\times{\proj}^2$ is almost transitive 
with the stabilizer $G$ of a general $(P, Q)$ isomorphic to $({\C}^{\times}\ltimes{\C})^2$. 
By the slice method we have $U/{\PGL}_3 \sim \varphi^{-1}(P, Q)/G$. 
The fiber $\varphi^{-1}(P, Q)$ is an open set of 
the linear system ${\proj}V$ of quartics $\Gamma$ with $\Gamma|_{\overline{PQ}}=2P+2Q$.  
By Miyata's theorem ${\proj}V/G$ is rational. 
\end{proof}

\subsection{$\mathcal{M}_{14,6,1}$ and quartic del Pezzo surfaces}\label{ssec:(14,6)}

\subsubsection{A period map}

We first note the following. 

\begin{lemma}\label{nodal -K curve (14,1)}
Let $Y$ be a quartic del Pezzo surface and let $(p, v)$ be a general point of ${\proj}TY$, 
where $p\in Y$ and $v\in{\proj}(T_pY)$. 
Then there uniquely exists an irreducible nodal $-K_Y$-curve $C_{p,v}$ 
whose node is $p$ with one of the tangents being $v$. 
\end{lemma}

\begin{proof}
Let $Y'\to Y$ be the blow-up at $p$ with the exceptional curve $D'$, 
and $\widehat{Y}\to Y'$ be the blow-up at $v\in D'$ with the exceptional curve $E$. 
Then $\widehat{Y}$ is the blow-up of ${\proj}^2$ at seven points in almost general position in the sense of \cite{De}, 
and the strict transform $D$ of $D'$ is the unique $(-2)$-curve on $\widehat{Y}$. 
Hence the anticanonical map $\phi\colon\widehat{Y}\to{\proj}^2$ is the composition of 
the contraction $\widehat{Y}\to\widehat{Y}_0$ of $D$ and a double covering $\widehat{Y}_0\to{\proj}^2$ 
branched along a quartic $\Gamma$ with exactly one node. 
The curve $L=\phi(E)$ is a line passing through the node and tangent to $\Gamma$ elsewhere. 
We have $\phi^{\ast}L=E+\iota(E)+D$ where $\iota$ is the involution of $\widehat{Y}$ 
induced by the covering transformation of $\widehat{Y}_0\to{\proj}^2$. 
Then the image of $\iota(E)$ in $Y$ is the desired curve $C_{p,v}$. 
The uniqueness of $C_{p,v}$ follows from intersection calculation. 
\end{proof}

Pulling back the pencil of lines through $L\cap\Gamma\backslash{\rm Sing}(\Gamma)$, 
we obtain the pencil of $-K_{\widehat{Y}}$-curves through $E\cap\iota(E)$. 
Then its image in $|\!-\!K_Y|$ is the pencil $l_{p,v}$ of $-K_Y$-curves 
whose general member $C$ is smooth and passes $p$ and $v$ with $(C.C_{p,v})_p=4$. 
(There is another pencil of $-K_Y$-curves passing through $(p, v)$ with $(C.C_{p,v})_p=4$: 
the one of those singular at $p$.)

Now let $f\colon\mathcal{Y}^5\to\mathcal{U}^5$ be the family of marked quartic del Pezzo surfaces 
constructed in \S \ref{ssec:equiv Weyl}. 
The kernel of the Cremona representation of the Weyl group $W_5$ is isomorphic to $({\Z}/2)^4$, 
which is the automorphism group of a general $f$-fiber. 
Let ${\proj}T_f^0\subset{\proj}T_f$ be the open set of triplets $(\mathbf{p}, p, v)$ such that 
there exists an anticanonical curve $C_{p,v}$ on $\mathcal{Y}^5_{\mathbf{p}}$ 
as in Lemma \ref{nodal -K curve (14,1)}. 

Let $\mathcal{F}$ be the pullback of the bundle $f_{\ast}K_f^{-1}$ 
by the natural projection ${\proj}T_f^0\to\mathcal{U}^5$. 
We consider the subbundle $\mathcal{E}$ of $\mathcal{F}$ 
such that the fiber of ${\proj}\mathcal{E}$ over a $(\mathbf{p}, p, v)$ is the pencil $l_{p,v}$ described above. 
By the uniqueness of $C_{p,v}$, $\mathcal{E}$ is invariant under the $W_5$-action on $\mathcal{F}$. 
An open set of ${\proj}\mathcal{E}$ parametrizes triplets 
$(\mathcal{Y}^5_{\mathbf{p}}, C_{p,v}, C)$ such that 
$\mathcal{Y}^5_{\mathbf{p}}=Y$ is a marked quartic del Pezzo surface, 
$C_{p,v}$ is an irreducible nodal $-K_Y$-curve, 
and $C$ is a smooth $-K_Y$-curve with $(C.C_{p,v})_p=4$. 
The $-2K_Y$-curve $C_{p,v}+C$ has the $D_8$-point $p$ as its only singularity. 
The 2-elementary $K3$ surface $(X, \iota)$ associated to the DPN pair $(Y, C_{p,v}+C)$ has invariant $(g, k)=(1, 4)$. 
We show that $(X, \iota)$ has parity $\delta=1$. 
If $\pi\colon Y\to{\proj}^2$ is the blow-down given by $\mathbf{p}$, 
the sextic $\pi(C_{p,v}+C)$ has the five nodes $\mathbf{p}$, the $D_8$-point $\pi(p)$, and no other singularity. 
Then we may apply Lemma \ref{delta=1} $(3)$. 
Thus we obtain a period map 
$\mathcal{P}\colon{\proj}\mathcal{E}\dashrightarrow\mathcal{M}_{14,6,1}$. 

\begin{proposition}\label{birat (14,6,1)}
The period map $\mathcal{P}$ is $W_5$-invariant and descends to a birational map 
${\proj}\mathcal{E}/W_5\dashrightarrow\mathcal{M}_{14,6,1}$. 
\end{proposition}

\begin{proof}
This is proved in the same way as Propositions \ref{birat (11,9)}, \ref{birat (12,8)}, and \ref{birat (13,7)}. 
The projection ${\cove}_{14,6,1}\dashrightarrow\mathcal{M}_{14,6,1}$ has degree $|{\Or}(D_{L_-})|$ 
for the lattice $L_-=U\oplus U(2)\oplus A_1^4$, 
which by \cite{M-S} is equal to $2^4\cdot|{\Or}^-(4, 2)|=2^4\cdot5!=|W_5|$. 
\end{proof}

\subsubsection{The rationality} 

We shall prove that ${\proj}\mathcal{E}/W_5$ is rational. 
By the no-name method for the $W_5$-linearized bundle $\mathcal{E}$ over ${\proj}T_f^0$, we have 
\begin{equation*}
{\proj}\mathcal{E}/W_5 \sim {\proj}^1\times({\proj}T_f/W_5). 
\end{equation*}
The variety ${\proj}T_f/W_5$ is the moduli of triplets $(Y, p, v)$ 
where $Y$ is an (unmarked) quartic del Pezzo surface, $p\in Y$, and $v\in{\proj}(T_pY)$. 
As in the proof of Lemma \ref{nodal -K curve (14,1)}, 
for a general $(Y, p, v)$ we blow-up $Y$ at $p$ and $v$ and then apply the anticanonical map 
to obtain a one-nodal quartic $\Gamma$ 
and a line $L$ passing through the node of $\Gamma$ and tangent to $\Gamma$ elsewhere. 
If $U\subset|{\Oplane}(4)|\times|{\Oplane}(1)|$ is the space of such pairs $(\Gamma, L)$, 
we thus obtain a rational map 
\begin{equation}\label{rat map (14,6,1)}
{\proj}T_f/W_5 \dashrightarrow U/{\PGL}_3. 
\end{equation} 
Conversely, given a $(\Gamma, L)\in U$, 
we take the double cover $\widehat{Y}_0\to{\proj}^2$ branched along $\Gamma$ 
and then the minimal desingularization $\widehat{Y}\to\widehat{Y}_0$. 
The covering transformation of $\widehat{Y}_0\to{\proj}^2$ induces an involution $\iota$ of $\widehat{Y}$. 
The pullback of $L$ to $\widehat{Y}$ is written as $E+\iota(E)+D$ 
where $D$ is the $(-2)$-curve over the double point of $\widehat{Y}_0$, 
and $E$ is a $(-1)$-curve with $(E.D)=(E.\iota(E))=1$. 
Notice that $\iota$ exchanges $E$ and $\iota(E)$, and leaves $D$ invariant. 
Then we contract $E$ (or equivalently, $\iota(E)$) and $D$ successively 
to obtain a point $(Y, p, v)$ of ${\proj}T_f/W_5$. 
This gives a well-defined inverse map of \eqref{rat map (14,6,1)}, 
and thus ${\proj}T_f/W_5$ is birational to $U/{\PGL}_3$. 

\begin{proposition}\label{rational (14,6,1)}
The quotient $U/{\PGL}_3$ is rational. 
Hence $\mathcal{M}_{14,6,1}$ is rational. 
\end{proposition}

\begin{proof}
We have the ${\PGL}_3$-equivariant morphism 
\begin{equation*}
\varphi:U\to{\proj}^2\times{\proj}^2, \qquad 
(\Gamma, L)\mapsto({\rm Sing}(\Gamma), L\cap\Gamma\backslash{\rm Sing}(\Gamma)). 
\end{equation*}
As in the proof of Proposition \ref{rational (13,7)}, 
we have $U/{\PGL}_3\sim \varphi^{-1}(P, Q)/({\C}^{\times}\ltimes{\C})^2$ 
for a general point $(P, Q)\in{\proj}^2\times{\proj}^2$. 
Then $\varphi^{-1}(P, Q)$ is identified with an open set of a linear system of quartics, 
so that our assertion follows from Miyata's theorem.  
\end{proof}

\begin{remark}
Del Pezzo surfaces in this section are not the quotient surfaces of 2-elementary $K3$ surfaces, 
but rather their "canonical" blow-down --- 
compare with \S \ref{sec:k=0} where del Pezzo surfaces appear as the quotient surfaces. 
These two series are related by mirror symmetry for lattice-polarized $K3$ surfaces. 
\end{remark}

\begin{remark}
The Weyl group symmetry translated to the moduli of labelled sextics $C_1+C_2$ 
is generated by the renumbering of labelings and by 
quadratic transformations based at ordinary intersection points of $C_1$ and $C_2$ 
(cf. Remark \ref{second proof (10,8,1)}). 
\end{remark}


\section{The case $g=1$ (II)}\label{sec:g=1 (2)} 

In this section we treat the case $g=1$, $k\geq4$, $(k, \delta)\ne(4,1)$. 
By the unirationality result \cite{Ma} we may assume $k\leq7$. 
For $k\geq5$ we use plane cubics with a chosen inflectional point to construct birational period maps. 
Together with \S \ref{sec:g=1 (1)}, we now cover the range $g=1$, $k>0$. 
Among the rest two, $\mathcal{M}_{10,10,0}$ is the moduli of Enriques surfaces and is rational by Kond\=o \cite{Ko}, 
while $\mathcal{M}_{10,10,1}$ will be treated in Appendix.

\subsection{The rationality of $\mathcal{M}_{14,6,0}$}

The space $\mathcal{M}_{14,6,0}$ is proven to be rational in the same way as for $\mathcal{M}_{10,8,0}$. 
The anti-invariant lattice $L_-$ for $\mathcal{M}_{14,6,0}$ is isometric to $U(2)^2\oplus D_4$. 
Then we have $L_-^{\vee}(2)\simeq U^2\oplus D_4$, 
so that $\mathcal{M}_{14,6,0}$ is birational to $\mathcal{M}_{14,2,0}$. 
In \S \ref{ssec:(8+k,8-k)} we proved that $\mathcal{M}_{14,2,0}$ is rational.

\subsection{The rationality of $\mathcal{M}_{15,5,1}$}

Let $U\subset|{\Oplane}(3)|\times|{\Oplane}(2)|\times{\proj}^2$ be the locus of triplets $(C, Q, p)$ such that 
$({\rm i})$ $C$ is smooth,  
$({\rm ii})$ $p$ is an inflectional point of $C$, and 
$({\rm iii})$ $Q$ is smooth, passes $p$, and is transverse to $C$.  
If $L\subset{\proj}^2$ is the tangent line of $C$ at $p$, 
the sextic $B=C+Q+L$ has a $D_8$-singularity at $p$, 
five nodes at $C\cap Q\backslash p$, one node at $Q\cap L\backslash p$, 
and no other singularity. 
The 2-elementary $K3$ surface associated to the sextic $B$ has invariant $(g, k)=(1, 5)$. 
Thus we obtain a period map $\mathcal{P}\colon U/{\PGL}_3 \to \mathcal{M}_{15,5,1}$. 

\begin{proposition}\label{birat (15,5)} 
The period map $\mathcal{P}$ is birational. 
\end{proposition}

\begin{proof}
We consider an $\frak{S}_5$-cover $\widetilde{U}$ of $U$ 
whose fiber over a $(C, Q, p)\in U$ corresponds to labelings of the five nodes $C\cap Q\backslash p$. 
As in the previous sections, 
$\mathcal{P}$ lifts to a birational map 
$\widetilde{U}/{\PGL}_3\dashrightarrow{\cove}_{15,5,1}$. 
The variety ${\cove}_{15,5,1}$ is an ${\Or}(D_{L_-})$-cover of $\mathcal{M}_{15,5,1}$ 
for the lattice $L_-=U\oplus U(2)\oplus A_1^3$. 
Since ${\PGL}_3$ acts on $U$ almost freely 
and since $|{\Or}(D_{L_-})|=|{\Or}^-(4, 2)|=5!$ by \cite{M-S}, 
$\mathcal{P}$ has degree $1$. 
\end{proof}

\begin{proposition}\label{rational (15,5,1)}
The quotient $U/{\PGL}_3$ is rational. 
Hence $\mathcal{M}_{15,5,1}$ is rational. 
\end{proposition}

\begin{proof}
Let $V\subset |{\Oplane}(2)|\times|{\Oplane}(1)|\times{\proj}^2$ be the locus of triplets $(Q, L, p)$ 
such that $p\in Q\cap L$. 
We have a ${\PGL}_3$-equivariant map 
$\pi\colon U\to V$ defined by $(C, Q, p)\mapsto(Q, L, p)$, where $L$ is the tangent line of $C$ at $p$. 
A general $\pi$-fiber is an open set of a linear system ${\proj}W$ of cubics. 
The group ${\PGL}_3$ acts on $V$ almost transitively 
with the stabilizer $G$ of a general point $(Q, L, p)$ isomorphic to ${\C}^{\times}$. 
Indeed, since $Q$ is anti-canonically embedded, 
$G$ is  identified with the group of automorphisms of $Q\simeq{\proj}^1$ fixing the two points $Q\cap L$. 
By the slice method we have $U/{\PGL}_3\sim {\proj}W/G$, and ${\proj}W/G$ is clearly rational. 
\end{proof}

\subsection{The rationality of $\mathcal{M}_{16,4,1}$ and $\mathcal{M}_{17,3,1}$}

For $k=6, 7$, let $U_k\subset|{\Oplane}(3)|\times|{\Oplane}(1)|^3$ be the locus of 
quadruplets $(C, L_1, L_2, L_3)$ such that 
$({\rm i})$ $C$ is smooth, 
$({\rm ii})$ $L_1, L_2, L_3$ are linearly independent, 
$({\rm iii})$ $p=L_1\cap L_2$ is an inflectional point of $C$ with tangent line $L_1$, 
$({\rm iv})$ $L_2, L_3$ are respectively transverse to $C$, and 
$({\rm v})$ $C$ passes (resp. does not pass) $L_2\cap L_3$ for $k=7$ (resp. $k=6$). 
When $k=6$, the singularities of the sextic $B=C+\sum_{i=1}^3L_i$ are 
\begin{equation*}
{\rm Sing}(B)= p + (L_2\cap C\backslash p) + (L_3\cap C) + (L_3\cap L_1) +(L_3\cap L_2), 
\end{equation*}
where $p$ is a $D_8$-singularity and the rest points are nodes. 
When $k=7$, denoting $q=L_2\cap L_3$, we have 
\begin{equation*}
{\rm Sing}(B)= p + q + (L_2\cap C\backslash \{ p, q\}) + (L_3\cap C\backslash q) + (L_3\cap L_1), 
\end{equation*}
where $p$ is a $D_8$-singularity, $q$ is a $D_4$-singularity, and the rest points are nodes. 
The 2-elementary $K3$ surface associated to $B$ has invariant $(r, a)=(10+k, 10-k)$, 
and we obtain a period map 
$\mathcal{P}_k\colon U_k/{\PGL}_3\to\mathcal{M}_{10+k,10-k,1}$. 

\begin{proposition}\label{birat (16,4)&(17,3)}
The map $\mathcal{P}_k$ is birational. 
\end{proposition}

\begin{proof}
For $k=6$ we label the three nodes $L_3\cap C$ and the two nodes $L_2\cap C\backslash p$ independently, 
which is realized by an $\frak{S}_3\times\frak{S}_2$-cover $\widetilde{U}_6$ of $U_6$. 
For $k=7$ we distinguish the two nodes $L_3\cap C\backslash q$, 
which defines a double cover $\widetilde{U}_7$ of $U_7$. 
Note that in both cases, the three lines $L_i$ are distinguished by their intersection properties, 
and hence the rest nodes are a priori labelled. 
As before, we see that $\mathcal{P}_k$ lifts to a birational map 
$\widetilde{U}_k/{\PGL}_3\dashrightarrow{\cove}_{10+k,10-k,1}$. 
Then ${\cove}_{10+k,10-k,1}$ is an ${\Or}(D_{L_-})$-cover of $\mathcal{M}_{10+k,10-k,1}$ 
for the lattice $L_-=U\oplus U(2)\oplus A_1^{8-k}$. 
By \cite{M-S} we have $|{\Or}(D_{L_-})|=12, 2$ for $k=6, 7$ respectively, which concludes the proof. 
\end{proof}

\begin{proposition}\label{rational (16,4)&(17,3)}
The quotient $U_k/{\PGL}_3$ is rational. 
Therefore $\mathcal{M}_{16,4,1}$ and $\mathcal{M}_{17,3,1}$ are rational. 
\end{proposition}

\begin{proof}
This follows from the slice method for the projection $\pi\colon U_k\to|{\Oplane}(1)|^3$. 
General $\pi$-fibers are open subsets of linear systems of cubics, 
and ${\PGL}_3$ acts on $|{\Oplane}(1)|^3$ almost transitively with a general stabilizer  
isomorphic to $({\C}^{\times})^2$. 
\end{proof}

\begin{remark}
Degenerating the above sextic models, one sees that 
\begin{enumerate}
\item  $\mathcal{M}_{18,2,1}$ is birational to the Kummer modular surface for ${\SL}_2({\Z})$, 
\item  $\mathcal{M}_{18,2,0}$ is birational to the pullback of the Kummer modular surface 
          for $\Gamma_0(2)$ by the Fricke involution,   
\item  $\mathcal{M}_{19,1,1}$ is birational to the modular curve for $\Gamma_0(3)$, 
\end{enumerate}
via their fixed curve maps. 
Note that $\mathcal{M}_{18,2,1}$ and $\mathcal{M}_{18,2,0}$ are Heegner divisors of $\mathcal{M}_{17,3,1}$, 
though they look like boundary divisors of toroidal compactification. 
\end{remark}


\section{The case $g=0$}\label{sec:g=0}

In this section we study the case $g=0$ with $k\geq4$. 
Dolgachev and Kond\=o \cite{D-K} proved the rationality of $\mathcal{M}_{11,11,1}$, 
which is the moduli of Coble surfaces and for which $(g, k)=(0, 0)$. 
The remaining case $1\leq k\leq3$ will be studied in Appendix. 
Here we treat the case $k=4$ by a similar approach as in \S \ref{sec:g=1 (1)}, 
and the cases $k=5, 6$ using cuspidal plane cubics. 
The range $k\geq7$ is covered by the unirationality result \cite{Ma}.

\subsection{$\mathcal{M}_{15,7,1}$ and cubic surfaces}\label{ssec:(15,7)}

Let $f\colon\mathcal{Y}^6\to\mathcal{U}^6$ be the family of marked cubic surfaces considered in \S \ref{ssec:(13,7)}. 
Let $\mathcal{W}\subset\mathcal{Y}^6$ be the locus of those $(\mathbf{p}, p)$ such that 
the $-K_Y$-curve $C_p$ on $Y=\mathcal{Y}^6_{\mathbf{p}}$ singular at $p$ is irreducible and cuspidal. 
($\mathcal{W}$ is the branch divisor of the double cover $\mathcal{Z}\to\mathcal{Y}^6$ in \S \ref{ssec:(13,7)}). 
Then we consider the locus $\mathcal{V}$ in $\mathcal{W}\times_{\mathcal{U}^6}\mathcal{Y}^6$ of 
those $(\mathbf{p}, p, q)$ such that the $-K_Y$-curve $C_q$ singular at $q$ is irreducible, nodal, and 
intersects with $C_p$ at $p$ with multiplicity $3$. 
The curves $C_p$ and $C_q$ do not intersect outside $p$. 
Then the $-2K_Y$-curves $C_p+C_q$ have an $E_7$-singularity at $p$, a node at $q$, and no other singularity. 
Taking the right resolution of the DPN pairs $(Y, C_p+C_q)$, 
we obtain a period map $\mathcal{P}\colon\mathcal{V}\to\mathcal{M}_{15,7,1}$. 

\begin{proposition}\label{birat (15,7)}
The map $\mathcal{P}$ descends to a birational map $\mathcal{V}/W_6\to\mathcal{M}_{15,7,1}$. 
\end{proposition}

\begin{proof}
This is similar to the arguments in \S \ref{sec:g=1 (1)}, so we shall be brief. 
The natural blow-down $\pi\colon\mathcal{Y}^6_{\mathbf{p}}\to{\proj}^2$ determined by $\mathbf{p}$ 
maps $(C_p, C_q)$ to the irreducible cubic pair $(\pi(C_p), \pi(C_q))$ such that 
(i) $\pi(C_p)$ is cuspidal, 
(ii) $\pi(C_q)$ is nodal and intersects with $\pi(C_p)$ at its cusp with multiplicity $3$, and 
(iii) $\pi(C_p)$ and $\pi(C_q)$ are transverse outside the cusp of $\pi(C_p)$. 
The six nodes of $\pi(C_p)+\pi(C_q)$ are the blown-up points of $\pi$, 
and hence naturally ordered by $\mathbf{p}$. 
One recovers $(\mathbf{p}, p, q)$ from the labelled sextic $\pi(C_p)+\pi(C_q)$ by blowing-up the ordered nodes. 
Thus, if $\widetilde{U}\subset|{\Oplane}(3)|^2\times({\proj}^2)^6$ is the space of such labelled sextics, 
$\mathcal{V}$ is birationally identified with $\widetilde{U}/{\PGL}_3$. 
Then we can apply our familiar recipe to $\widetilde{U}$ to see that 
$\mathcal{P}$ lifts to a birational map $\mathcal{V}\to{\cove}_{15,7,1}$. 
The map $\mathcal{P}$ is clearly $W_6$-invariant. 
We also note that $W_6$ acts on $\mathcal{V}$ almost freely. 
On the other hand, ${\cove}_{15,7,1}\to\mathcal{M}_{15,7,1}$ has the Galois group ${\Or}(D_{L_-})$ 
for the lattice $L_-=\langle2\rangle^2\oplus\langle-2\rangle^5$. 
By \cite{M-S} we have $|{\Or}(D_{L_-})|=|{\Or}^-(6, 2)|=|W_6|$. 
Therefore the induced map $\mathcal{V}/W_6\to\mathcal{M}_{15,7,1}$ has degree $1$. 
\end{proof}

Let $V\subset|{\Ospace}(3)|\times({\proj}^3)^2$ be the locus of triplets $(Y, p, q)$ such that 
(i) $Y$ is a smooth cubic surface containing $p$ and $q$, 
(ii) the tangent plane section $T_pY|_Y$ is irreducible and cuspidal, and 
(iii) the tangent plane section $T_qY|_Y$ is irreducible, nodal, and intersects with $T_pY|_Y$ at $p$ with multiplicity $3$. 
Getting rid of markings and considering anti-canonical models, 
we see that $\mathcal{V}/W_6$ is birationally identified with $V/{\PGL}_4$. 

\begin{proposition}\label{rational (15,7)}
The quotient $V/{\PGL}_4$ is rational. 
Therefore $\mathcal{M}_{15,7,1}$ is rational. 
\end{proposition}

\begin{proof}
For $(Y, p, q)\in V$, the point $p$ is contained in both tangent planes $T_pY$, $T_qY$. 
Moreover, the common tangent line of $T_pY|_Y$ and $T_qY|_Y$ at $p$ is given by $T_pY\cap T_qY$. 
Thus, if we consider the space $W\subset({\proj}^3)^2\times|{\Ospace}(1)|^2$ of quadruplets $(p_1, p_2, H_1, H_2)$ 
such that $p_1\in H_1\cap H_2$ and $p_2\in H_2$, 
we have the ${\PGL}_4$-equivariant map 
\begin{equation*}
\varphi : V\to W, \qquad (Y, p, q)\mapsto(p, q, T_pY, T_qY). 
\end{equation*}
By our second remark, the $\varphi$-fiber over a general $(p_1, p_2, H_1, H_2)\in W$ is 
the locus $U\subset|{\Ospace}(3)|$ of smooth cubic surfaces $Y$ such that
(i) the plane cubic $Y|_{H_2}$ is nodal at $p_2$ and passes through $p_1$ with tangent line $H_1\cap H_2$, and 
(ii) the plane cubic $Y|_{H_1}$ is cuspidal at $p_1$ with tangent line $H_1\cap H_2$. 
Then $U$ is an open set of a linear system. 
It is straightforward to see that ${\SL}_4$ acts on $W$ almost transitively, 
with the stabilizer $G\subset{\SL}_4$ of $(p_1, p_2, H_1, H_2)$ being connected and solvable. 
By the slice method we have 
\begin{equation*}
V/{\SL}_4 \sim U/G, 
\end{equation*}
and the quotient $U/G$ is rational by Miyata's theorem. 
\end{proof}

\subsection{The rationality of $\mathcal{M}_{16,6,1}$}\label{ssec:(16,6)}

Let $V\subset|{\Oplane}(3)|$ be the variety of cuspidal plane cubics. 
Let $U\subset V\times|{\Oplane}(2)|$ be the open set of pairs $(C, Q)$ such that 
$Q$ is smooth and transverse to $C+L$, where $L$ is the tangent line of $C$ at the cusp. 
Then the sextic $B=C+Q+L$ has an $E_7$-singularity at the cusp of $C$, 
eight nodes at $C\cap Q$ and $L\cap Q$, and no other singularity.  
The associated 2-elementary $K3$ surface has invariant $(g, k)=(0, 5)$, 
and we obtain a period map $\mathcal{P}\colon U/{\PGL}_3\to\mathcal{M}_{16,6,1}$. 

\begin{proposition}\label{birat (16,6)}
The map $\mathcal{P}$ is birational. 
\end{proposition}

\begin{proof}
We label the six nodes $C\cap Q$ and the two nodes $L\cap Q$ independently. 
This defines an $\frak{S}_6\times\frak{S}_2$-cover $\widetilde{U}$ of $U$. 
By the familiar method we see that $\mathcal{P}$ lifts to a birational map 
$\widetilde{U}/{\PGL}_3\dashrightarrow{\cove}_{16,6,1}$. 
The variety ${\cove}_{16,6,1}$ is an ${\Or}(D_{L_-})$-cover of $\mathcal{M}_{16,6,1}$ 
for the lattice $L_-=U(2)^2\oplus A_1^2$. 
By \cite{M-S} we calculate $|{\Or}(D_{L_-})|=2\cdot|{\rm Sp}(4, 2)|=2\cdot6!$, 
which implies that ${\rm deg}(\mathcal{P})=1$. 
\end{proof}

\begin{proposition}\label{rational (16,6)}
The quotient $U/{\PGL}_3$ is rational. 
Hence $\mathcal{M}_{16,6,1}$ is rational. 
\end{proposition}

\begin{proof}
The slice method for the projection $U\to V$ shows that 
$U/{\PGL}_3\sim |{\Oplane}(2)|/G$, 
where $G$ is the stabilizer of a $C\in V$. 
Since $G\simeq{\C}^{\times}$, 
the quotient $|{\Oplane}(2)|/G$ is clearly rational. 
\end{proof}

\begin{remark}
General members of $\mathcal{M}_{16,6,1}$ can also be obtained from six general lines on ${\proj}^2$. 
Matsumoto-Sasaki-Yoshida \cite{M-S-Y} studied this sextic model, 
and showed that the period map is the quotient map by the association involution. 
Therefore $\mathcal{M}_{16,6,1}$ is birational to the moduli of double-sixers (cf. \cite{D-O}), 
which is rational by Coble \cite{Co} (see \cite{B-V} for a proof by Dolgachev). 
This is an alternative approach for the rationality of $\mathcal{M}_{16,6,1}$. 
Conversely, this section may offer another proof of the result of Coble. 
\end{remark}

\subsection{The rationality of $\mathcal{M}_{17,5,1}$}\label{ssec:(17,5)}

Let $V$ be the space of cuspidal plane cubics and 
$U\subset V\times|{\Oplane}(2)|$ be the locus of pairs $(C, Q)$ such that 
$Q$ is the union of distinct lines and transverse to $C+L$, where $L$ is the tangent line of $C$ at the cusp. 
The sextic $C+Q+L$ has an $E_7$-singularity at the cusp of $C$, 
nine nodes at $C\cap Q$, $L\cap Q$, and ${\rm Sing}(Q)$, 
and no other singularity. 
Hence the associated 2-elementary $K3$ surface has invariant $(g, k)=(0, 6)$, 
and we obtain a period map $\mathcal{P}\colon U/{\PGL}_3\to\mathcal{M}_{17,5,1}$.

\begin{proposition}\label{birat (17,5)}
The map $\mathcal{P}$ is birational. 
\end{proposition}

\begin{proof}
Let $\widetilde{U}\subset U\times({\proj}^2)^6$ be the locus of $(C, Q, p_1,\cdots, p_6)$ 
such that $C\cap Q=\{p_i\}_{i=1}^{6}$ and that $p_1, p_2, p_3$ belong to the same component of $Q$. 
By $\widetilde{U}$, the six nodes $C\cap Q$ and the two components of $Q$ are 
labeled compatibly. 
As before, $\mathcal{P}$ lifts to a birational map $\widetilde{U}/{\PGL}_3 \to {\cove}_{17,5,1}$. 
Since ${\PGL}_3$ acts on $U$ almost freely, 
$\widetilde{U}/{\PGL}_3$ is an $\frak{S}_2\ltimes(\frak{S}_3)^2$-cover of $U/{\PGL}_3$. 
On the other hand, ${\cove}_{17,5,1}$ is an ${\Or}(D_{L_-})$-cover of $\mathcal{M}_{17,5,1}$ 
for the lattice $L_-=U(2)^2\oplus A_1$.  
By \cite{M-S} we calculate $|{\Or}(D_{L_-})|=|{\Or}^+(4, 2)|=2\cdot(3!)^2$. 
\end{proof}

\begin{proposition}\label{rational (17,5)}
The quotient $U/{\PGL}_3$ is rational. 
Hence $\mathcal{M}_{17,5,1}$ is rational. 
\end{proposition}

\begin{proof}
Consider the ${\PGL}_3$-equivariant map 
$\pi\colon U \to V\times{\proj}^2$, $(C, Q)\mapsto(C, {\rm Sing}(Q))$. 
A general fiber $\pi^{-1}(C, p)$ is an open set of the net of conics singular at $p$. 
It is straightforward to see that ${\PGL}_3$ acts on $V\times{\proj}^2$ almost freely.  
Then we may apply the no-name lemma \ref{no-name 2} to see that 
\begin{equation*}
U/{\PGL}_3 \sim {\proj}^2\times((V\times{\proj}^2)/{\PGL}_3). 
\end{equation*}
The quotient $(V\times{\proj}^2)/{\PGL}_3$ is of dimension $1$ and so is rational. 
\end{proof}

%
%


\appendix
\begin{large}
\section{}\label{appendix}
\end{large}


In this section we study 
$\mathcal{M}_{12,10,1}$, $\mathcal{M}_{13,9,1}$, $\mathcal{M}_{14,8,1}$, and $\mathcal{M}_{10,10,1}$ 
using the method of Dolgachev-Kond\=o \cite{D-K} and its generalization. 
See Introduction for the details of development. 

Let $L$ be an odd lattice. 
The \textit{even part} $L_0$ of $L$ is the kernel of the natural homomorphism 
\begin{equation*}\label{eqn: even part}
L \to {\Z}/2{\Z}, \qquad l\mapsto (l, l)\mod 2{\Z}, 
\end{equation*}
which is the maximal even sublattice of $L$. 
The index $2$ extension $L_0\subset L$ is given by an element $x$ of order $2$ 
in the discriminant group $D_{L_0}$ of $L_0$: 
that is, we have $L=\langle L_0, x\rangle$ inside $L_0^{\vee}$. 
Since $L$ is odd, $x$ must have norm $1$ modulo $2{\Z}$ with respect to the discriminant form. 

\

\begin{example}\label{I_n&D_n}
Let $I_{m,n}$ denote the odd unimodular lattice $\langle1\rangle^m\oplus\langle-1\rangle^n$. 
When $n\geq4$, the even part of $I_{0,n}$ is isometric to $D_n$. 
Indeed, if we take natural orthogonal basis $e_1,\cdots, e_n$ of $I_{0,n}$, 
the even part is expressed as 
$\{ \sum_{i}x_ie_i \in I_{0,n} \: | \: \sum_ix_i\in2{\Z} \}$. 
It is generated by the following vectors: 
\begin{equation}\label{eqn: D_n root basis}
d_1=e_1+e_2, \quad d_2=-e_1+e_2, \quad d_i=-e_{i-1}+e_i \; \; (3\leq i\leq n). 
\end{equation}
One checks that these form a root system of $D_n$-type for $n\geq4$. 
When $m<n$, $I_{m,n}$ is isometric to $U^m\oplus I_{0,n-m}$. 
Therefore the even part of $I_{m,n}$ is isometric to $U^m\oplus D_{n-m}$ for $n-m\geq4$. 

Let us describe the element $x\in D_{D_n}$ giving the extension $D_n\subset I_{0,n}$. 

\noindent
$(1)$ When $n=2m+1$ is odd, $D_{D_n}$ is isomorphic to ${\Z}/4{\Z}$ generated by 
$\frac{1}{4}(d_1-d_2)+\frac{1}{2}\sum_{i=1}^{m}d_{2i+1}$. 
In particular, $D_{D_n}\simeq\langle-n/4\rangle$ as a finite quadratic form. 
In this case $x$ is the unique element of order $2$ in $D_{D_n}$. 

\noindent
$(2)$ When $n=2m$ is even, $D_{D_n}$ is isomorphic to $({\Z}/2{\Z})^2$ generated by 
$\frac{1}{2}(d_1+\sum_{i=2}^{m}d_{2i})$ and $\frac{1}{2}(d_2+\sum_{i=2}^{m}d_{2i})$. 
As a quadratic form, we have 
$D_{D_n}\simeq\langle-m/2\rangle^{\oplus2}$ 
when $m$ is odd, and  
$D_{D_n}\simeq \begin{pmatrix}  m/2   &   1/2   \\
                                                            1/2    &  m/2      \end{pmatrix}$ 
when $m$ is even. 
In case $m$ is odd or divisible by $4$, $x$ is the unique element with norm $\equiv1$. 
In case $m\equiv2$ mod $4$, $x$ is one of the three non-zero elements of $D_{D_n}$, which all have norm $\equiv1$. 
\end{example}

If a lattice $L$ has signature $(2, n)$, 
we let $\Omega_L$ be the domain defined as in \eqref{eqn:period domain}. 
Given a finite-index subgroup $\Gamma$ of ${\Or}(L)$, we shall denote\footnote{
This is consistent with the notation in \S \ref{sec: 2-ele K3}, but may be slightly different from 
some other literatures where the notation ``$\mathcal{F}_L(\Gamma^+)$'' is used instead.} 
by $\mathcal{F}_L(\Gamma)$ the modular variety $\Gamma^+\backslash\Omega_L^+$ 
where $\Omega_L^+$ is a connected component of $\Omega_L$ 
and $\Gamma^+\subset\Gamma$ is the stabilizer of $\Omega_L^+$. 
We sometimes abbreviate $\mathcal{F}_L(\Gamma)$ as $\mathcal{F}(\Gamma)$. 
The key proposition in this appendix is the following. 

\begin{proposition}[cf. \cite{D-K}]\label{even part}
Let $L_0$ be the even part of an odd lattice $L$ of signature $(2, n)$, 
and $x\in D_{L_0}$ be the element giving the extension $L_0\subset L$. 
Let $\Gamma\subset{\Or}(L_0)$ be the stabilizer of $x$. 
Then the modular varieties $\mathcal{F}_{L_0}(\Gamma)$ and $\mathcal{F}_L({\Or}(L))$ are naturally isomorphic. 
\end{proposition}

\begin{proof}
Since any $\gamma\in{\Or}(L)$ preserves the canonical sublattice $L_0\subset L$, 
we have the inclusion ${\Or}(L)\subset{\Or}(L_0)$ inside ${\Or}(L\otimes{\Q})={\Or}(L_0\otimes{\Q})$. 
If we regard $\gamma\in{\Or}(L)$ as an isometry of $L_0$, 
it stabilizes the overlattice $L\supset L_0$ and hence preserves the subgroup $L/L_0\subset D_{L_0}$. 
Conversely, any $\gamma'\in{\Or}(L_0)$ preserving $L/L_0\subset D_{L_0}$ must be contained in ${\Or}(L)$. 
This establishes the equality ${\Or}(L)=\Gamma$ inside ${\Or}(L\otimes{\Q})$. 
\end{proof}

We learned this technique from Dolgachev-Kond\=o \cite{D-K}, 
who considered the case when $\Gamma$ coincides to ${\Or}(L_0)$ overall. 
It is also used by Allcock \cite{Al}. 

We shall apply this proposition for $L$ a scaling of the anti-invariant lattice or its dual lattice for 
$\mathcal{M}_{12,10,1}$, $\mathcal{M}_{13,9,1}$, $\mathcal{M}_{14,8,1}$, and $\mathcal{M}_{10,10,1}$. 
Consequently, those ${\moduli}$ get realized as a finite cover or a Heegner divisor of another $\mathcal{M}_{r',a',\delta'}$, 
or even become birational to $\mathcal{M}_{r',a',\delta'}$. 
Then we prove the rationality of ${\moduli}$ by making use of 
the geometric description of $\mathcal{M}_{r',a',\delta'}$ obtained in the earlier sections. 
The case of Heegner divisor is similar to \cite{D-K}, while the case of finite cover is novelty of this appendix.

\subsection{The rationality of $\mathcal{M}_{12,10,1}$}\label{ssec:(12,10)}

The following argument was independently suggested by the referee and Kond\=o. 
The anti-invariant lattice for $\mathcal{M}_{12,10,1}$ is isometric to 
the scaling $I_{2,8}(2)$ of the unimodular lattice $I_{2,8}$. 
Hence $\mathcal{M}_{12,10,1}$ is naturally birational to $\mathcal{F}({\Or}(I_{2,8}))$. 
By Example \ref{I_n&D_n}, the even part of $I_{2,8}$ is isometric to $U^2\oplus D_6$, 
and the element in $D_{U^2\oplus D_6}$ giving the extension $U^2\oplus D_6\subset I_{2,8}$ 
is invariant under ${\Or}(U^2\oplus D_6)$. 
Therefore we have $\mathcal{F}({\Or}(I_{2,8}))\simeq\mathcal{F}({\Or}(U^2\oplus D_6))$ by Proposition \ref{even part}. 
Then $U^2\oplus D_6$ is the anti-invariant lattice for $\mathcal{M}_{12,2,1}$, 
so that $\mathcal{F}({\Or}(U^2\oplus D_6))$ is birational to $\mathcal{M}_{12,2,1}$. 
In Proposition \ref{rational (11,3)&(12,2)} we proved that $\mathcal{M}_{12,2,1}$ is rational. 
Hence $\mathcal{M}_{12,10,1}$ is rational too.

\subsection{The rationality of $\mathcal{M}_{13,9,1}$}\label{ssec:(13,9)}

The anti-invariant lattice for $\mathcal{M}_{13,9,1}$ is isometric to $I_{2,7}(2)$. 
Scaling it by $1/2$, we see that $\mathcal{M}_{13,9,1}$ is birational to $\mathcal{F}({\Or}(I_{2,7}))$. 
By Example \ref{I_n&D_n}, the even part of $I_{2,7}$ is isometric to 
\begin{equation*}
M=U^2\oplus D_5, 
\end{equation*}
and the extension $M\subset I_{2,7}$ is given by the unique element of order $2$ in $D_M\simeq{\Z}/4{\Z}$. 
Therefore we have $\mathcal{F}({\Or}(I_{2,7}))\simeq\mathcal{F}({\Or}(M))$ by Proposition \ref{even part}. 

In order to study $\mathcal{F}({\Or}(M))$, 
we shall exploit the fact that the $D_5$-lattice is naturally 
the orthogonal complement of a $(-4)$-vector $l$ in the $D_6$-lattice 
(take $l=d_1+d_2+2\sum_{i=3}^{6}d_i$ where $d_i$ are the root basis as in \eqref{eqn: D_n root basis}). 
Hence $M$ can be embedded in the lattice 
\begin{equation*}
L=U^2\oplus D_6
\end{equation*}
as the orthogonal complement of a $(-4)$-vector. 
Then we want to realize $\mathcal{F}({\Or}(M))$ as a Heegner divisor of $\mathcal{F}({\Or}(L))$. 
For that we need the following lattice-theoretic properties. 

\begin{lemma}\label{structure as Heegner div}
Let $l\in L$ be a $(-4)$-vector with $l/2\in L^{\vee}$. 
Then its orthogonal complement $L'=l^{\perp}\cap L$ is isometric to $M$, 
and every isometry of $L'$ extends to that of $L$. 
Moreover, such $(-4)$-vectors are all equivalent under ${\Or}(L)$. 
\end{lemma}
 
\begin{proof}
We first note that ${\Z}l\subset L$ is primitive. 
Let $G\subset D_{{\Z}l}$ (resp. $G'\subset D_{L'}$) be the subgroup corresponding to 
the image of the orthogonal projection $L\to({\Z}l)^{\vee}$ (resp. $L\to(L')^{\vee}$). 
Then we have canonical isomorphisms 
$G\simeq L/({\Z}l\oplus L') \simeq G'$. 
In other words, we have an anti-isometry $\lambda\colon G\to G'$ 
whose graph $G''\subset D_{{\Z}l}\oplus D_{L'}$ gives the overlattice $L\supset{\Z}l\oplus L'$ (cf. \cite{Ni1}). 
The discriminant form of $L$ is identified with the quadratic form on $(G'')^{\perp}/G''$. 

\begin{claim}
We have $G=\langle l/2 \rangle$, $D_{L'} = \langle x \rangle\simeq{\Z}/4{\Z}$ with $(x, x)\equiv3/4$, 
and $\lambda$ is the unique isomorphism between $G$ and the index $2$ subgroup of $D_{L'}$. 
\end{claim}

\begin{proof}
We have $(l/2, G)=(l/2, G'')=0$ by the assumption $l/2\in L^{\vee}$. 
Hence $G\subset\langle l/2 \rangle$. 
If $G=0$, then we would have $L={\Z}l\oplus L'$, which contradicts the 2-elementariness of $D_L$. 
Hence $G=\langle l/2 \rangle$. 
Since the discriminant form of $L'$ is non-degenerate, 
we may find an element $x\in D_{L'}$ with $(x, \lambda(l/2))=1/2$. 
Then $x+\frac{l}{4}$ is perpendicular to $G''$. 
Since $(G'')^{\perp}/G''$ is 2-elementary, we have $2x+\frac{l}{2}\in G''$, which means that $2x=\lambda(l/2)$. 
Therefore 
\begin{equation*}
4(x, x) \equiv -(\frac{l}{2}, \frac{l}{2}) \equiv 1  \mod 2{\Z},
\end{equation*}
so that $(x, x)\equiv k/4$ for some $k\in\{1, 3, 5, 7\}$. 
In particular, the quadratic form on $\langle x \rangle \simeq {\Z}/4{\Z}$ is non-degenerate. 
Hence we have the orthogonal splitting $D_{L'}=\langle x \rangle \oplus H$ for a subgroup $H\subset D_{L'}$. 
Since $D_L \simeq ({\Z}/2{\Z})^2$, we actually have $H=0$. 
Now $D_L=(G'')^{\perp}/G''$ is represented by the following elements: 
\begin{equation*}
0, \quad \frac{l}{2}, \quad x+\frac{l}{4}, \quad x-\frac{l}{4}. 
\end{equation*}
In view of the structure of $D_L\simeq D_{D_6}$ (Example \ref{I_n&D_n}), 
the elements $x\pm\frac{l}{4}$ should have norm $1/2$. 
Therefore $k$ is determined as $k=3$.  
\end{proof}

By this claim the discriminant form of $L'$ is uniquely determined, 
so that we may apply Nikulin's theory \cite{Ni1} to see that the isometry class of $L'$ is uniquely determined. 
Looking the discriminant form of $M$ at Example \ref{I_n&D_n}, we see that $L'\simeq M$. 
The claim also says that the extension ${\Z}l\oplus L'\subset L$ is uniquely determined inside $({\Z}l\oplus L')^{\vee}$. 
It follows that for any isometry $\gamma$ of $L'$, 
the isometry $({\rm id}, \gamma)$ of ${\Z}l\oplus L'$ preserves the overlattice $L$. 
Finally, the uniqueness of $L'$ and of the extension ${\Z}l\oplus L'\subset L$ imply our uniqueness assertion for $l$. 
\end{proof}

Let 
\begin{equation*}
\mathcal{H} \subset\mathcal{F}({\Or}(L))
\end{equation*}
be the Heegner divisor defined by $(-4)$-vectors as in Lemma \ref{structure as Heegner div}, 
that is, the image of the analytic divisor $\sum l^{\perp}\subset\Omega_L^+$ 
where $l$ range over $(-4)$-vectors in $L$ with $l/2\in L^{\vee}$. 
Noticing that ${\Or}(M)$ contains an element exchanging the two components of $\Omega_M$, 
we deduce from Lemma \ref{structure as Heegner div} that 
$\mathcal{H}$ is irreducible and birational to $\mathcal{F}({\Or}(M))$: 
more precisely, we have a natural projection $\mathcal{F}({\Or}(M))\to \mathcal{H}$ 
which gives the normalization of $\mathcal{H}$. 

Now $L$ is the anti-invariant lattice for $\mathcal{M}_{12,2,1}$, 
so that $\mathcal{F}({\Or}(L))$ is birational to $\mathcal{M}_{12,2,1}$. 
We shall study the Heegner divisor $\mathcal{H}\subset\mathcal{M}_{12,2,1}$ 
through the geometric description of $\mathcal{M}_{12,2,1}$ given in \S \ref{ssec:(11,3)&(12,2)}. 
This requires to refine the period map there, especially to extend it. 

Let $Q={\proj}^1\times{\proj}^1$ and $U\subset|{\OQ}(3, 3)|$ be the locus of smooth curves $C$ 
such that there exists exactly one $(1, 0)$-fiber $F_1$ tangent to $C$ with multiplicity $3$ 
(compare $U$ with the locus $U_5$ in \S \ref{ssec:(11,3)&(12,2)}). 
Let $F_2$ be the $(0, 1)$-fiber through $C\cap F_1$. 
Taking the right resolution of $C+F_1+F_2$, we obtain a period map 
$p\colon U\to\mathcal{M}_{12,2,1}$ as a morphism of varieties. 
In Proposition \ref{birational (11,3)&(12,2)}, we showed that $p$ descends to a birational map 
from a rational quotient of $U$ by ${\aut}(Q)_0={\PGL}_2\times{\PGL}_2$. 
We refine it as follows. 

\begin{proposition}\label{precise period map (13,9)}
The map $p$ descends to an open immersion 
$\mathcal{P}\colon U/{\aut}(Q)_0\to\mathcal{M}_{12,2,1}$ 
from a geometric quotient $U/{\aut}(Q)_0$ of $U$. 
\end{proposition}

\begin{proof}
We first check the existence of a geometric quotient $U/{\aut}(Q)_0$. 
Let $|{\OQ}(3, 3)|_{sm}$ be the affine open locus in $|{\OQ}(3, 3)|$ of smooth curves, 
of which $U$ is a locally closed subset. 
We have a geometric quotient of $|{\OQ}(3, 3)|_{sm}$ by ${\aut}(Q)$ as an open set of $\mathcal{M}_4$. 
By \cite{GIT} Converse 1.12, this assures the existence of an ${\aut}(Q)_0$-linearized line bundle $\mathcal{L}$ 
over $|{\OQ}(3, 3)|_{sm}$ with respect to which every point of $|{\OQ}(3, 3)|_{sm}$ is stable. 
Now we can restrict $\mathcal{L}$ to $U$ and apply \cite{GIT} Theorem 1.10. 

Next we show that the induced morphism $\mathcal{P}$ from the geometric quotient has finite fibers. 
Indeed, if $(X, \iota)=\mathcal{P}(C)$ for $C\in U$, 
then $C$ can be recovered from the natural projection $f\colon X\to Q$, 
which in turn from the line bundle $f^{\ast}{\OQ}(1, 1)$ by Lemma \ref{covering map & LB}. 
Thus we have a surjective map onto $\mathcal{P}^{-1}(X, \iota)$ from a subset of $L_+(X, \iota)$, 
so that $\mathcal{P}^{-1}(X, \iota)$ is countable. 
Then we can apply Zariski's Main Theorem to conclude that $\mathcal{P}$ is an open immersion. 
\end{proof}

Let $V\subset U$ be the codimension $1$ locus where the $(0, 1)$-fibers $F_2$ are tangent to $C$ 
(outside $C\cap F_1$). 

\begin{proposition}\label{describe Heegner div (13,9)}
The divisor $V/{\aut}(Q)_0$ of $U/{\aut}(Q)_0$ is mapped by $\mathcal{P}$ 
to an open set of the Heegner divisor $\mathcal{H}$. 
\end{proposition}

\begin{proof}
Since $\mathcal{H}$ is irreducible, it suffices to see that $p(V)\subset \mathcal{H}$. 
For $C\in V$, the tangent point $q$ of $C$ and $F_2$ is a tacnode of $C+F_1+F_2$. 
If $(X, \iota)=\mathcal{P}(C)$ and $f\colon X\to Q$ is the natural projection, 
then $f^{-1}(q)$ is an $A_3$-configuration of $(-2)$-curves. 
We write $f^{-1}(q)=E_0+E_++E_-$ such that $(E_0, E_{\pm})=1$ and $(E_+, E_-)=0$. 
By Figure \ref{dual graph}, $\iota$ acts on $f^{-1}(q)$ by $\iota(E_0)=E_0$ and $\iota(E_{\pm})=E_{\mp}$. 
Then we have the $\iota$-anti-invariant cycle 
\begin{equation*}
D = E_+ - E_- \in L_-(X, \iota)
\end{equation*}
as an algebraic cycle on $X$. 
This means that the period of $(X, \iota)$ is in the Heegner divisor defined by $D$. 
Since $(D, D)=-4$, it remains to check that $D/2\in L_-(X, \iota)^{\vee}$. 
This follows from 
\begin{equation*}
(D, L_-(X, \iota)) = (D+E_++E_-, L_-(X, \iota)) = 2(E_+, L_-(X, \iota)) \subset 2{\Z}, 
\end{equation*}
where the first equality holds because $E_++E_-\in L_+(X, \iota)$. 
\end{proof}

Summing up the results so far, we obtain the birational equivalences 
\begin{equation*}
\mathcal{M}_{13,9,1} \sim \mathcal{F}({\Or}(M)) \sim \mathcal{H} \sim V/{\aut}(Q)_0. 
\end{equation*}

\begin{proposition}\label{ratl (13,9)}
The quotient $V/{\aut}(Q)_0$ is rational. 
Hence $\mathcal{M}_{13,9,1}$ is rational. 
\end{proposition}

\begin{proof}
Let $W\subset Q\times Q$ be the locus of pairs $(p, q)$ such that $p$ and $q$ lie on a same $(0, 1)$-fiber. 
We have the ${\aut}(Q)_0$-equivariant map 
\begin{equation*}
\varphi : V\to W, \qquad C\mapsto(C\cap F_1, C\cap F_2\backslash F_1), 
\end{equation*}
where $F_1$ and $F_2$ are the ruling fibers defined from $C$ as above. 
The group $({\SL}_2)^2$ acts on $W$ almost transitively, 
with the stabilizer of a general point $(p, q)$ isomorphic to ${\C}^{\times}\times({\C}^{\times}\ltimes{\C})$. 
The fiber $\varphi^{-1}(p, q)$ is an open set of a sub-linear system of $|{\OQ}(3, 3)|$. 
Therefore $V/{\aut}(Q)_0$ is rational by the slice method for $\varphi$ and Miyata's theorem. 
\end{proof}

\subsection{The rationality of $\mathcal{M}_{14,8,1}$}\label{ssec:(14,8)}

As before, $\mathcal{M}_{14,8,1}$ is birational to $\mathcal{F}({\Or}(I_{2,6}))$. 
By Example \ref{I_n&D_n}, the even part $L$ of $I_{2,6}$ is isometric to $U^2\oplus D_4$, 
with the extension $L\subset I_{2,6}$ given by one of the three non-zero elements of $D_L\simeq D_{D_4}$, say $x$. 
Let $\Gamma\subset{\Or}(L)$ be the stabilizer of $x$. 
Then we have 
\begin{equation*}
\mathcal{F}({\Or}(I_{2,6})) \simeq \mathcal{F}_L(\Gamma) 
\end{equation*}
by Proposition \ref{even part}. 

Note that $L$ is the anti-invariant lattice for $\mathcal{M}_{14,2,0}$. 
Since $\widetilde{{\Or}}(L)\subset\Gamma\subset{\Or}(L)$, 
the modular variety $\mathcal{F}_L(\Gamma)$ is realized as an intermediate cover 
\begin{equation}\label{intermi cover (14,8)}
{\cove}_{14,2,0} \to \mathcal{F}_L(\Gamma) \to \mathcal{M}_{14,2,0}. 
\end{equation}
The Galois group of ${\cove}_{14,2,0}\to\mathcal{M}_{14,2,0}$ is ${\Or}(D_L)$, 
which is isomorphic to $\frak{S}_3$ 
by the permutation action on the three non-zero elements of $D_L$. 
If $G\subset\frak{S}_3$ is the stabilizer of any one, then we have 
\begin{equation*}\label{eqn: birat (14,2,0)/G}
\mathcal{F}_L(\Gamma) \simeq {\cove}_{14,2,0}/G. 
\end{equation*}

In Example \ref{ex:4}, we gave a geometric description of ${\cove}_{14,2,0}\to\mathcal{M}_{14,2,0}$ 
in terms of plane quartics. 
Below let us reuse the notation there. 
We considered a space $U$ parametrizing smooth plane quartics $C$ 
with a tangent line $L_1$ of multiplicity $4$ and with a line $L_2$ through $C\cap L_1$. 
We also considered an $\frak{S}_3$-covering $\widetilde{U}\to U$ 
whose fibers consist of the numberings of the three nodes $C\cap L_2\backslash L_1$. 
Then we obtained a birational lift ${\lift}\colon\widetilde{U}/{\PGL}_3\to{\cove}_{14,2,0}$ of 
the period map $U/{\PGL}_3\to\mathcal{M}_{14,2,0}$. 

\begin{lemma}\label{equivariance (14,2,0)}
The map ${\lift}$ is $\frak{S}_3$-equivariant. 
Here $\frak{S}_3$ acts on $\widetilde{U}$ by renumbering of the three nodes, 
and on ${\cove}_{14,2,0}$ by permutation of the non-zero elements of $D_L$. 
\end{lemma}

\begin{proof}
This follows simply because any automorphism of $\frak{S}_3$ is inner, 
but here we shall provide geometric reasoning. 
For a point $(C, L_1,\cdots, p_3)$ of $\widetilde{U}$, 
let $(X, \iota)$ be the associated 2-elementary $K3$ surface with the lattice-marking $j\colon L_+\to L_+(X, \iota)$. 
If $f\colon X\to{\proj}^2$ is the natural projection, then $j(e_i)$ is the class of the $(-2)$-curve $f^{-1}(p_i)$. 
By Lemma \ref{L_+ and A-D-E}, the vectors 
\begin{equation*}
y_1 = \frac{1}{2}(e_2+e_3), \quad y_2 = \frac{1}{2}(e_3+e_1), \quad y_3 = \frac{1}{2}(e_1+e_2)
\end{equation*}
have integral pairing with $L_+$ and hence are contained in $L_+^{\vee}$. 
They represent the three non-zero elements of $D_{L_+}$. 
If $j_{\sigma}\colon L_+\to L_+(X, \iota)$ is the marking associated to the labelling 
$(p_{\sigma^{-1}(1)}, p_{\sigma^{-1}(2)}, p_{\sigma^{-1}(3)})$ for $\sigma\in\frak{S}_3$, 
then $j_{\sigma}^{-1}\circ j$ maps $y_i$ to $y_{\sigma(i)}$. 
Now choose isometries $\gamma, \gamma_{\sigma}\colon L\to L_-(X, \iota)$ such that 
$j\oplus\gamma$ and $j_{\sigma}\oplus\gamma_{\sigma}$ extend to $\Lambda_{K3}\to{\cohomology}$ respectively. 
If we denote by $x_i\in D_L$ the image of $y_i$ by the natural anti-isometry $D_{L_+}\to D_L$, 
then $\gamma_{\sigma}^{-1}\circ \gamma$ maps $x_i$ to $x_{\sigma(i)}$ by Nikulin \cite{Ni1}. 
In view of the definition of ${\lift}$, this proves our assertion. 
\end{proof}

By this lemma, we obtain 
\begin{equation*}
\mathcal{F}_L(\Gamma) \sim (\widetilde{U}/{\PGL}_3)/G. 
\end{equation*}
Let $U'\subset U\times{\proj}^2$ be the locus of those $(C, L_1, L_2, p)$ with $p\in C\cap L_2\backslash L_1$. 
Since $\widetilde{U}/G$ is naturally isomorphic to $U'$, we then have 
\begin{equation*}
\mathcal{F}_L(\Gamma) \sim U'/{\PGL}_3. 
\end{equation*}

\begin{proposition}
The quotient $U'/{\PGL}_3$ is rational. 
Hence $\mathcal{M}_{14,8,1}$ is rational. 
\end{proposition}

\begin{proof}
Let $V\subset({\proj}^2)^2\times|{\Oplane}(1)|$ be the locus of triplets $(p, q, L)$ such that $q\in L$. 
We consider the ${\PGL}_3$-equivariant map 
\begin{equation*}
\varphi : U'\to V, \qquad (C, L_1, L_2, p)\mapsto(p, L_1\cap L_2, L_1). 
\end{equation*}
Since the line $L_2$ is recovered as $\overline{pq}$, 
the $\varphi$-fiber over a general $(p, q, L)$ is identified with an open set of a sub-linear system of $|{\Oplane}(4)|$. 
The group ${\SL}_3$ acts on $V$ almost transitively with connected and solvable stabilizer. 
Hence $U'/{\PGL}_3$ is rational by the slice method for $\varphi$ and Miyata's theorem. 
\end{proof}

\begin{remark}
The lines of argument for $\mathcal{M}_{12,10,1}$, $\mathcal{M}_{13,9,1}$, and $\mathcal{M}_{14,8,1}$ 
would apply more generally to $\mathcal{F}({\Or}(I_{2,n}))$ until $n\leq18$, 
implying that those modular varieties are rational for $n\leq16$ and unirational for $n=17, 18$: 
\begin{itemize}
\item when $n$ is odd, $\mathcal{F}({\Or}(I_{2,n}))$ is realized as 
a $(-4)$-Heegner divisor of $\mathcal{M}_{19-n,2,\delta}$, where $\delta\equiv(n-1)/2$. 
For $n<17$ we then study a locus of $-2K_Y$-curves that acquire tacnode as extra point of tangency 
(and finally use Miyata's theorem), while $\mathcal{F}({\Or}(I_{2,17}))$ gets birational to $|L_{4,0}|/{\aut}({\F}_2)$; 
\item when $n=14$, $\mathcal{F}({\Or}(I_{2,14}))$ is realized as a triple cover of $\mathcal{M}_{6,2,0}$; 
\item for other even $n$, $\mathcal{F}({\Or}(I_{2,n}))$ becomes birational to $\mathcal{M}_{20-n,2,\delta}$, 
$\delta\equiv(n-2)/2$. 
\end{itemize}
The case $n=9$ is studied by Dolgachev-Kond\=o \cite{D-K}, 
and the case $n=10$ is attributed to Kond\=o \cite{Ko} and Allcock \cite{Al}. 
\end{remark}

\subsection{The rationality of $\mathcal{M}_{10,10,1}$}\label{ssec:(10,10,1)}

The anti-invariant lattice $L_-$ for $\mathcal{M}_{10,10,1}$ is isometric to $U\oplus I_{1,1}(2) \oplus E_8(2)$. 
We consider the odd lattice 
$L_-^{\vee}(2) = U(2) \oplus I_{1,1} \oplus E_8$,  
for which we have 
\begin{equation*}
\mathcal{M}_{10,10,1} \sim \mathcal{F}({\Or}(L_-)) \simeq \mathcal{F}({\Or}(L_-^{\vee}(2))). 
\end{equation*}
The even part of $L_-^{\vee}(2)$ is  
\begin{equation*}
L = U(2) \oplus U(2) \oplus E_8, 
\end{equation*} 
with the extension  $L\subset L_-^{\vee}(2)$ given by an element $x\in D_L$ of norm $1$. 
By Proposition \ref{even part}, we have 
\begin{equation*}
\mathcal{M}_{10,10,1} \sim \mathcal{F}_L(\Gamma)  
\end{equation*}
for the stabilizer $\Gamma\subset{\Or}(L)$ of $x$.  

Since $L$ is the anti-invariant lattice for $\mathcal{M}_{10,4,0}$, 
we can realize $\mathcal{F}_L(\Gamma)$ as an intermediate cover 
\begin{equation*}
{\cove}_{10,4,0} \to \mathcal{F}_L(\Gamma) \to \mathcal{M}_{10,4,0} 
\end{equation*}
as in \eqref{intermi cover (14,8)}. 
Let us describe the Galois group ${\Or}(D_L)$ of ${\cove}_{10,4,0} \to \mathcal{M}_{10,4,0}$. 
Let $\{ u_+, v_+\}$ and $\{ u_-, v_-\}$ be natural basis of the two summands $U(2)\oplus U(2)$ of $L$ with 
$(u_{\pm}, u_{\pm})=(v_{\pm}, v_{\pm})=0$ and $(u_{\pm}, v_{\pm})=2$. 
We define norm $1$ elements of $D_L$ by 
\begin{equation*}
x_{1\pm} = \frac{1}{2}(u_{\mp}+v_{\mp}+u_{\pm}), \quad 
x_{2\pm} = \frac{1}{2}(u_{\mp}+v_{\mp}+v_{\pm}), \quad 
x_{3\pm} = \frac{1}{2}(u_{\pm}+v_{\pm}). 
\end{equation*} 
Then we have the orthogonal decomposition 
\begin{equation}\label{eqn: norm 1 basis (10,10)}
D_L = \langle x_{1+}, x_{2+} \rangle \oplus \langle x_{1-}, x_{2-} \rangle 
\simeq ({\Z}/2{\Z})^2\oplus({\Z}/2{\Z})^2 
\end{equation}
with $x_{3\pm}=x_{1\pm}+x_{2\pm}$. 
It follows that the six elements $\{ x_{i\pm} \}_{i=1}^{3}$ are all of the norm $1$ elements in $D_L$, 
and the decomposition \eqref{eqn: norm 1 basis (10,10)} is canonical. 
Therefore 

\begin{lemma}\label{O(D_L) (10,10)}
The group ${\Or}(D_L)$ is identified with the group of permutations of the six elements $\{ x_{i\pm} \}_{i=1}^{3}$ 
preserving the (unordered) partition $\{ x_{i+} \}_{i=1}^{3} \cup \{ x_{i-} \}_{i=1}^{3}$. 
In particular, we have ${\Or}(D_L)\simeq \frak{S}_2\ltimes(\frak{S}_3)^2$. 
\end{lemma}

Since the six elements $\{ x_{i\pm} \}_{i=1}^{3}$ are all $\frak{S}_2\ltimes(\frak{S}_3)^2$-equivalent, 
if $G\subset\frak{S}_2\ltimes(\frak{S}_3)^2$ is the stabilizer of any one, 
then we have 
\begin{equation}\label{eqn: birat map to (10,4,0)}
\mathcal{F}_L(\Gamma) \simeq {\cove}_{10,4,0}/G. 
\end{equation}

In Example \ref{ex:3}, we studied the covering ${\cove}_{10,4,0} \to \mathcal{M}_{10,4,0}$ using curves on ${\F}_2$. 
Below we shall reuse the notation there. 
We considered an open locus $U\subset|L_{3,0}|\times|L_{0,2}|$, and 
an $\frak{S}_2\ltimes(\frak{S}_3)^2$-covering $\widetilde{U}\to U$ whose fiber over $(C, D)\in U$ consists of 
labellings of the six points $C\cap D$ that are compatible with the decomposition of $D$. 
Then we obtained a birational lift ${\lift}\colon\widetilde{U}/{\aut}({\F}_2)\to{\cove}_{10,4,0}$ of 
the period map $U/{\aut}({\F}_2)\to\mathcal{M}_{10,4,0}$. 

\begin{lemma}\label{equivariance (10,4,0)}
The map ${\lift}$ is $\frak{S}_2\ltimes(\frak{S}_3)^2$-equivariant. 
Here $\frak{S}_2\ltimes(\frak{S}_3)^2$ acts on $\widetilde{U}$ by relabelling of the six points, 
and on ${\cove}_{10,4,0}$ by permutation of the norm $1$ elements of $D_L$. 
\end{lemma}

\begin{proof}
This is similar to the proof of Lemma \ref{equivariance (14,2,0)}. 
For a point $(C, D, p_{1+},\cdots, p_{3-})$ of $\widetilde{U}$, 
let $(X, \iota)$ be the associated 2-elementary $K3$ surface with the lattice-marking $j\colon L_+\to L_+(X, \iota)$. 
If $f\colon X\to{\F}_2$ is the natural projection, the class of the $(-2)$-curve $f^{-1}(p_{i\pm})$ is given by $j(e_{i\pm})$. 
One checks that the six vectors 
\begin{equation*}
y_{1\pm} = \frac{1}{2}(e_{2\pm}+e_{3\pm}), \quad y_{2\pm} = \frac{1}{2}(e_{3\pm}+e_{1\pm}), 
\quad y_{3\pm} = \frac{1}{2}(e_{1\pm}+e_{2\pm})
\end{equation*}
are contained in $L_+^{\vee}$. 
They represent the six norm $1$ elements of $D_{L_+}$ with $(y_{i\pm}, y_{j\mp})\equiv0$. 
Choosing an isometry $L\simeq(L_+)^{\perp}\cap\Lambda_{K3}$ suitably in advance, 
we may assume that the image of $y_{i\pm}$ by the natural anti-isometry $D_{L_+}\to D_L$ is $x_{i\pm}$. 
If we transform the labelling $(p_{1+},\cdots, p_{3-})$ by an element $\sigma$ of $\frak{S}_2\ltimes(\frak{S}_3)^2$, 
then the new marking $j_{\sigma}\colon L_+\to L_+(X, \iota)$ satisfies 
$j_{\sigma}^{-1}\circ j(y_{i\pm})=y_{\sigma(i\pm)}$. 
As in the proof of Lemma \ref{equivariance (14,2,0)}, 
this shows the $\frak{S}_2\ltimes(\frak{S}_3)^2$-equivariance of ${\lift}$. 
\end{proof}

By \eqref{eqn: birat map to (10,4,0)} and Lemma \ref{equivariance (10,4,0)}, 
$\mathcal{F}_L(\Gamma)$ is birational to $(\widetilde{U}/{\aut}({\F}_2))/G$. 
If $U'\subset U\times{\F}_2$ is the locus of those $(C, D, p)$ with $p\in C\cap D$, then we have 
\begin{equation*}
(\widetilde{U}/{\aut}({\F}_2))/G \sim U'/{\aut}({\F}_2). 
\end{equation*}

\begin{proposition}\label{rational (10,10,1)}
The quotient $U'/{\aut}({\F}_2)$ is rational. 
Thus $\mathcal{M}_{10,10,1}$ is rational. 
\end{proposition}

\begin{proof}
For $(C, D, p)\in U'$, let $F$ be the component of $D$ \textit{not} through $p$. 
We have the ${\aut}({\F}_2)$-equivariant map 
\begin{equation*}
\varphi : U' \to {\F}_2\times|L_{0,1}|, \qquad (C, D, p)\mapsto(p, F). 
\end{equation*}
The $\varphi$-fiber over a general $(p, F)$ is identified with an open set of a hyperplane ${\proj}V$ of $|L_{3,0}|$. 
Since ${\aut}({\F}_2)$ acts on ${\F}_2\times|L_{0,1}|$ almost transitively, 
by the slice method we have 
\begin{equation*}
U'/{\aut}({\F}_2) \sim {\proj}V/H
\end{equation*}
for the stabilizer $H\subset{\aut}({\F}_2)$ of $(p, F)$. 
By \eqref{Aut(Hir)} and \eqref{desrcibe R}, $H$ is an extension of ${\C}^{\times}$ by ${\C}^{\times}\ltimes{\C}^2$ 
and hence connected and solvable. 
Since $L_{3,0}$ is ${\aut}({\F}_2)$-linearized, $H$ acts on $V$ linearly. 
Thus we may apply Miyata's theorem to see that ${\proj}V/H$ is rational. 
\end{proof}


\vspace{0.3cm}
\noindent
\textbf{Acknowledgements.}
I would like to express my gratitude to Professor Ken-Ichi Yoshikawa for introducing me to this rich subject. 
I am also indebted to Professors I. Dolgachev and S. Kond\=o, 
whose paper \cite{D-K} inspired me to resume the study after a period of abandonment. 
I received invaluable suggestions from the referee and Prof.~Kond\=o, 
especially the proof for $\mathcal{M}_{12,10,1}$ (p.\pageref{ssec:(12,10)}, from both) 
and the second construction of $\mathcal{M}_{12,8,1}\dashrightarrow\mathcal{M}_{12,4,1}$ 
(p.\pageref{2nd map}, from the referee). 
I should like to thank them for their generous permission to include those in this article. 
Finally, I am grateful to the editors for their help and patience.


\end{document}